\documentclass[10pt,twoside, a4paper,reqno]{amsart}

\usepackage{amsmath,amssymb,eucal,graphicx,mathtools}
\usepackage[foot]{amsaddr}
\usepackage{amsthm}
\usepackage{hyperref}
\usepackage{graphicx}
\usepackage{caption}
\usepackage{subcaption}
\usepackage{listings}
\usepackage{tikz}
\usetikzlibrary{shapes.geometric, positioning,calc, patterns,decorations.markings}
\usetikzlibrary{arrows}
\usetikzlibrary{decorations.pathmorphing}
\usepackage{tikz-cd}
\usepackage{stmaryrd}
\usepackage[capitalise]{cleveref}
\usepackage{comment}
\usepackage{verbatim}
\usepackage{tkz-euclide}
\usepackage{adjustbox}
\usepackage{mathdots}
\usepackage{circuitikz}
\usetikzlibrary{decorations.markings}
\usepackage[noadjust]{cite}
\usepackage{todonotes}
\usepackage{mathrsfs}
\usepackage{mathtools}
\usepackage{enumitem}
\usepackage[makeroom]{cancel}
\usepackage[margin=2.5cm]{geometry}

\definecolor{friendly_deepblue}{RGB}{51,34,136}
\definecolor{friendly_blue}{RGB}{136,204,238}
\definecolor{friendly_orange}{RGB}{221,104,19}
\definecolor{friendly_red}{RGB}{136,34,85}
\definecolor{friendly_green}{RGB}{68,170,153}

\mathchardef\mhyphen="2D 
\DeclareMathOperator{\mods}{mod}

\newcommand{\stautiltpair}{\mathrm{s\tau\mhyphen tilt{\,}pair}}
\newcommand{\taurigidpair}{\mathrm{\tau\mhyphen rigid{\,}pair}}

\DeclareMathOperator{\silt}{silt}

\DeclareMathOperator{\presilt}{presilt}

\DeclareMathOperator{\Hom}{Hom}

\DeclareMathOperator{\End}{End}

\DeclareMathOperator{\add}{add}
\DeclareMathOperator{\thick}{thick}
\DeclareMathOperator{\thicksubc}{thk}
\DeclareMathOperator{\fullsubc}{full}

\DeclareMathOperator{\proj}{proj}

\DeclareMathOperator{\wide}{wide}

\DeclareMathOperator{\spann}{span}
\DeclareMathOperator{\ind}{ind}
\DeclareMathOperator{\coind}{coind}

\newcommand{\op}{^{\mathrm{op}}}

\DeclareMathOperator{\coker}{coker}

\DeclareMathOperator{\Mod}{Mod}

\DeclareMathOperator{\starr}{star}

\DeclareMathOperator{\image}{image}

\newcommand{\Kb}{\mathcal{K}^{b}}
\newcommand{\Db}{\mathcal{D}^{b}}
\newcommand{\Ktwo}{\mathcal{K}^{[-1,0]}}
\newcommand{\per}{\mathrm{per}}
\newcommand{\pertwo}{\mathrm{per}^{[-1,0]}}
\newcommand{\sth}{\;\vert\;}
\newcommand{\eqclass}[1]{ \llbracket #1 \rrbracket }
\newcommand{\partfan}[1]{\mathfrak{F}_{\bfg}(#1)}
\newcommand{\catof}{\mathfrak{C}}

\newcommand{\sqbinom}[2]{\left\llbracket \!\begin{smallmatrix} #1 \\ #2 \end{smallmatrix}\!\right\rrbracket}

\newcommand{\tcmc}[1]{\mathfrak{W}(#1)}

\newcommand{\derotimes}[1]{\otimes_{#1}^{\mathbf{L}}}

\makeatletter
\newtheorem*{rep@theorem}{\rep@title}
\newcommand{\newreptheorem}[2]{%
	\newenvironment{rep#1}[1]{%
		\def\rep@title{#2 \ref{##1}}%
		\begin{rep@theorem}}%
		{\end{rep@theorem}}}
\makeatother

\newreptheorem{theorem}{Theorem}
\newreptheorem{proposition}{Proposition}
\newreptheorem{lemma}{Lemma}
\newreptheorem{corollary}{Corollary}
\newreptheorem{setup}{Setting}

\newcommand{\A}{\mathscr{A}}

\newcommand{\calI}{\mathcal{I}}
\newcommand{\calN}{\mathcal{N}}

\newcommand{\Wfrak}{\mathfrak{W}}
\newcommand{\Sfrak}{\mathfrak{S}}
\newcommand{\Pfrak}{\mathfrak{P}}
\newcommand{\R}{\mathbb{R}}

\newcommand{\calB}{\mathcal{B}}
\newcommand{\calC}{\mathcal{C}}
\newcommand{\calD}{\mathcal{D}}
\newcommand{\calE}{\mathcal{E}}
\newcommand{\calF}{\mathcal{F}}
\newcommand{\calT}{\mathcal{T}}
\newcommand{\calW}{\mathcal{W}}

\newcommand{\calZ}{\mathcal{Z}}
\newcommand{\E}{\mathbf{E}}
\renewcommand{\k}{\mathbf{k}}
\newcommand{\boldC}{\mathbf{C}}
\newcommand{\K}{\mathrm{K}}
\newcommand{\Ksp}{\mathrm{K}^{\mathrm{sp}}}
\newcommand{\bfg}{\mathbf{g}}
\newcommand{\Sfan}{\mathfrak{S}}
\newcommand{\Z}{\mathbb{Z}}

\counterwithin*{equation}{section}

\newtheorem{theorem}{Theorem}[section]
\newtheorem{corollary}[theorem]{Corollary}
\newtheorem{lemma}[theorem]{Lemma}
\newtheorem{proposition}[theorem]{Proposition}

\theoremstyle{definition}
\newtheorem{definitiontheorem}[theorem]{Definition-Theorem}
\newtheorem{definitionproposition}[theorem]{Definition-Proposition}
\newtheorem{definition}[theorem]{Definition}
\newtheorem{setup}[theorem]{Setting}
\newtheorem{example}[theorem]{Example}
\newtheorem{remark}[theorem]{Remark}

\title[$\mathbf{g}$-vector fans and picture categories for 0-Auslander extriangulated categories]{$\mathbf{g}$-vector fans and picture categories for 0-Auslander extriangulated categories}
\author[E. D. Børve]{Erlend D. Børve}
\address[E.D.B]{Department of Mathematics, Aarhus University, Ny Munkegade 118, 8000 Aarhus C, Denmark.}
\email{erlend.d.borve@math.au.dk}
\thanks{E.D.B. gratefully acknowledges support from the French ANR grant CHARMS (ANR-19-CE40-0017-02), the Deutsche Forschunggemeinschaft (DFG, German Research Foundation) -- Project ID 281071066 -- TRR 191, from the NAWI Graz Postdoc fellowship, and from the Aarhus University Research Foundation -- Grant number AUFF-E-2024-9-43.}

\author[M. Kaipel]{Maximilian Kaipel}
\address[M.K]{Fakultät für Mathematik, Universität Bielefeld, 33501 Bielefeld, Germany}
\email{mkaipel@math.uni-bielefeld.de}
\thanks{M.K. acknowledges support from the Deutsche Forschungsgemeinschaft (DFG, German Research Foundation) as part of SFB-TRR 191 (Project ID 281071066) and as part of SFB-TRR 358/1 2023 (Project ID 491392403).}


\begin{document}

\begin{abstract}
    We extend the notion of $\bfg$-vector fan so that it is defined for a Hom-finite Krull--Schmidt 0-Auslander $k$-linear extriangulated category $\calC$ with a {projective} silting object {$T$}. Moreover, we show that the $\bfg$-vector fan admits an admissible partition, in the sense of the second-named author, which is induced by thick subcategories. One can thus define the picture category of $\calC$. We establish a bijection between thick subcategories of $\calC$ generated by presilting objects containing all projective-injective objects and $\tau$-perpendicular subcategories of the endomorphism $k$-algebra of $T$. This shows that our construction unifies all previous constructions of picture categories and $\tau$-cluster morphism categories of finite-dimensional algebras. We introduce morphisms of partitioned fans to {provide a common framework for} the functorial relationships between picture categories of different algebras and categories.
\end{abstract}

\maketitle

\section{Introduction}

Cluster algebras \cite{FZ02i,FZ02ii,BFZ05,FZ07} and cluster categories \cite{BMRRT06} have had a revolutionary impact on tilting theory. It is not always possible to replace an indecomposable direct summand of a tilting module to obtain a new tilting module, but if possible, the new indecomposable direct summand is unique up to isomorphism \cite{BB80,HR82}. Nowadays, this process is called \textit{mutation}, so that the terminology is consistent with the generalisations of tilting theory inspired by cluster algebras. For example, cluster-tilting objects \cite{BMR07} in cluster categories, which are \textit{triangulated categories} in the sense of Verdier \cite{Ver96}, extend tilting objects in such a way that mutation always is possible. Several equivalent completions of tilting theory have since been put forward, such as
\textit{$\tau$-tilting theory} \cite{AIR2014}, which takes place in the abelian category of finitely generated modules, and \textit{two-term silting theory} \cite{DerksenFei2015,AIR2014}, where the setting is the category of morphisms of finitely generated projective modules up to homotopy. This latter category is an extension-closed subcategory of the triangulated bounded homotopy category. This makes it an \textit{extriangulated category} in the sense of Nakaoka--Palu \cite{NP19}. Moreover, it is an example of a \textit{0-Auslander} extriangulated category \cite{GNP23}. This class of extriangulated categories provides a unifying framework for many contexts within representation theory where mutation appears. This paper focuses on a combinatorial object encoding the mutation theories in these categories.

\begin{repsetup}{setting}
    Let $k$ be a field. Throughout $(\calC,\E)$ (or just $\calC$) denotes a $\Hom$-finite Krull--Schmidt 0-Auslander extriangulated $k$-category with a basic projective silting object $T$. 
\end{repsetup}

One particularly successful approach to cluster algebras is via the study of polyhedral fans. These are convex-geometric objects consisting of polyhedral cones in a vector space, which only overlap in their faces. For example, the clusters of a finite type cluster algebra define a complete simplicial fan \cite[{Thm.} 1.10]{FZ2003}, which is the normal fan of the generalised associahedron of the same type \cite[{Thm.} 1.4]{CFZ2002}, see also \cite{HLT2011,HPS2018,ReadingSpeyer2009}. These geometric realisations often reveal connections to other areas of mathematics. For example, one particularly insightful connection is with toric geometry \cite{GHK15,GHKK18}. 

A big breakthrough in the study of cluster algebras was achieved in \cite{GHKK18}, where it was established that the collection of Fock-Goncharov cluster chambers \cite{FG06} forms a simplicial fan \cite[{Thm.} 0.8]{GHKK18}. As a consequence, it is possible to establish the sign-coherence of $\bfg$-vectors using a new approach \cite[{Thm.} 5.11]{GHKK18}. These $\bfg$-vectors are integer vectors encoding the cluster variables with respect to a fixed cluster. In fact, many conjectures about $\bfg$-vectors have been resolved using the representation theory of finite-dimensional algebras and quivers \cite{DWZ2010}. This solidified the consideration of $\bfg$-vectors within representation theory \cite{DK08}, in particular within $\tau$-tilting theory \cite{AIR2014, DerksenFei2015}. Only a small class of finite-dimensional algebras are motivated by cluster algebras and yet every algebra exhibits a simplicial polyhedral fan called the \textit{$\bfg$-vector fan} \cite{DIJ2019}. Instead of compatible cluster variables, in the $\bfg$-vector fan of a finite-dimensional algebra, compatible $\tau$-rigid pairs define a cone. The connection between these two theories is established by cluster-tilting objects in cluster categories \cite{BMRRT06}. We are able to generalise this {connection} as follows.

\begin{reptheorem}{thm:IsFan}
    Let $\calC$ be as in \Cref{setting}. There exists a simplicial polyhedral fan $\partfan{\calC}$ whose maximal dimensional cones correspond to silting objects in $\calC$.
\end{reptheorem}

In particular, if $\Lambda$ is a finite-dimensional $k$-algebra, then the $\bfg$-vector fan of $\Lambda$ is {simply} given by $\partfan{\Ktwo(\proj \Lambda)}$, where the category $\Ktwo(\proj \Lambda)$ is the extension-closed subcategory of the bounded homotopy category $\Kb(\proj \Lambda)$ consisting of objects concentrated in cohomological degrees $-1$ and $0$. Here, silting objects are simply 2-term silting {objects} in $\Kb(\proj \Lambda)$ and thus correspond to support $\tau$-tilting pairs in the module category \cite{AIR2014} and cluster-tilting objects in the cluster category if $\Lambda$ is hereditary and representation-finite \cite{AIR2014}. We additionally extend \cite[{Thm.} 4.7]{Asa21} in \Cref{prop:Asa21.4.7}, showing that $\partfan{\calC}$ is a complete fan precisely when $\calC$ is {silting-}finite and {contains no nonzero projective-injective objects}. We also show that various functorial relationships between different 0-Auslander extriangulated categories induce morphisms between their fans.

An important feature of the aforementioned types of objects are their associated reduction techniques. Loosely speaking, every summand of a silting object, a cluster-tilting object or a $\tau$-tilting module gives rise to a reduction procedure. The output of this process is the part of the ambient category which is compatible with the summand. This often allows global properties of the categories to be studied locally. More precisely, one has 2-Calabi--Yau reduction \cite{IY08} involving summands of cluster-tilting objects, see also \cite{BMR2008,CK2006}, $\tau$-tilting reduction, involving summands of $\tau$-tilting modules \cite{Jas15,DIRRT17} and Iyama--Yang reduction of triangulated categories \cite{IY18,IY20}, involving summands of (not necessarily 2-term) silting objects, which has been generalised to extriangulated categories \cite{Bor24}. We obtain the following duality result on reduction techniques.

\begin{reptheorem}{prop:Monica}
    Let $\calC$ and $T$ be as in \cref{setting}, let $N$ be the maximal basic projective-injective object in $\calC$, and let $\Lambda = \End_{\calC}(T)/[N]$. 
    There are mutually inverse inclusion-reversing bijections
    \[\begin{tikzcd} \left\{ \begin{varwidth}{20em} \begin{center} thick subcategories of $\mathcal{C}$ generated by  presilting objects of the form $X \oplus N$ \end{center} \end{varwidth} \right\} \arrow[r, bend left, "\calW_{\calC}"] & \left\{ \begin{varwidth}{12em} \begin{center} $\tau$-perpendicular subcategories of $\mods(\Lambda)$ \end{center} \end{varwidth} \right\}. \arrow[l, bend left, "\calT_{\calC}"]  \end{tikzcd} \]
\end{reptheorem}

In the above, we call an abelian extension-closed subcategory of the category $\mods(\Lambda)$ of finite-dimensional right $\Lambda$-modules a \textit{$\tau$-perpendicular} subcategory if it is obtained via the process of $\tau$-tilting reduction \cite{Jas15,BH21}. On the other hand, a \textit{thick subcategory} of $\calC$ is an extriangulated subcategory, and in \cref{prop:Monica} above, we consider those which are the smallest thick subcategories containing $X \oplus N$ for some presilting object $X$. \cref{prop:Monica} generalises the bijection of \cite[Thm. 3.15]{Gar23} both in terms of the setting and in terms of the classes of subcategories involved in the bijection. We emphasise that this is a novel result even for $\calC = \Ktwo(\proj \Lambda)$.

For a finite-dimensional algebra $\Lambda$, the $\tau$-cluster morphism category \cite{IT17,BM18w,BH21,Bor21} is a category which encapsulated all possible $\tau$-tilting reductions within $\mods(\Lambda)$. That is, its objects are $\tau$-perpendicular subcategories and its morphisms are indexed by the {direct} summands of $\tau$-tilting pairs in such a way that reduction with respect to these summands yields the relevant subcategory. This category is closely related to exceptional sequences in the hereditary case \cite{IT17} and to $\tau$-exceptional sequences \cite{BM18t} in the general case. More precisely, its morphisms factorise into signed versions of such sequences {\cite[{§}11]{BM18t}}. Moreover, the classifying space of the $\tau$-cluster morphism category is a conjectural $K(\pi,1)$-space for the picture group of $\Lambda$ as defined in \cite{ITW16,HI21,Kai24}, see also \cite[Conjecture A.11]{BorveKaipeltautilt}. This inspires the name \textit{picture category} as introduced in \cite{Bor24}, which we will henceforth use. 

We show in \cref{prop:STTWgeneralisation} that $\partfan{\calC}$ admits an admissible partition, in the sense of \cite{Kai23}, which is induced by the thick subcategories of $\calC$. Therefore \cite[{§}3]{Kai23} yields a construction of a natural candidate for the picture category of $\calC$. We show that this definition encompasses the other established constructions of $\tau$-cluster morphism categories and picture categories. 

\begin{theorem}\label{thm:introthm1}
    Let $\calC$ {and $T$} be as in \Cref{setting} and let $\Lambda$ be any finite-dimensional $k$-algebra. There is a well-defined picture category $\Wfrak(\calC)$ satisfying:
    \begin{enumerate}
        \item The $\tau$-cluster morphism category of $\Lambda$ is equivalent to $\Wfrak(\Ktwo(\proj \Lambda))$;
        \item If $\calC$ does not admit nonzero projective-injective objects, then $\Wfrak(\Ktwo(\proj \End_{\calC}(T)) \simeq \Wfrak(\calC)$;
        \item If $\calC$ admits an exact dg enhancement in the sense of \cite{Che23}, then the picture category of $\calC$ as defined in \cite{Bor24} is equivalent to $\Wfrak(\calC)$. 
    \end{enumerate}
\end{theorem}
\begin{proof}
    The category $\Wfrak(\calC)$ is defined as the category of the partitioned fan $(\partfan{\calC}, \Pfrak_{\thick})$, where $\Pfrak_{\thick}$ identifies cones whose presilting objects generate the same thick subcategory. Thus \cite[{Prop.} 3.8]{Kai23} establishes this as a well-defined category. Then (1) and (2) are \cref{prop:samePartition} and (3) is \cref{prop:sameasBor}.
\end{proof}

In particular, \cref{prop:Monica}, which essentially implies \cref{thm:introthm1}(1), shows that our construction of the $\tau$-cluster morphism category of a finite-dimensional algebra from the $\bfg$-vector fan extends the main result of \cite{STTW23}. With the goal of establishing the classifying space of the picture category as a $K(\pi,1)$ space, some functorial relationships between the picture categories of different algebras or extriangulated categories have been established \cite{BH21,Bor24,BorveKaipeltautilt,Kai24}. In \cref{def:morpartfan} we introduce the notion of a morphism of partitioned fans as a morphism of fans, that is, a morphism of the ambient vector space mapping cones into cones, which preserves the identification of cones. We show that this combinatorial notion induces functors between picture categories, which unify the aforementioned established relationships.

\begin{theorem}\label{thm:introthm2}
    Let $\calC_1$ and $\calC_2$ be two categories satisfying \Cref{setting}. Assume that there is a morphism of partitioned fans $f: (\calC_1, (\Pfrak_{\thick})_1) \to (\calC_2, (\Pfrak_{\thick})_2)$, then the following hold: 
    \begin{enumerate}
        \item There is a functor $\mathfrak{C} f: \Wfrak(\calC_1) \to \Wfrak(\calC_2)$;
        \item If $f$ is injective on cones, then the functor $\mathfrak{C} f$ is faithful;
        \item If $f$ is surjective on cones, then the functor $\mathfrak{C} f$ is dense.
    \end{enumerate}
    Conversely, let $G: \calC_1 \to \calC_2$ be an extriangulated functor. Assume that $G$ preserves presilting objects and maps the projective silting object in $\calC_1$ to the projective silting object in $\calC_2$, then the following hold:
    \begin{enumerate}
        \item[(4)] The functor $G$ induces a morphism of partitioned fans $g: (\calC_1, (\Pfrak_{\thick})_1) \to (\calC_2, (\Pfrak_{\thick})_2)$;
        \item[(5)] If $G$ is injective on presilting objects, then $g$ is injective on cones;
        \item[(6)] If $G$ is surjective on presilting objects, then $g$ is surjective on cones.
    \end{enumerate}
\end{theorem}
\begin{proof}
    The first half is a special case of \cref{thm:partmorphgivesfunctor}(1)-(3) for the thick partitions $(\Pfrak_{\thick})_1$ and $(\Pfrak_{\thick})_2$. The second half is a reformulation of \cref{cor:partmorphgivesfunctor}(1)-(3).
\end{proof}

We remark that \cref{thm:partmorphgivesfunctor} works for any admissible partition of the fans. Morphisms of partitioned fans provide a unifying framework for all known relationships between $\tau$-cluster morphism categories and picture categories of different algebras and extriangulated categories as we now explain. The following also act as natural examples giving rise to morphisms of partitioned fans:
\begin{itemize}
    \item Let $\Lambda$ be a finite-dimensional $k$-algebra and let $K:k$ be a MacLane separable field extension. In \cref{prop:BK} we provide a new perspective on our earlier work by using our results of \cite{BorveKaipeltautilt} to establish an injective morphism of partitioned fans, which induces the faithful functor $\Wfrak(\Lambda) \to \Wfrak(\Lambda \otimes_k K)$ constructed in \cite[Thm. 0.4]{BorveKaipeltautilt}. 
    \item Let $\Lambda$ be a finite-dimensional $k$-algebra and $I$ an ideal of $\Lambda$. The derived tensor functor 
    \[ - \otimes_\Lambda^{\boldsymbol{L}} \Lambda/I: \Ktwo(\proj \Lambda) \to \Ktwo(\proj \Lambda /I)\]
    induces a morphism of partitioned fans. Via \cref{thm:introthm2}, we recover the functor $\Wfrak(\Lambda) \to \Wfrak({\Lambda}/I)$ constructed as part of \cite[Thm. 1.3]{Kai24}, see \cref{rem:generlises}\eqref{rem:generlises_Kai24}.
    \item Let $\Lambda$ be a finite-dimensional $k$-algebra and let $\calW$ be a $\tau$-perpendicular subcategory of $\mods(\Lambda)$. Then $\calW$ is equivalent to $\mods(\Gamma)$ for some finite-dimensional $k$-algebra $\Gamma$ by \cite[Thm. 1.1]{Jas15}, so it gives rise to a $\tau$-cluster morphism category $\Wfrak(\Gamma)$. In \cref{lem:redadj}\eqref{lem:redadj3} we give a fan-theoretic proof of the fact that $\Wfrak(\Gamma)$ is a full subcategory of $\Wfrak(\Lambda)$, which recovers \cite[Prop. 6.14]{BH21}. A similar result holds for picture categories, see \cite[§6]{Bor24}. 
    \end{itemize}

In order to establish that the classifying space of the picture category $\Wfrak(\calC)$ is a $K(\pi,1)$ space it is beneficial to obtain a faithful functor from  $\Wfrak(\calC)$ to a group $G$ regarded as a groupoid with one object \cite{Igu14}. If this is the case we say that $\Wfrak(\calC)$ admits a faithful group functor. We conclude this article by giving the following characterisation.

\begin{reptheorem}{thm:faithfulgroupfunctor}
    Let $\calC$ and $T$ be as in \cref{setting}, let $N$ be the maximal basic projective-injective object in $\calC$, and let $\Lambda = \End_{\calC}(T)/[N]$. 
    Then $\Wfrak(\calC)$ admits a faithful group functor if and only if $\Wfrak(\Lambda)$ admits a faithful group functor. 
\end{reptheorem}

This allows for the extension of various results which have established such a faithful functor for the $\tau$-cluster morphism category of finite-dimensional algebras to the picture categories of 0-Auslander extriangulated categories. For an overview of the most general results in this direction, see \cite{BorveKaipeltautilt,Kai23,Treffinger26}.

\textbf{Acknowledgements.} 
E.D.B. wishes to thank Nebojsa Pavic for useful discussions about polyhedral fans.
Both authors are grateful to Monica Garcia for helpful discussions concerning \Cref{prop:Monica}, and to Raphael Bennett-Tennenhaus for pointing out the paper \cite{ZLZ2026} to us. 

\section{Preliminaries}

We begin by introducing the central objects studied in this paper. Following \cite{NP19} and \cite{GNP23}, we recall (0-Auslander) extriangulated categories, which are the main object of study of this work. Silting objects and their corresponding silting reductions play important roles throughout and we recall their definitions and collect results needed for later sections. Finally, the convex geometric notion of a polyhedral fan will be discussed at the end of this section. The aim of later sections is to combine these notions and encode the silting theory of 0-Auslander extriangulated categories via polyhedral fans.

\subsection{0-Auslander extriangulated categories}

Throughout, we fix an arbitrary field $k$.
Let $\calC$ be an additive $k$-category. Given a $k$-linear bifunctor 
\begin{equation}\label{eq:Eext_bifun}
    \E\colon \calC\op\times\calC \to \Mod(k),
\end{equation}
we will refer to an element $\xi\in \E(C,A)$ as \textit{$\E$-extension}, where the objects $A,C\in \calC$ are arbitrary. 
A \textit{realisation} {$\mathfrak{s}$} of $\E$ is a set of maps $\mathfrak{s}_{C,A}$, indexed by objects $A,C\in \calC$, from the first to the second of the following sets:
\begin{enumerate}
    \item elements in $\E(C,A)$
    \item\label{threecomplexes} equivalence classes of three-term complexes in $\calC$ under the equivalence relation {which identifies} the top and bottom row {of} any morphism of complexes of the form
\begin{equation*}
\begin{tikzcd}
A \arrow[r,"\iota"]\arrow[d,equal] & B\arrow[r,"\pi"]\arrow[d] & C\arrow[d,equal] \\ 
A \arrow[r,"\iota'"]  & B' \arrow[r,"\pi'"] & C
\end{tikzcd}
\end{equation*}
where the outer morphisms are identity morphisms in $\calC$.
\end{enumerate}
We say that the realisation $\mathfrak{s}$ is \textit{exact} if it  satisfies the following additional properties:
\begin{enumerate}
	\item\label{realiz1} The $\E$-extension $0\in \E(C,A)$ is sent to the \textit{split sequence} $A \to A\oplus C\to C$, i.e. the first morphism is natural when $A\oplus C$ is regarded as a coproduct of $A$ and $C$, and the latter is natural when $A\oplus C$ is regarded as a product.
	\item We have that $\mathfrak{s}_{C\oplus C',A\oplus A'}(\xi\oplus \xi') = \mathfrak{s}_{C,A}(\xi)\oplus \mathfrak{s}_{C',A'}(\xi')$ for $\xi\in \E(C,A)$ and $\xi'\in \E(C',A')$.
	 \end{enumerate}
The three-term {complexes} in the image of {a given} $\mathfrak{s}_{C,A}$ will be called \textit{$\mathfrak{s}$-conflations}, or just \textit{conflations} when $\mathfrak{s}$ can be safely suppressed. 
{An \textit{$\E$-triangle} is defined by a pair} $(\begin{tikzcd}
A \arrow[r,"\iota"]& B\arrow[r,"\pi"]& C
\end{tikzcd},\xi)$ {where the first term is {a conflation} and the second term is an $\E$-extension whose realisation $\mathfrak{s}_{C,A}(\xi)$ is equivalent to the first term.}
An {$\E$-triangle} will be displayed as follows:
\begin{equation}\label{eq:realiz_conf}
\begin{tikzcd}
A \arrow[r,"\iota",tail]& B\arrow[r,"\pi",two heads]& C \arrow[r,dashed,"\xi"] & {}
\end{tikzcd}
\end{equation}
The morphism $\iota$ in \eqref{eq:realiz_conf} is called an \textit{inflation}, whereas $\pi$ is {called} a \textit{deflation}. Moreover, the morphism $\pi$ (or the object $C$) is {called} the \textit{cone} of $\iota$, and the morphism $
\iota$ (or the object $A$) is {called} the \textit{cocone} of $\pi$. {A \textit{morphism} of $\E$-triangles is given by}
\begin{equation*}\label{eq:morph_conf}
\begin{tikzcd}
A_1 \arrow[r,"\iota_1",tail]\arrow[d,"a"]& B_1 \arrow[r,"\pi_1",two heads]\arrow[d,"b"] & C_1 \arrow[r,dashed,"\xi_1"] \arrow[d,"c"] & {} \\
A_2 \arrow[r,"\iota_2",tail]& B_2 \arrow[r,"\pi_2",two heads]& C_2 \arrow[r,dashed,"\xi_2"] & {} 
\end{tikzcd}
\end{equation*}
{where the solid part of the diagram is a morphism of the underlying complexes, and $\E(c,A_2)(\xi_2) = \E(C_1,a)(\xi_1)$ as elements of $\E(C_1,A_2)$. We may thus consider a category of $\E$-triangles, and thus commutative diagrams of $\E$-triangles.}

We say that a full subcategory $\mathcal{D}$ of $\calC$ is \textit{extension-closed} if for {all $\E$-triangle}s of the form shown in \eqref{eq:realiz_conf} with $A,C\in\mathcal{D}$, we have $B\in \mathcal{D}$. Similarly, one can give a definition for a full subcategory of $\calC$ to be \textit{closed under cones} or \textit{closed under cocones}.

\begin{definition}[{\cite[{Def.}~2.12]{NP19}, see also \cite[{Def.} 2.32 and {Prop.} 4.3]{HLN17}}]
\label{def:extricat}
		Let $k$ be a field and let $\calC$ be an additive $k$-category.
        A pair $(\E,\mathfrak{s})$ is called an \textit{extriangulation} on $\calC$ if the axioms (\textbf{ET1})--(\textbf{ET4}$\op$) below are met. The triple~$(\calC,\E,\mathfrak{s})$ (or just $(\calC,\E)$, or indeed $\calC$, if the choice of $\E$ and/or $\mathfrak{s}$ cannot be confused) is called an \textit{extriangulated $k$-category}.
		\begin{enumerate}[leftmargin=1.42cm]
			\item[\textbf{(ET1)}$\;\;\;$]\label{ET1} $\E\colon \calC\op\times\calC\to \Mod({k})$ is a $k$-linear bifunctor.
                \item[\textbf{(ET2)}$\;\;\;$]\label{ET2} $\mathfrak{s}$ is an {exact realisation} of $\E$.
			\item[\textbf{(ET3)}$\;\;\;$]\label{ET3} 
			{Consider the following setup}
			\begin{equation}
				\begin{tikzcd}
				A_1\arrow[r,tail]\arrow[d] & E_1\arrow[r,two heads]\arrow[d] & B_1\arrow[d,dotted] \arrow[r,dashed] & {}\\
					A_2\arrow[r,tail] & E_2\arrow[r,two heads] & B_2 \arrow[r,dashed] & {}
				\end{tikzcd}
			\end{equation}
			{where the rows are $\E$-triangles in $(\calC,\E)$ and the solid part commutes. }  
			{Then one can complete the diagram with the dotted morphism to form a morphism of $\E$-triangles.}
			\item[\textbf{(ET3$\op$})]\label{ET3op} {The dual} of \textbf{(ET3}) above {holds} \cite[(ET3$\op$) in Definition~2.12]{NP19}.
			\item[\textbf{(ET4)}$\;\;\;$]\label{ET4} 
			{Let} 
			$$\begin{tikzcd} A\arrow[r,tail,"\iota"]& E\arrow[r,two heads,"\pi"]& B \arrow[r,dashed,"\xi_1"] & {}  \end{tikzcd} \quad \text{and} \quad \begin{tikzcd} E \arrow[r,tail,"\iota'"]& F \arrow[r,two heads,"\pi'"]& C\arrow[r,dashed,"\xi_2"] & {}\end{tikzcd}$$
			be $\E$-triangles {in $(\calC,\E)$}. 
			{We have a {commutative} diagram of {morphisms of} $\E$-triangles}  
			\begin{center}
				\begin{tikzcd}
					A\arrow[r,"\iota", tail]\arrow[d,equal] & E\arrow[r,"\pi",two heads]\arrow[d,tail,"\iota'"] & B\arrow[d,tail] \arrow[r,dashed, "\xi_1"] & {} \\
					A\arrow[r,tail] & F\arrow[r,two heads] \arrow[d,two heads,"\pi'"] & G\arrow[d,two heads] \arrow[r,dashed, "\xi_3"] & {}  \\
					& C\arrow[r,equal] \arrow[d,dashed, "\xi_2"] & C \arrow[d,dashed, "\xi_4"] & \\
					& {} & {} &
				\end{tikzcd}
			\end{center}
			\item[\textbf{(ET4$\op$)}]\label{ET4op} {The dual} of \textbf{(ET4)} above {holds} \cite[{Rem.}~2.22]{NP19}.
		\end{enumerate}
\end{definition}

When working with categories with additional structure, it is natural to consider functors which preserve this structure.

\begin{definition}
    Let $(\calC,\E,\mathfrak{s})$ and $(\calC',\E',\mathfrak{s}')$ be extriangulated $k$-categories. An \textit{extriangulated functor} \cite[{Def.} 2.32]{B-TS20} from $(\calC,\E,\mathfrak{s})$ to $(\calC',\E',\mathfrak{s}')$ is given by a pair $(F,\psi)$, where $F$ is a $k$-linear functor $F\colon \calC \to \calC'$ and $\psi$ a natural transformation $\psi\colon \E \Longrightarrow \E'$ of bifunctors $\calC\op\times\calC\xrightarrow{}\Mod(k)$, such that the following square commutes for all $A,C\in\calC$:
    \begin{equation*}
        \begin{tikzcd}[column sep=4em]
            \E(C,A) \arrow[r,"\psi_{C,A}"] \arrow[d,"{\mathfrak{s}_{C,A}}"] & \E'(FC,FA) \arrow[d,"{\mathfrak{s}'_{FC,FA}}"] \\
            \mathcal{C}om_{\calC}(C,A)/\sim \arrow[r,"F"] & \mathcal{C}om_{\calC'}(FC,FA)/\sim
        \end{tikzcd}
    \end{equation*}
    where the sets along the bottom row are equivalence classes of three-term complexes in $\calC$.
    An extriangulated functor $(F,\psi)$ is an \textit{equivalence of extriangulated $k$-categories} {if} $F$ and $\psi$ are both invertible \cite[{Prop.} 2.13]{NOS22}, or equivalently if it is inverted by an extriangulated functor from $(\calC',\E',\mathfrak{s}')$ to $(\calC,\E,\mathfrak{s})$ \cite[{Prop.} 4.11]{BTHSS23}.
\end{definition}

{In this article, we require {extriangulated categories} to admit certain {properties}, which provide the foundation for a well-behaved theory {of mutation}.}

\begin{definition}[{\cite[Def. 3.23, Prop. 3.24 and Def. 3.25]{NP19}}]
Let $(\calC,\E)$ be an extriangulated $k$-category.
    \begin{enumerate}
        \item An object $X$ in $(\calC,\E)$ is \textit{$\E$-projective} (resp. \textit{$\E$-injective}) if the functor $\E(X,-)\colon \calC\to \Mod({k})$ (resp. $\E(-,X)\colon \calC\op\to \Mod({k})$) is identically zero. An \textit{$\E$-projective-injective object} in $\calC$ is {an object which is} both $\E$-projective and $\E$-injective. We will simply use the terms \textit{projective}, \textit{injective} and \textit{projective-injective} when there can be no confusion as to the choice of bifunctor $\E$.
        \item We say that $(\calC,\E)$ \textit{has enough projectives} if for every object ${X} \in \calC$ one can find a deflation $\begin{tikzcd} T_X \arrow[r,two heads,"\pi_X"] & X\end{tikzcd}\!\!$, where $T_X$ is projective.
    \end{enumerate}    
\end{definition}

We are now ready to define the main classes of extriangulated categories which we will consider throughout this paper. 

\begin{definitionproposition}[{\cite[Prop. 2.1, Def. 2.3 and Def. 3.7]{GNP23}}]\label{defprop:0Aus}
Let $(\calC,\E)$ be an extriangulated $k$-category.
\begin{enumerate}
    \item We say that $(\calC,\E)$ is \textit{hereditary} if $\E(X,-)\colon \calC\to \Mod(k)$ is right exact for all objects $X \in \calC$. The property of \textit{having enough injectives} is defined dually.
    In the case where $(\calC,\E)$ has enough projectives, it is hereditary if and only if all objects $X\in\calC$ admit an $\E$-triangle 
    \begin{equation*}
        \begin{tikzcd} T_1\arrow[r,tail] &  T_0 \arrow[r,two heads] & X \arrow[r,dashed] & {} \end{tikzcd}
    \end{equation*}
    where $T_0$ and $T_1$ are projective.
        \item We say that $(\calC,\E)$ is \textit{0-Auslander} if it is hereditary, has enough projectives, and any projective object $T\in\calC$ is the domain of an inflation $\begin{tikzcd} T \arrow[r,tail] & Q \end{tikzcd}$ with projective-injective codomain.
    \item We say that $(\calC,\E)$ is \textit{reduced} if every projective-injective object therein is a zero object.
\end{enumerate}
\end{definitionproposition}

In \cite{GNP21}, \textit{the higher extension} bifunctors $\E^i(-,-): \calC\op \times \calC \to \Mod(k)$ are defined for all $i \geq 1$. It is shown in \cite[Thm. 3.5]{GNP21} that these give rise to long exact sequences. However, by \cite[Prop. 2.1(1)(iii)]{GNP23}, the bifunctor $\E^2(-,-)$ vanishes for hereditary extriangulated categories and thus we obtain the following:

\begin{proposition}[{\cite[Cor. 3.12]{NP19}}]
    \label{lem:LES}
    Let $(\calC,\E)$ be a small extriangulated $k$-category, let 
   	\[ \begin{tikzcd}
			X \arrow[r,tail, "f"] & Y \arrow[r,two heads, "g"] & Z\arrow[r,dashed, "\delta"] & {}
	\end{tikzcd} 
	\]
	be an $\E$-triangle in $\calC$, and let $U$ be an object in $\calC$. The covariant Hom-functor $\calC(U,-)$ induces an exact sequence of $k$-vector spaces
    \begin{equation*}
    \begin{tikzcd}
        \calC(U,X) \arrow[r, "f \circ -"] & \calC(U,Y) \arrow[r, "g \circ -"] & \calC(U,Z) \arrow[r, "(\delta_U)_{\#}"]  &  \E(U,X) \arrow[r, "f_*"] & \E(U,Y) \arrow[r, "g_*"] & \E(U,Z),
        \end{tikzcd}
    \end{equation*}
    where we denote $\E(U,f)$ by $f_*$ (resp. $\E(U,g)$ by $g_*$), and $(\delta_U)_{\#}$ sends $h$ to $\E(h,X)(\delta)$.
    Dually, the contravariant Hom-functor $\calC(-,U)$ induces an exact sequence    
    \begin{equation*}
    \begin{tikzcd}
        \calC(Z,U) \arrow[r, "- \circ g"] & \calC(Y,U) \arrow[r, "- \circ f"] &\calC(X,U) \arrow[r, "\delta_U^{\#}"] & \E(Z,U) \arrow[r, "g^*"] & \E(Y,U) \arrow[r, "f^*"] & \E(X,U),
        \end{tikzcd}
    \end{equation*}
    where we denote $\E(g,U)$ by $g^*$ (resp. $\E(f,U)$ by $f^*$), and $\delta_U^{\#}$ sends $h$ to $\E(Z,h)(\delta)$.
\end{proposition}

If an extriangulated $k$-category has enough injectives, we obtain the following useful version for extriangulated categories we use in this paper. If $\mathcal{X}$ is a full subcategory of $\calC$, let $[\mathcal{X}]$ denote the ideal of $\calC$ consisting of morphism factoring through an object in $\mathcal{X}$, and let $\calC/[\mathcal{X}]$ denote the ideal quotient.  If $\mathcal{X}$ is of the form $\add(X)$, for some object $X\in \calC$, we write $[X]$ when we mean $[\add(X)]$. If $h$ is a morphism in $\calC$ we write $[h]$ for its equivalence class in $\calC/[\mathcal{X}]$.

\begin{lemma}\label{lem:LESmodinj}
	Let $(\calC,\E)$ be a small extriangulated $k$-category with enough injectives and let $\mathcal{I}$ be the full subcategory of injective objects in $\calC$. Let 
   	\[ \begin{tikzcd}
			X \arrow[r,tail, "f"] & Y \arrow[r,two heads, "g"] & Z\arrow[r,dashed, "\delta"] & {}
	\end{tikzcd} 
	\]
	be an $\E$-triangle in $\calC$, and let $U$ be an object in $\calC$. There is an exact sequence
	 \begin{equation*}
    \begin{tikzcd}
        \frac{\calC}{[\calI]}(Z,U) \arrow[r, "- \circ {[g]}"] & \frac{\calC}{[\calI]}(Y,U) \arrow[r, "- \circ {[f]}"] &\frac{\calC}{[\calI]}(X,U) \arrow[r, "\overline{\delta_U^{\#}}"] & \E(Z,U) \arrow[r, "g^*"] & \E(Y,U) \arrow[r, "f^*"] & \E(X,U).
        \end{tikzcd}
    \end{equation*}
\end{lemma}
\begin{proof}
	\underline{Exactness at $\E(Y,U)$} follows from \Cref{lem:LES}.
	
	\underline{Exactness at $\E(Z,U)$:} We show that the morphism $\delta_U^{\#}$ factors through $\frac{\calC}{[\calI]}(X,U)$. Denote by $[\calI](X,U)$ the ideal of $\calC$ consisting of morphisms $X \to U$ which factor through an injective object. It is sufficient to show that $[\calI](X,U) \subseteq \ker(\delta_U^{\#})$. Thus, let $j: X \xrightarrow{h} J \xrightarrow{l} U$ be a morphism in $\calC$, where $J$ is injective. Then we have
	\[\delta_U^{\#}(j) = \delta_U^{\#}(l\circ h) = \E(Z,l\circ h)(\delta) =\E(Z,l)\left(\E(Z,h)(\delta)\right)=0,\]
	because $\E(Z,h)(\delta) \in \E(Z,J)= 0$ as $J$ is injective. We may therefore unambiguously define $\overline{\delta_U^{\#}}([h]) = \delta_U^{\#}(h)$ and conclude that $\image (\overline{\delta_U^{\#}})= \image(\delta_U^{\#}) = \ker( g^*)$.
	
	\underline{Exactness at $\frac{\calC}{[\calI]}(X,U)$:} Let $h: Y \to U$, then $\delta_U^{\#}(h \circ f) = 0$ by the exactness of the sequence in $\calC$, see \cref{lem:LES}. This implies $\image( - \circ {[f]}) \subseteq \ker(\overline{\delta_U^{\#}})$. To show the converse, let $[h] \in \ker(\overline{\delta_U^{\#}})$, which means $\delta_U^{\#}(h) = 0$. From the exactness of the long exact sequence in $\calC$, see \cref{lem:LES}, there exists $j: Y \to U$ such that $h = j \circ f$. Therefore we have $[h] = [j \circ f]$ and hence $\ker(\overline{\delta_U^{\#}}) \subseteq \image(- \circ {[f]})$ as required.

	\underline{Exactness at $\frac{\calC}{[\calI]}(Y,U)$:} Since $\calC$ has enough injectives, let 
	\[ \begin{tikzcd}
			X \arrow[r,tail, "\alpha"] & J \arrow[r,two heads] & A \arrow[r,dashed] & {}
	\end{tikzcd} 
	\]
	be an $\E$-triangle with $J \in \calI$. Applying $\calC(-,I)$ for any $I \in \calI$ and using $\E(-,I)=0$ in the long exact sequence obtained via \cref{lem:LES} shows that $\alpha$ is a left $\calI$-approximation of $X$. This means that every morphism $X \to U$ which factors through an injective object factors through $\alpha$. Using a ``shifted'' version of (\textbf{ET4}) \cite[Lem. A.10]{GNP23}, we can find a {commutative} diagram of $\E$-triangles
		\begin{equation*}
			\begin{tikzcd}[column sep=4em]
                                  X \arrow[r,tail, "f"] \arrow[d, "\alpha", tail, swap] & Y \arrow[r,two heads, "g"] \arrow[d, tail, "m"] & Z\arrow[r,dashed, "\delta"] \arrow[d, equal] & {} \\
                                  J \arrow[r,tail, "i"] \arrow[d, two heads] & M \arrow[r,two heads] \arrow[d, two heads] & Z\arrow[r,dashed, "{\E(Z,\alpha)(\delta)}"] & {} \\
                                  A \arrow[r, equal] \arrow[d,dashed] & A \arrow[d, dashed]\\
                                  {} & {}
	\end{tikzcd}
		\end{equation*}
	Because $J$ is injective in $\calC$, the middle row is a split conflation and $i$ admits a retraction $r: M \to J$. From the commutative top-left square we get $r\circ m\circ f = r\circ i\circ \alpha = \alpha$. Let now $[h] \in \ker(- \circ {[f]})$. Then $h: Y \to U$ is such that $h \circ f \in [\calI](X,U)$. Since $\alpha$ is a left $\calI$-approximation of $X$ we can write $h \circ f = j \circ \alpha$ for some $j: J \to U$. Using $r\circ m \circ f= \alpha$ we obtain
	\[ ( h - j \circ r \circ m )\circ f = h \circ f - j \circ r \circ m \circ f = h \circ f - j \circ \alpha = 0. \]
	It follows from the exactness of the long exact sequence in $\calC$, see \cref{lem:LES}, there exists $l: Z \to U$ such that $h-j \circ r \circ m = l \circ g$. Because $j \circ r \circ m$ factors through an injective object $J$, we have $[h] = [l \circ g]$ in $\calC/[\calI]$. In other words, $\ker( - \circ {[f]}) \subseteq \image( - \circ {[g]})$. The reverse inclusion is immediate because $((- \circ f) \circ g) = - \circ (f \circ g) = 0$ holds in $\calC$. This concludes the proof.
\end{proof}

Note that 0-Auslander extriangulated categories have enough injectives \cite[Prop. 3.6]{GNP23}. The following are some commonly studied and interesting examples of 0-Auslander extriangulated categories.

\begin{example}\label{eq:0Aus}
\,
\begin{enumerate}
    \item\label{eg:0Aus_K2} Let $\Lambda$ denote a finite-dimensional $k$-algebra. Then $\Ktwo(\proj\Lambda)$, namely the category of morphisms of finitely generated projective $\Lambda$-modules up to homotopy, is a reduced 0-Auslander extriangulated $k$-category \cite[§3.3.3]{GNP23}.
	\item\label{eg:0Aus_per2} 
    More generally, let $A$ denote a non-positive {differential graded (}dg{)} $k$-algebra.
    Consider $\pertwo(A)$, namely the subcategory of the {perfect} derived category {$\per(A)$} spanned by objects $X$ fitting in a triangle
    \begin{equation*}
    \begin{tikzcd}
        T_1\arrow[r] & T_0\arrow[r] & X \arrow[r] & \Sigma T_1, 
    \end{tikzcd}
    \end{equation*}
    where $T_0$ and $T_1$ are in $\add(A)$.
    This is also an example of a reduced 0-Auslander extriangulated $k$-category \cite[§3.3.8]{GNP23}. 
	\item\label{eg:0Aus_mod} 
	Let $n$ be a positive integer, and let $\Lambda_n$ be the path $k$-algebra of the linearly oriented $A_n$-quiver, namely
	\begin{equation*}
	\Lambda_n = k \Big( \begin{tikzcd}
	1\arrow[r] & 2 \arrow[r] & \cdots \arrow[r] & n
	\end{tikzcd}\Big).
	\end{equation*}
	Then the category $\mods(\Lambda_n)$ of finite-dimensional $\Lambda_n$-modules is a 0-Auslander extriangulated $k$-category \cite[§3.3.2]{GNP23} with abelian extriangulated structure. For $n\geq 2$, it is not reduced.
\end{enumerate}
\end{example}

In fact, the projective and the injective objects of a 0-Auslander extriangulated category are closely related to another.

\begin{lemma}[{\cite[Lem. 3.12]{GNP23},\cite[Rem. 4.32 and Cor. 4.35]{PPPP23}}]\label{lem:GNP23.3.12}
    Let $(\calC,\E)$ be a 0-Auslander extriangulated $k$-category, let $\mathcal{P}$ denote the full subcategory of projective objects in $\calC$, let $\mathcal{I}$ denote the full subcategory of injective objects in $\calC$ and let $\mathcal{N}$ denote the full subcategory of projective-injective objects in $\calC$.
    There are mutually quasi-inverse additive equivalences
    \begin{equation}\label{eq:GNP23.3.12}
        \begin{tikzcd}[column sep=5em]
            \mathcal{P}/[\mathcal{N}]\arrow[r,"\Sigma", bend left=10] & \arrow[l,"\Omega", bend left=10] \mathcal{I}/[\mathcal{N}],
        \end{tikzcd}
    \end{equation}
    given by exchanging $P \in \mathcal{P}$ and $I \in \mathcal{I}$ in an $\E$-triangle
    \begin{equation}\label{eq:SigmaOmegaConfl}
        \begin{tikzcd}
            P \arrow[r,tail] & N' \arrow[r,two heads] & I \arrow[r, dashed] & {},
        \end{tikzcd}
    \end{equation}
    where $N'$ is projective-injective (necessarily $0$ whenever $\calC$ is reduced).
\end{lemma}

We note that all choices of projective-injective object in {\cref{eq:SigmaOmegaConfl}} yield the same equivalences of categories in \eqref{eq:GNP23.3.12}.
For example, if we apply \Cref{lem:GNP23.3.12} to $\Ktwo(\proj\Lambda)$, for some finite-dimensional $k$-algebra $\Lambda$, the equivalences in \eqref{eq:GNP23.3.12} simply become 
\begin{equation*}
        \begin{tikzcd}[column sep=5em]
            \add(\Lambda)\arrow[r,"{\Sigma}", bend left=10] & \arrow[l,"{\Sigma^{-1}}", bend left=10] \add(\Sigma\Lambda),
        \end{tikzcd}
    \end{equation*}
    namely the suspension and desuspension functors in the derived category of $\Lambda$.
    
   \subsection{Silting objects}

We now introduce a class of objects which plays a central role throughout this paper. Silting theory was first introduced in the context of triangulated categories in \cite{KV88}. The definitions can be adapted to the setting of extriangulated categories \cite{AT22,AT23,GNP23}.

\begin{definition}
    Let $(\calC,\E)$ be a hereditary extriangulated $k$-category.
    \begin{enumerate}
    \item A full subcategory of $\calC$ is \textit{thick} if it is closed under isomorphisms, direct summands, {$\E$-}extensions, cones and cocones. The smallest thick subcategory of $\calC$ containing a full subcategory $\mathcal{X}\subseteq \calC$ will be denoted $\thick_{\calC}(\mathcal{X})$, or simply $\thick(\mathcal{X})${, and called the \textit{thick subcategory generated}} by $\mathcal{X}$. 
    \item A full subcategory $\mathcal{R}\subseteq \calC$ is \textit{presilting} if $\E(R_1,R_2)=0$ for all $R_1,R_2\in \mathcal{R}$. A presilting subcategory is \textit{silting} if additionally $\thick(\mathcal{R})=\calC$.
        \item An object $R\in \calC$ is \textit{presilting} {if} $\add(R)$ is a presilting subcategory of $\calC$, where $\add(R)$ denotes the smallest subcategory of $\calC$ which contains $R$ and is closed under isomorphisms, direct sums and direct summands. A presilting object $R\in \calC$ {is} \textit{silting} {if} $\add(R)$ is a silting subcategory of $\calC$.
    \item Denote the set of presilting {objects} in $\calC$ by $\presilt(\calC)$, and the set of silting {objects} by $\silt(\calC)$. {Let {$R$} be presilting, then we denote the set of presilting (resp. silting) {objects} in $\calC$ which contain {$R$ as a direct summand} by $\presilt_{{R}}(\calC)$ (resp. by $\silt_{R}(\calC)$).} When $\calC=\Ktwo(\proj \Lambda)$ for some finite-dimensional $k$-algebra $\Lambda$, we write $\presilt(\Lambda)$ and $\silt(\Lambda)$ for $\presilt(\calC)$ and $\silt(\calC)$, respectively.
    \end{enumerate}
\end{definition}

Recall that an additive category is \textit{Krull--Schmidt} if every object decomposes into a finite direct sum of objects with local endomorphism rings. 

\begin{remark}\label{rem:silt}
\begin{enumerate}
    \item\label{rem:siltobj} If the extriangulated k-category $(\calC,\E)$ is Krull--Schmidt and contains a silting object $T$, then every presilting subcategory of $\calC$ is of the form $\add(X)$ for some presilting object $X\in\calC$ {by} \cite[Prop. 5.4]{AT22}. Moreover, every basic silting object has the same number of indecomposable direct summands \cite[Cor. 4.2]{AT23}. 
    \item\label{rem:proj-inj-silt}
    Let $(\calC,\E)$ be a 0-Auslander extriangulated $k$-category and let $\calN\subseteq \calC$ be a subcategory in which every object is {projective}-injective. Then $\calN$ is contained in every silting subcategory of $(\calC,\E)$ {by} \cite[Rem. 4.4(d)]{GNP23}. In particular, by \eqref{rem:siltobj} above, if $(\calC,\E)$ has a silting object and is Krull--Schmidt, then every {indecomposable} projective-injective object in $(\calC,\E)$ is a direct summand of every silting object in $(\calC,\E)$. In that case, there is a uniquely determined basic projective-injective object $N\in \calC$ with a maximal number of indecomposable direct summands, and $N$ is a direct summand of every silting object in $\calC$.
\end{enumerate}
\end{remark}

Throughout the main matter of this text, we work in the following setting.

\begin{setup}\label{setting}
    Let $(\calC,\E)$ (or just $\calC$) denote a Hom-finite Krull--Schmidt 0-Auslander extriangulated $k$-category with a basic projective silting object $T$. Furthermore, denote by $(\overline{\calC},\overline{\E})$ the corresponding reduction by the projective-injective objects {as in \cite[Prop. 3.30]{NP19}}. This is to set $\overline{\calC}\coloneqq {\calC}/[{\add(N)}]$ and $\overline{\E}$ to be the induced bifunctor, where $N$ is the maximal basic projective-injective presilting object in $\calC$. Let $\Lambda$ denote the endomorphism $k$-algebra of $T$ in $\overline{\calC}$, which is then a finite-dimensional $k$-algebra. We have
    \[ \Lambda \coloneqq \End_{\overline{\calC}}(T) \simeq  \End_{\calC/[N]}(T) \simeq \End_{\calC}(T)/[N]. \]
\end{setup}

We remark that some works we cite additionally assume the condition (WIC) of \cite[Cond. 5.8]{NP19} on the category $\calC$, which postulates that if a composite morphism $g \circ f$ in $\calC$ is an inflation (resp. a deflation), then $f$ is an inflation (resp. $g$ is a deflation). However, an extriangulated category satisfies (WIC) if and only if it is weakly idempotent complete \cite[Prop. C]{Kla22}. In particular, Krull--Schmidt categories are (weakly) idempotent complete \cite[Cor. 4.4]{Kra15}.

In \Cref{setting}, the relationship with module categories is described in the following.

\begin{lemma}[{\cite[{Prop.} 3.11]{GNP23}}]\label{lem:GNP23.3.11}
   Let $\calC$, $\Lambda$ and $T$ be as in \cref{setting}.
    The composite functor
    \begin{equation*}
       \overline{(-)}\colon \calC \to \overline{\calC} \xrightarrow{\overline{\calC}(T,-)} \mods(\Lambda)
    \end{equation*}
    induces a $k$-linear equivalence of categories
    \begin{equation}\label{eq:GNP23.3.11}
      \overline{(-)}\colon \calC/[\mathcal{I}] \to \mods(\Lambda),
    \end{equation}
    where $\mathcal{I}$ denotes the full subcategory of $\calC$ spanned by injective objects. 
\end{lemma}

The assumptions of \cref{setting} allow us to establish a well-behaved relationship between presilting and silting objects.

\begin{definitionproposition}[{see \cite[Corollary 4.8]{GNP23},\cite[{Lem.} 4.2]{IJY14}}]\label{defprop:Bongartz} 
    Let $\calC$ and $N$ be as in \Cref{setting}, and let $I$ denote the maximal basic injective object in $\calC$. Let $U$ be a presilting object in $\calC$. There is an $\E$-triangle 
    \begin{equation}\label{eq:Bongartz1}
    \begin{tikzcd}
	V^U \arrow[r ,tail] & U' \arrow[r,"\beta_U",two heads] &I \arrow[r,dashed] & {}
    \end{tikzcd}
    \end{equation}
    where the deflation $\beta_U$ is a (minimal) right $\add(U\oplus N)$-approximation of $I$. One defines the \textit{Bongartz completion} of $U$ to be the basic object $U^+$ in $\calC$ such that $ \add(U^+)= \add(U \oplus N \oplus V^U)$. This is a silting object in $\calC$ having $U$ as a direct summand. The existence of Bongartz completions shows that presilting objects in 0-Auslander extriangulated $k$-{categories} satisfying \Cref{setting} are precisely the direct summands of silting objects. 
\end{definitionproposition}

We will need a ``shifted'' version of the triangle in \eqref{eq:Bongartz1} to conduct some of our proofs below.

\begin{lemma}\label{prop:Bongartz}
Let $\calC$, $T$ and $N$ be as in \Cref{setting} and let $U$ be a presilting object in $\calC$.
	For every $T_0\in \add(T)$, there exists {an $\E$-triangle}
	 \begin{equation}\label{eq:Bongartz2}
    \begin{tikzcd}[ampersand replacement=\&,column sep=3em]
	T_0  \arrow[r,"\alpha" ,tail] \&  X_0 \arrow[r,two heads] \& U_0 \arrow[r,dashed] \& {}
    \end{tikzcd}
    \end{equation}
    in which $\alpha$ is a left $\add(U^+)$-approximation and $U_0$ is in $\add(U)$.
	\end{lemma}
		\begin{proof}
	Assume that $T_0$ is indecomposable. This is sufficient, because we may otherwise take \eqref{eq:Bongartz2} to be a direct sum of $\E$-triangles. If $T_0$ is injective as well, we can take \eqref{eq:Bongartz2} to be the split $\E$-triangle
    \begin{equation*}
    \begin{tikzcd}[ampersand replacement=\&,column sep=3em]
	T_0  \arrow[r,"\alpha" ,tail] \&  T_0 \arrow[r,two heads] \& 0, \arrow[r,dashed] \& {}
    \end{tikzcd}
    \end{equation*}
    since $U^+$ contains every indecomposable projective-injective as a direct summand. We therefore proceed by assuming that $T_0$ is indecomposable and non-injective. This will ensure that $\Sigma T_0$ is non-zero when regarded as an object in $\calC$. By \Cref{defprop:Bongartz}, we can find an $\E$-triangle
    \begin{equation}\label{eq:Bongartz2dp}
    \begin{tikzcd}
	V^U_0 \arrow[r ,tail] & U'_0 \arrow[r,"\beta_0",two heads] & \Sigma T_0 \arrow[r,dashed] & {}
    \end{tikzcd}
    \end{equation}
    where the deflation $\beta_{0}$ is a right $\add(U\oplus N)$-approximation of $\Sigma T_0$. We also fix an $\E$-triangle
     \begin{equation}\label{eq:Bongartz2Sigma}
    \begin{tikzcd}
	T_0\arrow[r ,tail] & N_0 \arrow[r,two heads] & \Sigma T_0 \arrow[r,dashed] & {}
    \end{tikzcd}
    \end{equation}
    with $N_0$ projective-injective, see \Cref{lem:GNP23.3.12}. By a ``shifted'' version of (\textbf{ET4}$^\text{op}$) \cite[Dual of Lem. A.10]{GNP23}, we can find a {commutative} diagram of $\E$-triangles
		\begin{equation}\label{eq:1.13ET4op}
			\begin{tikzcd}
                                   & T_0 \arrow[r, equal] \arrow[d,tail] & T_0 \arrow[d,tail]                                  &    \\
V^T_U \arrow[r,tail] \arrow[d, equal] & X'_0 \arrow[d,two heads] \arrow[r,two heads]          & N_0 \arrow[d,two heads] \arrow[r, dashed]                & {} \\
V^T_U \arrow[r,tail]                    & U'_0 \arrow[r, two heads,"\beta_0"] \arrow[d, dashed] & \Sigma T_0 \arrow[r, dashed] \arrow[d, dashed] & {} \\
                                   & {}                             & {}                                           &   
\end{tikzcd}
		\end{equation}
		The bottom row in \eqref{eq:1.13ET4op} is the $\E$-triangle in \eqref{eq:Bongartz2dp} and the right column in \eqref{eq:1.13ET4op} is the $\E$-triangle in \eqref{eq:Bongartz2Sigma}. We will adapt the middle column of \eqref{eq:1.13ET4op} to an $\E$-triangle with the desired properties. Set $U_0=U'_0/N'$, where $N'$ is the largest projective-injective direct summand of $U'_0$. By the projectivity of $N'$, we have that $N'$ is a direct summand of $X'_0$, whence we set $X_0=X'_0/N'$. We then have an $\E$-triangle 
		 \begin{equation*}    
		 \begin{tikzcd}[ampersand replacement=\&,column sep=3em]
	T_0  \arrow[r,"\alpha" ,tail] \&  X_0 \arrow[r,two heads] \& U_0 \arrow[r,dashed] \& {}
    \end{tikzcd}
    \end{equation*}
    which will now be shown to meet our requirements. First of all, we have $U_0\in \add(U)$ by construction.
    The middle row in {\eqref{eq:1.13ET4op}} is a split $\E$-triangle because $N_0$ is projective. Since $N_0$ is projective-injective, we have $N_0 \in \add(U^+)$ as $U^+$ is silting. It follows that $X'_0 \in \add(U^+)$ and therefore $X_0 \in \add(U^+)$ because also $V_U^T \in \add (U^+)$. To show that $\alpha$ is a left $\add(U^+)$-approximation of $T_0$, consider the induced exact sequence
     \begin{equation*}    
		 \begin{tikzcd}[ampersand replacement=\&,column sep=3em]
		 \calC(X_0,U) \arrow[r,"-\circ \alpha"] \&  \calC(T_0,U)  \arrow[r,two heads] \& \E(U_0,U)=0,
    \end{tikzcd}
    \end{equation*}
    in which the first morphism is indeed an epimorphism, as $U$ is a presilting object. The proof is complete.
	\end{proof}

0-Auslander extriangulated categories are closely related to module categories of finite-dimensional algebras. We need to define the following important class of objects in those categories.

\begin{definition}[{\cite{AIR2014}}]\label{def:tautilt}
    Let $\Lambda$ be a finite-dimensional $k$-algebra, and let $\tau$ be the Auslander--Reiten translation in $\mods(\Lambda)$.
    If a $\Lambda$-module $M$ satisfies $\Hom_{\Lambda}(M,\tau M)=0$, we say that $M$ is \textit{$\tau$-rigid}. A pair $(M,Q)$ of right $\Lambda$-modules is \textit{$\tau$-rigid} if $M$ is a $\tau$-rigid ${\Lambda}$-module and $Q$ is a projective $\Lambda$-module such that $\Hom_{{\Lambda}}(Q,M)=0$. A $\tau$-rigid pair $(M,Q)$ is \textit{support $\tau$-tilting} if in addition $|M|+|Q|=|{\Lambda}|$, where {$|-|$ denotes} the number of non-isomorphic indecomposable direct summands. 
	\end{definition}

\setcounter{theorem}{15}

Moreover, the following proposition establishes a relationship between the presilting objects in the extriangulated category and the $\tau$-rigid objects in the module category.

\begin{proposition}\label{prop:F}
Let $\calC$, $T$, $N$ and $\overline{\calC}$ be as in \Cref{setting}.
\begin{enumerate}
    \item\label{prop:PZ24.4.5} The functor $\overline{(-)}$ of \cref{lem:GNP23.3.11} induces a map
    \begin{equation}\label{eq:presilttaurigidbij}
    \begin{tikzcd}[row sep=0.5em]
        \presilt(\calC) \arrow[r] & \taurigidpair(\Lambda) \\
        X \arrow[r,mapsto] \arrow[u,phantom, sloped, "\in"] & (\overline{X},\overline{\Omega I_X})\arrow[u,phantom, sloped, "\in"] 
    \end{tikzcd}
    \end{equation}
    into the set of $\tau$-rigid pairs in $\mods(\Lambda)$, where $I_X$ is the maximal injective non-projective direct summand of $X$ and $\Omega$ is as in \Cref{lem:GNP23.3.12}. If $\calC$ is reduced, then the map in \eqref{eq:presilttaurigidbij} is a bijection and furthermore restricts to a bijection
    \begin{equation*}
    \begin{tikzcd}[row sep=0.5em]
        \silt(\calC) \arrow[r] & \stautiltpair(\Lambda) \\
        X \arrow[r,mapsto] \arrow[u,phantom, sloped, "\in"] & (\overline{X},\overline{\Omega I_X})\arrow[u,phantom, sloped, "\in"] 
    \end{tikzcd}
    \end{equation*}
    into the set of support $\tau$-tilting pairs in $\mods(\Lambda)$. 
\end{enumerate}
    Let $\mathbf{p}\colon \mods(\Lambda) \to \Ktwo(\proj \Lambda)$ denote the functor assigning $M\in \mods(\Lambda)$ its minimal projective presentation.
    \begin{enumerate}
    \setcounter{enumi}{1}
        \item\label{thm:modulecatreduction_presilt} 
    There is a map
    \begin{equation}\label{eq:presilttaurigidbijF}
    \begin{tikzcd}[row sep=0.5em]
        F \colon \presilt(\calC) \arrow[r] & \presilt (\Lambda)  \\
        X \arrow[r,mapsto] \arrow[u,phantom, sloped, "\in"] & \mathbf{p}\overline{X} \oplus \Sigma (\overline{\Omega I_X}) \arrow[u,phantom, sloped, "\in"] 
    \end{tikzcd}
    \end{equation}
    Here, $\Omega$ and $\Sigma$ are maps defined as in \cref{lem:GNP23.3.12} but each in the relevant categories for this statement to be well defined.
    If $\calC$ is reduced, then the resulting map
     \begin{equation}\label{eq:presilt_C_Lambda}
        \presilt(\calC) \to \presilt(\Lambda)
    \end{equation}
    is a bijection preserving the number of {indecomposable} direct summands. In particular, this bijection restricts to a bijection 
    \begin{equation}\label{eq:silt_C_Lambda}
        \silt(\calC) \to \silt(\Lambda).
    \end{equation} 
    \end{enumerate}
    \end{proposition}
    \begin{proof}
    The functor $\calC \to \overline{\calC}$ is seen to induce maps 
    \begin{equation*}
    \begin{tikzcd}[row sep=0.5em]
        \silt(\calC)\arrow[d,phantom, sloped, "\subseteq"] \arrow[r] & \silt(\overline{\calC})\arrow[d,phantom, sloped, "\subseteq"] \\
        \presilt(\calC) \arrow[r] & \presilt(\overline{\calC})   
    \end{tikzcd}
    \end{equation*}
    By \cite[{Prop.} 4.5]{PZ24}, the functor $\overline{\calC}(T,-)\colon \overline{\calC} \to \mods{\Lambda}$ induces a bijection 
    \begin{equation*}
    \begin{tikzcd}[row sep=0.5em]
        \silt(\overline{\calC})\arrow[d,phantom, sloped, "\subseteq"] \arrow[r] & \stautiltpair(\Lambda) \arrow[d,phantom, sloped, "\subseteq"] \\
        \presilt(\overline{\calC}) \arrow[r] & \taurigidpair(\Lambda) \\
        X \arrow[r,mapsto] \arrow[u,phantom, sloped, "\in"] & (\overline{\calC}(T,X),\overline{\calC}(T,{\Omega I_X}))\arrow[u,phantom, sloped, "\in"] 
    \end{tikzcd}
    \end{equation*}
    Composing the maps proves our first claim in \eqref{prop:PZ24.4.5}. Since all maps are bijections when $\calC$ is reduced, our second claim in \eqref{prop:PZ24.4.5} follows.

    By Adachi--Iyama--Reiten \cite[Lem. 3.4]{AIR2014}, the following map
     \begin{equation}\label{eq:presilt_C_Lambda2}
    \begin{tikzcd}[row sep=0.5em]
        \taurigidpair(\Lambda) \arrow[r] & \presilt(\Lambda) \\
        (M,P) \arrow[r,mapsto] \arrow[u,phantom, sloped, "\in"] & \mathbf{p}M\oplus {\Sigma P}\arrow[u,phantom, sloped, "\in"] 
    \end{tikzcd}
    \end{equation}
    is a bijection, and, by \cite[Thm. 3.2]{AIR2014}, restricts to a bijection 
    \begin{equation*}
    \begin{tikzcd}[row sep=0.5em]
        \stautiltpair(\Lambda) \arrow[r] & \silt(\Lambda){.}
    \end{tikzcd}
    \end{equation*}
    Composing the bijections in \eqref{eq:presilttaurigidbij} and \eqref{eq:presilt_C_Lambda2} yields the desired bijection in \eqref{eq:presilttaurigidbijF}. The remaining assertions in \eqref{thm:modulecatreduction_presilt} now follow from \eqref{prop:PZ24.4.5}. 
    \end{proof}

We remark that the bijection between silting objects and support $\tau$-tilting pairs can be upgraded to a poset isomorphism between basic silting objects and support $\tau$-tilting pairs with both constituent modules being basic. The poset structures are well-known and not important for the discussions in this article, which is why we choose to omit their definitions. 

We adopt the following terminology from the area of $\tau$-tilting theory, where the following notion is also often referred to as $\tau$-tilting finiteness.

\begin{definition}
    Let $(\calC,\E)$ be as in \Cref{setting}. Then $(\calC,\E)$ is \textit{silting-finite} if it contains only finitely many basic silting objects up to isomorphism. Using \Cref{defprop:Bongartz}, one can show that it is equivalent to have only finitely many basic (or indeed indecomposable) presilting objects.
\end{definition}

For certain ideal quotients of extriangulated categories, there is a close relationship between (pre)silting objects.

\begin{proposition}[see {\cite[Cor. 3.3 and Prop. 3.5]{FGPPP23}}]\label{prop:FGPPP23}
Let $\calC$ be a 0-Auslander extriangulated $k$-category. Let $\mathcal{J}$ be an ideal of $\calC$ generated by morphisms with injective domain and projective codomain. 
    \begin{enumerate}
        \item\label{prop:FGPPP23.3.3}
          The ideal quotient $\calC/\mathcal{J}$ is a 0-Auslander extriangulated $k$-category, and the natural functor ${\calC} \to \calC/\mathcal{J}$ carries the structure of an extriangulated functor. 
        \item\label{prop:FGPPP23.3.5} The natural functor 
    \begin{equation}\label{eq:toJ}
        {\calC} \to \calC/\mathcal{J}
    \end{equation}
    induces a map
    \begin{equation*}
        \presilt({\calC}) \xrightarrow{} \presilt(\calC/\mathcal{J}),
    \end{equation*}
    which restricts to a bijection
    \begin{equation*}
        \silt({\calC}) \xrightarrow{} \silt(\calC/\mathcal{J}).
    \end{equation*}
    \end{enumerate}
\end{proposition}

In \eqref{prop:FGPPP23.3.3} above, we mean that the functor $\calC \to \calC/\mathcal{J}$ induces an extriangulated structure on $\calC/\mathcal{J}$.

\begin{example}
    Let $A$ be a non-positive dg $k$-algebra such that $H^0A$ is finite-dimensional, and let $\mathcal{J}$ denote the ideal of $\pertwo(A)$ generated by \textit{all} morphisms with injective domain and projective codomain. Since $\pertwo(A)$ is (reduced) 0-Auslander, it follows from \Cref{prop:FGPPP23}\eqref{prop:FGPPP23.3.3} that the ideal quotient $\pertwo(A)/\mathcal{J}$ is 0-Auslander, and one sees that it is also reduced. By results of Brüstle--Yang \cite[Appendix A]{BY14}, the ideal quotient $\pertwo(A)/\mathcal{J}$ is equivalent to $\pertwo(H^0A)$ (and in turn to $\Ktwo(\proj H^0A)$), and one can identify the natural functor $\calC\to\calC/\mathcal{J}$ with the derived tensor functor
    \begin{equation*}
    \begin{tikzcd}[column sep=5em]
        \pertwo(A) \arrow[r,"{-\derotimes{A}H^0A}"] & \pertwo(H^0A),
    \end{tikzcd}
    \end{equation*}
    where we regard $H^0A$ as an $A$-$H^0A$-dg-bimodule. This induces a bijection
    \begin{equation*}
        {\presilt(\pertwo(A)) \to \presilt(\pertwo(H^0A))}
    \end{equation*}
    which restricts to a bijection
        \begin{equation*}
       {\silt(\pertwo(A)) \to \silt(\pertwo(H^0A))}.
    \end{equation*}
\end{example}

\subsection{Silting reduction}

To formulate silting reduction for 0-Auslander extriangulated categories, following the first-named author \cite{Bor24}, we will first have to consider the general setup for localisation of extriangulated categories \cite[§4.1]{NOS22}. 
For a thick subcategory $\calN$ of a extriangulated category $\calC$, let $\mathit{Inf}_{\calN}$ denote the set of inflations in $\calC$ with cone in $\calN$, and dually let $\mathit{Def}_{\calN}$ denote
the set of deflations in $\calC$ with cocone in $\calN$. Let $\mathcal{S}_{\calN}$ be the subcategory
of $\calC$ generated by $\mathit{Inf}_{\calN}$ and $\mathit{Def}_{\calN}$. The localisation functor of $\calC$ with respect to $\calN$ {is defined} as the Gabriel--Zisman localisation {functor} \cite{GZ12} of $\calC$ with respect to $\mathcal{S}_{\calN}$. We denote the codomain of the localisation functor by $\calC/\calN$. It will also be referred to as the localisation of $\calC$ with respect to $\calN$.
The localisation of an extriangulated category does not in general admit an extriangulation in a natural way. This technical difficulty will not be a concern for our purposes, as we will see in \Cref{prop:siltred}\eqref{prop:siltred_equiv} below.

\begin{lemma}\label{lem:Ver2.3.1}
    Let $\calC$ be an extriangulated category, and let $\calN$ be a thick subcategory of $\calC$ such that $\calC/\calN$ is extriangulated and the localisation functor $L\colon\calC \to \calC/\calN$ is an extriangulated functor. 
    The functor $L$ then induces an isomorphism from the first to the second of the following posets:
    \begin{enumerate}
        \item thick subcategories of $\calC$ containing $\calN$,
        \item thick subcategories of $\calC/\calN$.
    \end{enumerate}
    The inverse is given by forming preimages along $L$.
\end{lemma}
\begin{proof}
    Let $\calE$ be a thick subcategory of $\calC$ containing $\calN$ and let $\calF$ be a thick subcategory of $\calC/\calN$ this makes sense since we assume $\calC/\calN$ carries the structure of an extriangulated category. It is clear that $L$ induces a well-defined inclusion-preserving map $\mathcal{E}\mapsto L\mathcal{E}$, and that the map $\mathcal{F}\mapsto L^{-1}\mathcal{F}$ defines an inclusion-preserving map in the other direction. It is also clear that $L(L^{-1}\mathcal{F})=\mathcal{F}$, whence the map $\mathcal{E}\mapsto L\mathcal{E}$ surjects. 
    We complete the proof by showing that {$L^{-1}(L \mathcal \calE)$} {$= \calE$}. This will be achieved by showing the following {equivalent reformulation}: any object {of} $\calC$ which is isomorphic to an object {$X_\ell$} in ${L} \mathcal{E}$ when considered as an object {of} $\calC/\calN$ is in $\mathcal{E}$. We adapt Verdier's proof for the special case of triangulated categories \cite[{{Prop.} 2.3.1(d)$^{\mathrm{bis}}$}]{Ver96} (see also \cite[{Lem.} 3.1]{Tak13}).
    Let $X_0$ be an object {of} $\calC$ which is isomorphic to an object in ${L}\mathcal{E}$ when considered as an object {of} $\calC/\calN$. We show that $X_0$ is in $\mathcal{E}$. By assumption, there exists a zigzag diagram
    \begin{equation*}
        \begin{tikzcd}[row sep=2em,column sep=1.75em]
X_0 \arrow[r, "q_{0.5}"] & X_{0.5} & X_1 \arrow[l, "q_1"'] \arrow[r, "q_{1.5}"] & X_{1.5} & X_{2} \arrow[l, "q_2"'] \arrow[r, "q_{2.5}"] & \cdots \arrow[r, "q_{(\ell-2).5}"] & X_{(\ell-2).5} & X_{\ell-1} \arrow[r, "q_{(\ell-1).5}"] \arrow[l, "q_{\ell-1}"'] & X_{(\ell-1).5}                                                                                       & X_\ell \in \mathcal{E}, \arrow[l, "q_{\ell}"']                 
\end{tikzcd}
    \end{equation*}
   where $X_{\ell}\in \mathcal{E}$, and the morphisms are in $\mathcal{S}_{\calN}$. In the case where $q_{\ell}$ is either in $\mathit{Inf}_{\calN}$ or in $\mathit{Def}_{\calN}$, one uses that $X_{\ell}\in \mathcal{E}$ and the fact that $\mathcal{E}$ is a thick subcategory containing $\calN$ to show that $X_{(\ell-1).5}\in {\calE}$. In general $q_{\ell}$ is a composite of morphisms in either $\mathit{Inf}_{\calN}$ or $\mathit{Def}_{\calN}$ but not both. One uses (\textbf{ET4}) {or (\textbf{ET4}${}^\textnormal{op}$)} iteratively to show that $X_{(\ell-1).5}\in \mathcal{E}$ still holds. Arguing similarly for the morphisms $q_{(\ell-1).5}, q_{\ell-1},\dots$, and $q_{0.5}$, in order, we conclude that $X_0$ is in $\mathcal{E}$, as desired.
\end{proof}

\begin{example}\label{eg:Ogawa}
    Let $\calC$ be an extriangulated $k$-category, and let $\mathcal{X}$ be an additive subcategory of $\calC$ spanned by projective-injective objects. Because all extensions in $\mathcal{X}$ are {split} and $\mathcal{X}$ is closed under direct summands, it is clear that $\mathcal{X}$ is a thick subcategory of $\calC$ and that the ideal quotient $\calC/[\mathcal{X}]$ is extriangulated \cite[{Prop.} 3.30]{NP19}. Moreover, the localisation $\calC/\mathcal{X}$ coincides with the ideal quotient $\calC/[\mathcal{X}]$ (see {\cite[Example 2.6]{Oga22}}), whence the localisation functor $\calC\to\calC/\mathcal{X}$ is extriangulated.
    By \Cref{lem:Ver2.3.1}, the localisation functor $\calC\to \calC/[\mathcal{X}]$ induces an isomorphism of posets between:
    \begin{enumerate}
        \item thick subcategories of $\calC$ containing $\mathcal{X}$;
        \item thick subcategories of $\calC/[\mathcal{X}]$.
    \end{enumerate}
\end{example}

The next result is formulated in less generality than in the cited reference. In general, the result holds for presilting subcategories satisfying a technical condition called (\textbf{gCTCP}), which holds whenever they have Bongartz completions \cite[Thm. 5.9(iii)]{Bor24}. In particular, it holds in \Cref{setting} because then every presilting subcategory admits a Bongartz completion, defined as the additive closure of the Bongartz completion of the largest basic presilting object in the subcategory via \cref{defprop:Bongartz}.

For an extriangulated $k$-category $(\calC,\E)$ and a full subcategory $\mathcal{X}$ of $\calC$, we define the \textit{right orthogonal subcategory} of $\mathcal{X}$ as follows:
\begin{equation*}
	\mathcal{X}^{\perp_{\E}} \coloneqq \{Y\in \calC \sth \E(X,Y)=0 \quad \forall X\in \mathcal{X} \},
\end{equation*}
and the \textit{left orthogonal subcategory} of $\mathcal{X}$ in a dual manner:
\begin{equation*}
	{^{\perp_{\E}}\mathcal{X}} \coloneqq \{Y\in \calC \sth \E(Y,X)=0 \quad \forall X\in \mathcal{X} \}.
\end{equation*}
Given an object $X\in \calC$, we write ${X}^{\perp_{\E}}$ for the right orthogonal subcategory of the full subcategory {$\{ X\}$} of $\calC$, and ${^{\perp_{\E}}{X}}$ for its {analogously defined} left orthogonal subcategory.

\begin{proposition}[{\cite[§5]{Bor24}}]\label{prop:siltred}
    Let $(\calC,\E)$ be as in \Cref{setting}.
    Let $U$ be a presilting object in $\calC$, and let $\calZ_U \coloneqq U^{\perp_{\E}}\cap {^{\perp_{\E}}U}${, which inherits an extriangulated structure from $\calC$}.
    \begin{enumerate}
        \item\label{prop:siltred_0Aus} The extriangulated $k$-categories $\calZ_U$ and $\calZ_U/[U]$ are 0-Auslander, and remain Hom-finite and Krull--Schmidt. 
        \item\label{prop:siltred_0Aus_proj} The Bongartz completion $U^+$ is a basic projective silting object in both $\calZ_U$ and $\calZ_U/[U]$. Consequently, the 0-Auslander extriangulated $k$-categories $\calZ_U$ and $\calZ_U/[U]$ satisfy \Cref{setting}.
        \item\label{prop:siltred_equiv} The localisation functor $\calC \to \calC/\thick(U)$ induces an extriangulated equivalence 
    \begin{equation}
       \calZ_U/[U] \to \calC/\thick(U).
    \end{equation}
    Consequently, the localisation $\calC/\thick(U)$ is extriangulated and the localisation functor $\calC \to \calC/\thick(U)$ is an extriangulated functor.
    \item\label{prop:siltred_bij} The natural functor $\calZ_U \to \calZ_U/[U]$
    (or indeed the localisation functor $\calC \to \calC/\thick(U)$) induces a bijection
    \begin{equation*}
        \presilt_U(\calC) \to \presilt(\calZ_U/[U]) \simeq \presilt(\calC/\thick(U)),
    \end{equation*}
    which restricts to a bijection
    \begin{equation*}
        \silt_U(\calC) \to \silt(\calC/\thick(U)).
    \end{equation*}
    \end{enumerate}
\end{proposition}

This result generalises various reduction theorems in the literature, including silting reduction \cite{AI12,IY18} in triangulated categories. It is also closely related to $\tau$-tilting reduction in module categories \cite{Jas15}, as will be precisely addressed in \Cref{prop:Monica}. We include a useful lemma for later.

\begin{lemma}\label{lem:endoinred} 
	Let $\calC$, $T$, and $\Lambda$ be as in \Cref{setting}, let $U$ be a presilting object in $\calC$, let $U^+$ be the Bongartz completion of $U$, and consider the functor $\overline{(-)}\colon \calC \longrightarrow \mods{(}\Lambda{)}$ defined in \Cref{lem:GNP23.3.11}. Then
	\[ \mods( \End_{\calC/\thick(U)}(U^+)/[V]) \simeq \mods(\End_{\Lambda}(\overline{U^+})[\overline{U}]), \]
	where $V$ is the maximal basic injective object in $\calC/\thick(U)$. 
\end{lemma}
\begin{proof}
	Since \Cref{prop:siltred}\eqref{prop:siltred_equiv} gives an extriangulated equivalence ${\calZ_U}/[U] \rightarrow \calC/\thick(U)$ sending $U^+$ to $U^+$, we obtain a $k$-algebra isomorphism
	\[\End_{\calC/\thick(U)}(U^+)/[V] \simeq \End_{\calZ_U/[U]}(U^+)/[V'], \]
	where $V'$ is the maximal basic injective object in $\calZ_U/[U]$. 
	 It induces an exact equivalence
	\[ \mods( \End_{\calC/\thick(U)}(U^+)/[V]) \simeq \mods(\End_{\calZ_U/[U]}(U^+)/[V']). \]
	Applying \Cref{lem:GNP23.3.11} to $\calZ_U/[U]$ yields the first of the following equivalence of categories:	
	\[\mods(\End_{\calZ_U/[U]}(U^+)/[V']) \simeq (\calZ_U/[U])/[V'] \simeq \calZ_U/[U\oplus V'].   \] 
	From the fact that $\E_{\calZ_U/[U]}(-,X)=0$ if and only if $\E_{\calZ_U}(-, X)=0$ for $X \in \calZ_U$ by \cite[Prop. 2.9(2)]{GNP23}, we deduce that $U \oplus V'$ is the maximal basic injective object in $\calZ_U$. Since $N$ is also injective in $\calZ_U$, we obtain $N \in \add (U \oplus V')$. We get equivalences
	\[ \calZ_U/[U\oplus V'] \simeq \calZ_U/[N \oplus U\oplus V'] = (\calZ_U/[N \oplus U])/[V']. \] 
	From the above discussion we know that $\add(V')$ is the full subcategory of $\calZ_U/[N \oplus U]$ containing all injectives. The subcategory $\overline{\calZ_U} \coloneqq U^{\perp_{\overline{\E}}} \cap {^{\perp_{\overline{\E}}}U} \subseteq \calC/[N]$ may naturally be identified with $\calZ_U/[N]$. Consequently, $\overline{\calZ_U}/[U]$ may naturally be identified with $\calZ_U/[U \oplus N]$. The desired equivalence
	\[ (\calZ_{U}/[U \oplus N])/[V']  \simeq (\overline{\calZ_U}/[U])/[V'] \simeq \mods(\End_{\Lambda}({\overline{U^+}})/[\overline{U}])\] 
follows from \cite[Above Thm. 4.22 on p. 18]{PZ24}, since $\overline{\calZ_U}/[U]$ is the subcategory of objects compatible with $U$ in the reduced category $\overline{\calC}$.
\end{proof}

\subsection{Polyhedral fans}
Polyhedral fans are ubiquitous mathematical objects arising in research areas ranging from toric geometry to optimisation theory. Our focus, however, lies on polyhedral fans arising within the context of representation theory. We begin by recalling some definitions.

\begin{definition}[{see \cite{Zie95},\cite{CLS24}}]\label{def:cones+fans}
    Let $n$ be a non-negative integer, let $N$ be a lattice (i.e. free abelian group) of rank $n$ {(so isomorphic to $\Z^n$)} and embed $N$ into the $n$-dimensional $\R$-vector space $N_{\R}\coloneqq N\otimes_{\Z}\R$.
    \begin{enumerate}
        \item Given a subset $X=\{v_1 , \dots , v_{\ell}\}\subseteq N_{\R}$, its {non-negative} span $\boldC_{X}\coloneqq \spann_{\geq 0}(X) \subseteq N_{\R}$ is called a \textit{(polyhedral) cone} in $N_{\R}$. {The \textit{dimension} of $\boldC_{X}$ is the dimension of {the subspace} $\spann(X)$ of $N_{\R}$.}
        \item If each $v_i\in X$ is of the form $n_i\otimes 1$ for some $n_i\in N$, we say that the cone $\boldC_{X}$ is \textit{rational}.
        \item A cone is \textit{strongly convex} if it contains no positive-dimensional subspace of $N_{\R}$.
        \item A strongly convex cone is \textit{simplicial} if {it can be expressed by} $\boldC_{X}$, where $X$ is a linearly independent subset of $N_{\R}$. 
        \item Let $M$ denote the dual lattice of $N$ and let $\langle -, - \rangle\colon M_{\R}\times N_{\R} \to \R$ denote the pairing of the dual $\R$-vector spaces $M_{\R}$ and $N_{\R}$. For a cone $\sigma$ in $N_{\R}$, we define the \textit{dual cone} of $\sigma$ by 
        \begin{equation*}
            \sigma^{\vee} \coloneqq \{m \in M_{\R} \sth \langle m, u \rangle \geq 0 \quad \forall u\in\sigma \}.
        \end{equation*}
        For $m\in \sigma^{\vee}$, one considers the hyperplane
        \begin{equation*}
            H_m \coloneqq \{n \in N_{\R} \sth \langle m, n \rangle = 0 \}.
        \end{equation*}
        A \textit{face} of a {cone} $\sigma$ {in $N_{\R}$} is a subset of $N_{\R}$ of the form $\sigma\cap H_m$ for some $m\in \sigma^{\vee}$. A face of a cone is also a cone. In the special case where $\sigma$ is simplicial, there exists a linearly independent subset $X$ of $N_{\R}$ such that $\sigma=\boldC_{X}$, and a face of $\sigma$ is then of the form $\boldC_{Y}$, where $Y$ is a subset of $X$.
        \item\label{def:cones+fans_fan}  
        A \textit{polyhedral fan} in $N_{\R}$ is given by a family of {strongly-convex} polyhedral cones in $N_{\R}$ denoted $\Sfan = \{\boldC_{X}\}_{X\in I_{\Sfan}}$ for which
\begin{enumerate}
    \item $\Sfan$ is closed under the formation of faces, 
    \item the intersection of two cones $\boldC_{X}$ and $\boldC_{Y}$ {is a face of both $\boldC_{X}$ and $\boldC_{Y}$.}
\end{enumerate} 
    \item {Let $\Sfan$ be a polyhedral fan. A cone $\sigma\in\Sfan$ is \textit{maximal} if it is not properly contained in any other cone in $\Sfan$.}
    \item\label{def:cones+fans_sr}  A polyhedral fan $\Sfan$ is \textit{simplicial} if all of its constituent cones are simplicial, and \textit{rational} if all of {its} constituent cones are rational. 
    \item A polyhedral fan $\Sfan$ is \textit{finite} if $|\Sfan|$ is finite, and \textit{complete} if $\bigcup\limits_{\boldC_{X}\in\Sfan} \boldC_{X} = N_{\R}$.
    \end{enumerate}
\end{definition}

We will display several concrete examples of polyhedral fans in \Cref{ex:fans} below. For now, let us consider some general constructions.

\begin{definition}\label{def:morfan}
    Let $N_i$ be a lattice of rank $n_i$ and let $\Sfrak_i$ be a polyhedral fan in $(N_i)_{\R}$ for $i=1,2$. A \textit{morphism of fans} between $\Sfrak_1$ and $\Sfrak_2$ is given by an $\R$-linear transformation $f\colon (N_1)_{\R} \rightarrow (N_2)_{\R} $ such that for all cones $\sigma_{1}\in\Sfrak_{1}$ there exists a cone $\sigma_2\in \Sfrak_2$ {with} $f(\sigma_1)\subseteq \sigma_2$. An \textit{isomorphism of fans} is a morphism of fans admitting an inverse morphism of fans.
\end{definition}

We sometimes write $f: ((N_1)_\R, \Sfan_1) \to ((N_2)_\R, \Sfan_2)$ to denote a morphism of fans as in \cref{def:morfan}.
The following examples of morphisms of fans are frequently referred to in the remaining text.

\begin{example}\label{eg:orthproj}
    Let $\Sfrak$ be a polyhedral fan in $N_{\R}$ and let $\sigma$ be a cone in $\Sfan$.
    \begin{enumerate}
        \item\label{eg:orthproj1} Denote the set of cones containing $\sigma$ by
        \begin{equation}\label{eq:star}
 	 \starr(\sigma) \coloneqq \{\tau\in\Sfan \sth \sigma\subseteq \tau \}.
        \end{equation}
        Note that taking the closure {of $\mathrm{star}(\sigma)$} under faces defines a polyhedral fan in $N_{\R}$. We set
        \[ \overline{\starr}(\sigma) \coloneqq \{ \kappa \in \Sfan \sth \kappa \subseteq \tau, \text{ with } {\sigma \subseteq \tau} \}. \]
        
        Clearly, the identity endomorphism on $N_{\R}$ gives a morphism of fans from $\overline{\mathrm{star}}(\sigma)$ to $\Sfan$.  
        \item\label{eg:orthproj_proj} 
        Define $\sigma^{\perp} \coloneqq \{v\in M_{\R} \sth \langle v,u\rangle = 0 \quad \forall u\in {\spann(}\sigma{)} \}$.
        This is an $\R$-linear subspace of $M_{\R}$, which is the dual vector space of $N_{\R}$.
        Suppose that $\sigma$ is rational and simplicial, and that we have fixed a linearly independent subset $X\subseteq N{_{\R}}$ such that $\sigma = \boldC_X$. We also fix a basis $B$ of $N{_{\R}}$ containing $X$. Consider the composite $\R$-linear transformation $\pi_{\sigma}$ defined by
        \begin{equation}\label{eq:othcomphomomo}
            \begin{tikzcd}[column sep=4em]
            N_{\R} \arrow[r,"\varphi_B"] \arrow[rr,"\pi_{\sigma}", bend right] & M_{\R} \arrow[r,"\mathrm{proj}_{\sigma}"] & \sigma^{\perp},
            \end{tikzcd}
        \end{equation}
        where the first homomorphism is the isomorphism sending $B$ to the dual basis of $M_{\R}$, and the second is orthogonal projection. Note that $\mathrm{proj}_{\sigma}$ projects $M$ onto a lattice $L$ embedded in $\sigma^{\perp}$, and that $L_{\R}$ is isomorphic to $\sigma^{\perp}$. 
        The images of the cones in $\mathrm{star}(\sigma) \subseteq \Sfan$ under $\pi_{\sigma}$ define a polyhedral fan $\Sfan_\sigma = {\pi_{\sigma}(\starr(\sigma))}$ in $\sigma^\perp${, see \cite[p. 52]{Ful93}.} Therefore, the map $\pi_{\sigma}$ defines a morphism of fans from $\overline{\mathrm{star}}(\sigma)$ to $\Sfan_\sigma$.
    \end{enumerate}
\end{example}

In subsequent sections, our examples of polyhedral fans will have a Grothendieck group as the ambient lattice, whence it is useful to define this notion before commencing the next section.

\begin{definition}
For an extriangulated $\k$-category $(\calC,\E,\mathfrak{s})$, let $\mathrm{iso}(\calC)$ denote the set of isomorphism classes of objects in $\calC$. The \textit{split Grothendieck group} of $(\calC,\E,\mathfrak{s})$ is defined by
$\Ksp_0(\calC)\coloneqq \Z^{\mathrm{iso}(\calC)}$, namely the free abelian group generated by $\mathrm{iso}(\calC)$. The \textit{Grothendieck group} of $(\calC,\E,\mathfrak{s})$ is defined as \cite{ZZ21}:
\begin{equation*}
	\K_0(\calC) \coloneqq \Ksp_0(\calC) / \big\langle[B] - [A] - [C] \sth
	 \begin{tikzcd}
			A \arrow[r,tail] & B \arrow[r,two heads] & C \arrow[r,dashed] & {}
	\end{tikzcd}
	\text{ is an {$\E$-triangle}} \big\rangle.
\end{equation*}
The \textit{real split Grothendieck group} of $\calC$ is defined to be $\Ksp_0(\calC)_{\R}\coloneqq \Ksp_0(\calC)\otimes_\Z \R$, and the \textit{real Grothendieck group} of $\calC$ by $\K_0(\calC)_{\R}\coloneqq \K_0(\calC)\otimes_\Z \R$. \end{definition}

If $\calC$ is Krull--Schmidt and contains a silting object $T$, we may assume it to be basic with $n$ indecomposable direct summands $T_i$ up to isomorphism for some non-negative integer $n$. The isomorphism classes $[T_i]$ form a basis of $\K_0(\calC)$, whence $\K_0(\calC)$ is a lattice of rank $n$ \cite[{Thm.} 4.1]{AT23}. The assumption that $\calC$ is Krull--Schmidt has later been shown to be unnecessary \cite[Thm. A]{Wan25}. Thus, we have that the real Grothendieck group $\K_0(\calC)_{\R}$ is an $n$-dimensional $\R$-vector space in this case.

\section{\texorpdfstring{$\bfg$}{g}-vector fans of 0-Auslander extriangulated categories}\label{sec:fan}

Let $\calC$ be a Hom-finite Krull--Schmidt 0-Auslander extriangulated $k$-category and let $T$ be a basic projective silting object in $\calC$, as in \Cref{setting}. For an object $X$ in $\calC$ we fix {an $\E$-triangle}
\begin{equation}\label{eq:Xconf}
        \begin{tikzcd}
            T_1 \arrow[r, tail] & T_0 \arrow[r,two heads] & X \arrow[r, dashed] & {},
        \end{tikzcd}
\end{equation}
with $T_0,T_1\in\add(T)$ using the defining properties of 0-Auslander extriangulated categories. Define the \textit{index} of $X$ with respect to $T$ by $\ind_T(X) \coloneqq [T_0]-[T_1] \in \K_0(\add{T})$, see \cite[§4.3]{GNP23}. The index is well-defined; the choice of $\E$-triangle in \eqref{eq:Xconf} does not affect the element $\ind_T(X)$ in $\K_0(\add_{{\calC}}{T})$. Note that $\K_0(\add_{\calC}{T})=\Ksp_0(\add{T})$ since $T$ is projective in $\calC$, {and ${\K}_0(\add_{{\calC}} T) = \K_0(\calC)$ because $T$ is silting \cite[Thm. 4.1]{AT23}}.

\begin{definition}\label{def:gfan}
{Let $\calC$ and $T$ be as in \cref{setting}. Let $X \in \calC$ be presilting.}
    \begin{enumerate}
        \item The \textit{$\bfg$-vector} of $X$ (with respect to $T$) is then defined as the element $\bfg^X \coloneqq \mathrm{ind}_T(X) \otimes 1$ in the real Grothendieck group $ \K_0(\add_{{\calC}}{T})_{\R}$. 
        \item If $X$ is basic and has a decomposition $X \simeq X_1 \oplus \dots \oplus X_\ell$, where each $X_i$ is indecomposable, we define the \textit{$\bfg$-vector cone} $\boldC(X)$ of $X$ as the non-negative span in $\K_0(\add_{\calC}{T})_{\R}$ of the set $\{\bfg^{X_i}\}_{i=1}^{\ell}$.
        \item We define $\partfan{\calC}$ to be the collection of cones 
    $\partfan{\calC} \coloneqq \{ \boldC(X) : X \in \presilt(\calC)\} $
    in ${\K}_0(\add T)_\R$.
    \end{enumerate}
\end{definition}

Let $R$ be an arbitrary silting object in $\calC$. Since $\calC$ is 0-Auslander, the object $R$ is also \textit{tilting} \cite[{Thm.} 4.3]{GNP23}, which is to say that it is presilting and that every $V\in\add(T)$ admits {an $\E$-triangle}
\begin{equation*}
        \begin{tikzcd}
            V \arrow[r,tail] & R^0 \arrow[r,two heads] & R^1 \arrow[r, dashed] & {},
        \end{tikzcd}
\end{equation*}
where $R^0,R^1\in \add(R)$ {by} \cite[{Thm.} 4.3]{GNP23}. {Using this}, one formulates a well-defined notion of \textit{co-index} of $V$ with respect to $R$ by defining $\coind_R(V) \coloneqq [R^0]-[R^1]\in \Ksp_0(\add{R})${, see} \cite[§4.3]{GNP23} \cite[Lem. 4.36]{PPPP23}.

We now establish that certain functors interact well with the index map. The following two lemmas play central roles in many of the constructions throughout.
\begin{lemma} \label{lem:indexFanIso2} Let $\calC$ and $T$ be as in \Cref{setting}, and let $\calD$ and $S$ satisfy the same setup. If $G: \calC \to \calD$ is an extriangulated functor such that
    \begin{enumerate}
        \item\label{lem:indexIso2a} $G$ sends presilting objects in $\calC$ to presilting objects in $\calD$; and
        \item\label{lem:indexIso2b} $G(T) \in \add(S)$;
    \end{enumerate}
    then $G$ induces a commutative diagram 
     \begin{equation}\label{eq:index_commsq}
    \begin{tikzcd}[column sep=6em]
        \presilt({\calC}) \arrow[d,"\ind_T "] \arrow[r, "G"] & \presilt({\calD}) \arrow[d,"\ind_S"] \\
        \K_0(\add_{\calC}T) \arrow[r,"g"] & \K_0(\add_{\calD}S)
    \end{tikzcd}
    \end{equation}
    where $g$ is defined by its action on the basis consisting of indecomposable direct summands $T_i$ of $T$ as $g([T_i]) = [G(T_i)]$. 
\end{lemma}
\begin{proof}
The conditions we impose on $G$ guarantee that the functor $G$ induces a diagram as in \eqref{eq:index_commsq}; indeed condition \eqref{lem:indexIso2a} yields a well-defined map in the upper row and condition \eqref{lem:indexIso2b} in the lower row.
    It remains to show commutativity, that is, for a presilting object $X \in \presilt \calC$,
    showing $g(\ind_T(X)) = \ind_S(G(X))$. Let 
    \begin{equation}\label{eq:indexIsoTriangle} \begin{tikzcd}
            T_1 \arrow[r, tail] & T_0 \arrow[r,two heads] & X \arrow[r, dashed] & {}
        \end{tikzcd} \end{equation}
        be an $\E$-triangle in $\calC$ with $T_0$ and $T_1$ in $\add_{\calC} T$. Then the index of $X$ is given by $[T_0]- [T_1] \in \K_0(\calC)$. Because $G$ is an extriangulated functor, {the} $\E$-triangle {\cref{eq:indexIsoTriangle}} is sent to an $\E$-triangle in $\calD$ of the form
        \[ \begin{tikzcd}
            G(T_1) \arrow[r, tail] & G(T_0) \arrow[r,two heads] & G(X) \arrow[r, dashed] & {}
        \end{tikzcd}\]
        Since $G(T) \in \add(S)$ by assumption, 
        we have $G(T_0), G(T_1) \in \add(S)$.
        This gives $\ind_S(G(X)) = [G(T_0)]-[G(T_1)]$ (see \cite[§4.3]{GNP23} about the independence of the chosen $\E$-triangle in $\calD$ for the index). The result now follows from the definition of the map $g$.
        \end{proof}
        
        Even though there is no canonical extriangulated functor from $\calC$ to $\Ktwo(\proj \Lambda)$, the bijection of presilting objects in \cref{prop:F}\eqref{thm:modulecatreduction_presilt} fits into a similar commutative diagram.

\begin{lemma}\label{lem:Findexsquare}
Let $\calC$, $\Lambda$ and $T$ be as in \Cref{setting}, and suppose that $\calC$ is reduced. The following square commutes and the horizontal maps are bijections:\begin{equation}\label{eq:modulecatreduction_commsq}
    \begin{tikzcd}[column sep=6em]
        \presilt(\calC) \arrow[d," \ind_T"] \arrow[r,"{\eqref{eq:presilt_C_Lambda}}"] & {\presilt( \Lambda)} \arrow[d,"\ind_{\Lambda} "] \\
        \K_0(\add{T})  \arrow[r,"{\calC(T,-)}"] & \K_0({\add}{\Lambda})
    \end{tikzcd}
    \end{equation}
    \end{lemma}
    \begin{proof}
        Let $X$ be a presilting object in $\calC$, let $I_X$ be the maximal injective direct summand of $X$. We then have an $\E$-triangle
    \begin{equation*}
        \begin{tikzcd}[column sep=5em,ampersand replacement=\&]
           \Omega I_X \oplus T'_1 \arrow[r,"{\begin{pmatrix}
               0 & x'
           \end{pmatrix}}", tail] \& T_0 \arrow[r,two heads] \& X \arrow[r, dashed] \& {}
        \end{tikzcd}
    \end{equation*}
    {with ${T_0, T_1'} \in \add(T)$.}
    We denote $P_0 = \calC(T,T_0)$, $P'_1 = \calC(T,T'_1)$, and $P_{\Omega} = \calC(T,\Omega I_X)$, all being regarded as objects in $\proj(\Lambda){= \add \Lambda}$. The image of $X$ under the composite
    \begin{equation*}
        \begin{tikzcd}[column sep=5em]
            \presilt(\calC)\arrow[r,"\ind_T"] & \K_0(\add{T}) \arrow[r,"{\calC(T,-)}"] & \K_0(\add{\Lambda})
        \end{tikzcd}
    \end{equation*}
    is then seen to be $[P_0]-[P_1']-[P_{\Omega}]$, and this is also the image of $X$ under the other composite in \eqref{eq:modulecatreduction_commsq}. Thus, the square commutes. We have established in {\Cref{prop:F}\eqref{thm:modulecatreduction_presilt}} that the horizontal map along the top of \eqref{eq:modulecatreduction_commsq} is a bijection. The fact that the bottom horizontal map is a bijection follows from the fact that both categories have silting objects with the same number of indecomposable direct summands, see \cite[Thm. 4.1]{AT23} and \cref{lem:GNP23.3.11}.
    \end{proof}

We apply this result to extend an important result on $\bfg$-vectors \cite{DK08, DIJ2019} of finite-dimensional $k$-algebras to the setting of 0-Auslander extriangulated categories. 

\begin{lemma}\label{lem:IsFan}
    Let $\calC$ and $T$ be as in \Cref{setting}.
    \begin{enumerate}
        \item\label{lem:IsFan_GNP} Let $R$ be a silting object.
        The $\Z$-linear maps
        \begin{equation*}
            \begin{tikzcd}
        \K_0(\add{T}) \arrow[rr, "\coind_R"', bend right=10] &  & \Ksp_0(\add{R}) \arrow[ll, "\ind_T"', bend right=10]
        \end{tikzcd} 
        \end{equation*}
        are mutually inverse isomorphisms of abelian groups. Consequently, the induced $\R$-linear maps
        \begin{equation}\label{eq:g_coind}
            \begin{tikzcd}
        \K_0(\add{T})_{\R} \arrow[rr, "\coind_R{(-)} \otimes {1}"', bend right=10] &  & \Ksp_0(\add{R})_{\R} \arrow[ll, "\bfg"', bend right=10]
        \end{tikzcd} 
        \end{equation}
        are mutually inverse isomorphisms of $\R$-vector spaces.
        \item\label{lem:IsFan_linindep} Let $U$ be a basic presilting object in $\calC$ and let $\{U_1,\dots,U_{\ell}\}$ be the set of the indecomposable direct summands of $U$. Then the set 
    $\{ \bfg^{U_1}, \dots, \bfg^{U_\ell}\}$ of ${\K}_0(\add{T})_{\R}$ is linearly independent. This set forms a basis of ${\K}_0(\add{T})_{\R}$ if and only if $U$ is silting.
    \item\label{lem:IsFan_inj} {The index map induces an injection}
    \begin{equation}\label{eq:IsFan_inj_ind}
    \begin{tikzcd}[column sep=6em]
        \presilt(\calC) \arrow[r, "\ind_T"] & \K_0(\add{T}),
        \end{tikzcd}
    \end{equation}
    {whence taking $\bfg$-vectors yields and injection}
     \begin{equation*}
    \begin{tikzcd}[column sep=6em]
        \presilt(\calC) \arrow[r, "\bfg"] & \K_0(\add{T})_{\R}.
        \end{tikzcd}
    \end{equation*}
    \end{enumerate}
\end{lemma}
\begin{proof}
    The assertions in \eqref{lem:IsFan_GNP} follow directly from \cite[Prop. 4.10]{GNP23}. We will now deduce \eqref{lem:IsFan_linindep}.\footnote{It should be stressed that we are not proving an original result. Our assertions in \eqref{lem:IsFan_linindep} are indeed well-known and appear in various forms in the literature{, see} \cite[{Thm.} 1.2]{OS23} \cite[{Thm.} 4.1]{AT23} \cite[{Thm.} 2.27]{AI12}.} Let $U^+$ denote the Bongartz completion of $U$, as defined in \Cref{defprop:Bongartz}. Applying \eqref{lem:IsFan_GNP} to the silting object $U^+$, one deduces that the $\bfg$-vectors of the indecomposable direct summands of $U^+$ provide a basis for $\K_0(\add{T})_{\R}$. Since the $\R$-linear map $\bfg^{(-)}$ in \eqref{eq:g_coind} is an isomorphism, it sends the basis of $\Ksp_0(\add U^+)_{\R}$ given by {indecomposable direct summands of} $U^+$ to a basis of $\K_0(\add{T})_{\R}$ of the same cardinality. Our assertions in \eqref{lem:IsFan_linindep} easily follow.
    
     	To prove \eqref{lem:IsFan_inj}, we first reduce the problem to the case where $\calC$ is reduced. Let $N$ be a maximal basic projective-injective object in $\calC$. By \cref{prop:FGPPP23}, the functor $\calC \to \calC/[N]$ satisfies the assumptions of \cref{lem:indexFanIso2}, which means that it induces a commutative diagram as in \eqref{eq:index_commsq}. This restricts to a commutative diagram
	\begin{equation*}
    \begin{tikzcd}[column sep=6em]
        \presilt_N(\calC) \arrow[d," \ind_T"] \arrow[r] & {\presilt( \calC/[N])} \arrow[d,"\ind_{T} "] \\
        \K_0(\add_{\calC}{T})  \arrow[r] & \K_0(\add_{\calC/[N]} T)
    \end{tikzcd}
    \end{equation*}
    in which the top map is bijective by \cref{prop:siltred}\eqref{prop:siltred_bij}. Consequently, if 
    \[ \ind_T: \presilt(\calC/[N]) \to \K_0(\add_{\calC/[N]} T) \]
    is injective, then so is
    \[ \ind_T: \presilt_N(\calC) \to \K_0(\add_{\calC}{T}). \]
    Assume now that the index map is injective on presilting objects containing $N$ as a direct summand. To extend this to all presilting objects, take $X, Y \in \presilt(\calC)$ and assume $\ind_T(X) = \ind_T(Y)$. Because $N$ is projective-injective, the objects $X \oplus N$ and $Y \oplus N$ are again presilting. Since the index map is additive, it follows that $\ind_T(X \oplus N) = \ind_T(Y \oplus N)$. However, as these objects now contain $N$ as a direct summand and their index coincides, we deduce $X \oplus N \simeq Y \oplus N$ and in turn that $X \simeq Y$, as required.
    
	Thus, we may assume that $\calC$ is reduced. We now show that the reduced case follows from known result about $\bfg$-vectors of finite-dimensional algebras. From \cite[Cor. 6.7(a)]{DIJ2019}, we know that 
	\[ \ind_{\Lambda}: \presilt(\Lambda) \to \K_0(\add \Lambda) \]
	is injective. It follows directly from the commutativity of the square in \eqref{eq:modulecatreduction_commsq}, that
	\[ \ind_T: \presilt(\calC) \to \K_0(\add T) \]
	is injective, because the horizontal maps are bijective. This concludes the proof.
	\end{proof}
	
We are now able to establish the fan structure of $\partfan{\calC}$. 

\begin{theorem}\label{thm:IsFan}
    Let $\calC$ be as in \Cref{setting}. Then $\partfan{\calC}$ is a rational and simplicial polyhedral fan.
\end{theorem}
\begin{proof}
    We will prove that $\partfan{\calC}$ satisfies \Cref{def:cones+fans}\eqref{def:cones+fans_fan} and \eqref{def:cones+fans_sr}. It is clearly sufficient to restrict our attention to $\bfg$-vector cones of the form $\boldC(U)$, where $U$ is basic.
    
    It is immediate from \Cref{lem:IsFan} that each $\bfg$-vector cone $\boldC(U)$ is rational and simplicial. As a consequence, faces of $\boldC(U)$ are given by cones spanned by a subset of the generators of $\boldC(U)$, which correspond {to} direct summands of $U$. {Since direct summands of presilting objects are presilting, this implies} that faces of cones in $\partfan{\calC}$ are also cones in $\partfan{\calC}$.

    We now show that the intersection of two cones is a cone, and that it is a face of both. Thus, let $R$ and $S$ be two basic presilting objects in $\calC$ and let $U$ be their largest common direct summand. We need to show that $\boldC(U)= \boldC({R}) \cap \boldC({S})$. The inclusion $\subseteq$ trivially holds. To show that the reverse inclusion holds, we show that an arbitrary element $v \in \boldC({R}) \cap \boldC({S})$ is also contained in $\boldC(U)$. 
    We want to reduce the problem to considering rational points in the intersection. Here, an element of $\K_0(\add T)_\R$ is rational if it can be written as $\sum_{i=1}^n [T_i] \otimes q_i \in \K_0(\add T)_\R$ with $q_i \in \mathbb{Q}$. Write 
    \[ v = \sum_{i=1}^t  \ind_T({R_i}) \otimes \alpha_i = \sum_{j=1}^{s} \ind_T({S_j}) \otimes \beta_j \]
    with $0 \leq \alpha_i, \beta_j \in \R$. Then, there exists $0 \leq \epsilon_i \in \R$ for each $i=1, \dots, t$ and ${\epsilon} \in \K_0(\add T)_\R$ such that 
    \[ v - {\epsilon} =  \sum_{i=1}^t  \ind_T({R_i}) {\otimes (\alpha_i-\epsilon_i)}, \quad \text{and} \quad v + {\epsilon} =  \sum_{i=1}^t \ind_T({R_i}) {\otimes (\alpha_i+\epsilon_i)} \]
    are rational, because each $\ind_T(R_i) \otimes 1$ is also rational, and lie in $\boldC(R) \cap \boldC(S)$. Since $\boldC(R)$, $\boldC(S)$ and $\boldC(U)$ are convex, if the rational points $v-\epsilon$ and $v + \epsilon$ both lie in $\boldC(U)$, then so does $v$. We may therefore assume that $\alpha_i$ and $\beta_j$ are rational numbers.
    
       We can consider the sets $\boldsymbol{\alpha} = \{i: \alpha_i \neq 0\}$ and $\boldsymbol{\beta} = \{ j: \beta_j \neq 0\}$. Then there exist a positive integer $\lambda$ as well as positive integers $a_i$ and $b_j$, for each $i \in \boldsymbol{\alpha}$ and $j \in \boldsymbol{\beta}$, such that the presilting objects
    \[ \bigoplus_{i \in \boldsymbol{\alpha}} {R_i^{a_i}} \quad \text{and} \quad \bigoplus_{j \in \boldsymbol{\beta}} S_j^{b_j}\]
 have the same index, equal to $\lambda v$. One then uses \Cref{lem:IsFan}\eqref{lem:IsFan_inj} to show that {these two presilting objects} are isomorphic. Their largest basic direct summand is thus a direct summand of both $R$ and $S$, whence $v \in \boldC(U)$. This shows that $\partfan{\calC}$ consists of simplicial cones. We are left with proving that they are not strongly-convex. Assume for a contradiction that there exists a 1-dimensional subspace contained in $\boldC(U) = \spann_{\geq 0}(\bfg^{X_1}, \dots, \bfg^{X_\ell})$  and let $x$ and $-x$
 be two distinct elements of this subspace. Then we may write $x = \sum_{j=1}^s \lambda_j \bfg^{X_{a_j}}$ and $-x = \sum_{h = 1}^{t} \lambda_h \bfg^{X_{b_h}}$ with $\lambda_j, \lambda_h \in \R_{> 0}$ and $\{X_{a_j}\}_{j=1}^s \cap \{X_{b_h}\}_{h=1}^t = \emptyset$. By rearranging the terms, we may express $\bfg^{X_{a_1}}$ as a linear combination of $\{ \bfg^{X_{a_2}}, \dots, \bfg^{X_{a_s}}, \bfg^{X_{b_1}}, \dots, \bfg^{X_{b_t}}\}$, yielding a contradiction to the cone being simplicial. We conclude that $\partfan{\calC}$ is a simplicial polyhedral fan.     
\end{proof}

We call $\partfan{\calC}$ the \textit{$\bfg$-vector fan of $\calC$}. 
We want to develop relationships between the $\bfg$-vector fans of closely related extriangulated categories. For this it is necessary to first understand what the relationship between their $\bfg$-vectors is. This will be achieved in our next two lemmas. 

\begin{lemma}[Horseshoe Lemma]\label{lem:horseshoe}
    {Let $\calC$ and $T$ be as in \cref{setting}.} Given an $\E$-triangle
    \begin{equation*}
    \begin{tikzcd}
        A \arrow[r,tail, "a"] & B \arrow[r,two heads, "b"] & C \arrow[r,dashed, "\xi"] & {}
    \end{tikzcd}
    \end{equation*}
    in $\calC$, there exists a commutative diagram of conflations
    \begin{equation*}
        \begin{tikzcd}
T^A_1 \arrow[d, tail] \arrow[r, tail]        & T^A_1 \oplus T^C_1 \arrow[d, tail] \arrow[r, two heads]     & T^C_1 \arrow[d, tail] \arrow[r, dashed]      & {} \\
T^A_0 \arrow[d, two heads] \arrow[r, tail]   & T^A_0\oplus T^C_0 \arrow[d, two heads] \arrow[r, two heads] & T^C_0 \arrow[d, two heads] \arrow[r, dashed] & {} \\
A \arrow[r, "a", tail] \arrow[d, dashed] & B \arrow[r, "b", two heads] \arrow[d, dashed]             & C \arrow[r, "\xi", dashed] \arrow[d, dashed] & {} \\
{}                                           & {}                                                          & {}                                           &   
\end{tikzcd}
    \end{equation*}
    where the two top rows are split $\E$-triangles of projectives.
\end{lemma}
\begin{proof}
	This is the dual of the Horseshoe Lemma induced by inflations as in \cite[Cor. 5.13]{ZLZ2026}. More precisely, it follows from the dual of \cite[Thm. 5.12]{ZLZ2026}, which exists as it is variant (4)' in the list of all versions of the $4 \times 4$ Lemma in extriangulated categories \cite[p. 41, Table of 4 × 4 Lemma and its variants]{ZLZ2026}. 
\end{proof}

\begin{lemma}\label{lem:compositeind}
	Let $\calC$ be as in \Cref{setting} and let $U$ be a presilting object in $\calC$. Let $X$ be a presilting object in the extension-closed subcategory $\calZ_U \subseteq \calC$. Let 
	\[ \ind_{U^+}^{\calZ_U}(X) = [(U^+)_0] - [(U^+)_1] \in \K_0(\add_{\calZ_U} U^+) \cong \Ksp( \add U^+) \]
	be the index of $X$ with respect to the projective silting object $U^+$ in $\calZ_U$. Then we have
	\[ [ \ind_T^{\calC} ( (U^+)_0)] - [\ind_T^{\calC}((U^+)_1)] = \ind_T^{\calC} (X) \in \K_0(\add_{\calC} T).\]
\end{lemma}
\begin{proof}
	Follows directly from \cref{lem:horseshoe} and the definitions.
\end{proof}

Let us now show how our definition of the $\bfg$-vector interacts with the reduction techniques of 0-Auslander extriangulated categories.

\begin{lemma}\label{lem:ZUfan}
{Let $\calC$ be as in \Cref{setting}} and let $U$ be a basic presilting object in $\calC$. 
    \begin{enumerate}
        \item\label{lem:ZUfan1} Let $\overline{\starr}(\boldC(U))$ be as defined in \Cref{eg:orthproj}\eqref{eg:orthproj1}.
    Then the composite $\Z$-linear isomorphism
    \[ \K_0(\add_{\calZ_U} (U^+)) = \Ksp_0(\add (U^+))  \to \K_0( \add_{\calC} (T)), \quad [X] \mapsto [ \ind_T(X)] \]
    which is given by \cref{lem:IsFan}\eqref{lem:IsFan_GNP}, induces an $\R$-linear isomorphism of vectors spaces, denoted $\Phi_U$, which in turn determines an isomorphism of fans 
    \[\partfan{\calZ_U} \simeq \overline{\starr}(\boldC(U)). \]
    \item\label{lem:ZUfan2} {Similarly, $\Phi_U$ induces a map $\overline{\Phi}_U$ which gives rise to} an isomorphism of fans
    \[ ( \K_0(\add_{\calZ_U/[U]} (U^+)),\partfan{\calZ_U/[U]}) \simeq {(\sigma^\perp,} \pi_{\boldC(U)}(\starr(\boldC(U)))). \]
    \end{enumerate}
\end{lemma}
\begin{proof}
	First of all, we verify that our claim is mathematically rigourous. By \cref{prop:siltred}, the basic projective silting object in $\calZ_U$ is isomorphic to $U^+$. This means that ambient vector space of $\partfan{\calZ_U}$ is $\K_0(\add_{\calZ_U} (U^+))_\R$, and we know that the ambient vector space of $\partfan{\calC}$ is $\K_0(\add_{{\calC}}(T))_\R$. 
	
	Now we verify that the $\R$-linear map $\Phi_U$ sends cones in $\partfan{\calZ_U}$ to cones in $\partfan{\calC}$. Write $\boldC_{\calZ_U}(X)$ for the cone spanned by the $\bfg$-vectors of a presilting object $X$ in $\calZ_U$ with respect to $U^+$ in ${\K}_0(\add_{\calZ_U}(U^+))_\R$, and likewise $\boldC_{\calC}(X) \subseteq \K_0(\add_{\calC}(T))_\R$ for a presilting object $X$ in $\calC$.
	Let $\sigma \in \partfan{\calZ_U}$ be a cone. By definition, there exists a basic presilting object $X$ in $\calZ_U$ such that $\sigma = \boldC_{\calZ_U}(X)$ and let $\{ X_i\}_{i=1}^r$ be a complete set of its indecomposable direct summands. Since $\calZ_U$ is an extension-closed subcategory of $\calC$, the object $X$ is also presilting in $\calC$, therefore there is also a cone $\boldC_{\calC}(X) \in \partfan{\calC}$. Using the linearity of $\Phi_U$, we show
	\begin{align*}
	\Phi_U(\sigma) = \Phi_U(\boldC_{\calZ_U}(X)) &= \Phi_U( \spann_{\geq 0} \{ \ind_{U^+}^{\calZ_U}(X_i) \}_{i =1, \dots, r}) \\
	&= \spann_{\geq 0} \{ \Phi_U(\ind_{U^+}^{\calZ_U} (X_i)) \}_{i=1, \dots, r} \\
	&=  \spann_{\geq 0} \{ \ind_{T}^{\calC} (X_i) \}_{i=1, \dots, r} & \text{(by \cref{lem:compositeind})} \\
	&= \boldC_{\calC}(X) \subseteq \K_0(\add_{\calC}(T)) 
	 \end{align*}
	Therefore, the image of every cone in $\partfan{\calZ_U}$ is a cone {in} $\partfan{\calC}$. Moreover, the definition of $\calZ_U$ implies that $X \oplus U$ is presilting in $\calC$, which implies $\boldC_{\calC}(X) \in \overline{\starr}(\boldC(U))$. It follows from \cref{lem:IsFan}\eqref{lem:IsFan_GNP} that the $\R$-linear map induced by the coindex gives the desired inverse morphism of fans. This concludes the proof of \eqref{lem:ZUfan1}.
	
	To show \eqref{lem:ZUfan2}, we first discuss the ambient vector spaces of the two fans. The $\bfg$-vector fan $\partfan{\calZ_U/[U]}$ {is defined over} the vector space $\K_0(\add_{\calZ_U/[U]}(U^+))_\R$ which can be canonically identified with $\Ksp_0(\add (U^+/U))_\R$. Consider the basis of $\K_0(\add_{\calC} T)_\R$ induced by the indecomposable direct summands of $T$ which gives rise to a canonical isomorphism between this $\R$-vector space and its dual. Fix a basis of $\Ksp_0(\add U^+)_\R$ induced by the indecomposable direct summands of $U^+$. Similarly, let $\mathrm{P}_U$ be the orthogonal projection endomorphism of $\Ksp_0(\add U^+)_\R$ which sends the basis elements induced by indecomposable direct summands of $U$ to zero. {Similarly, let $\mathrm{P}_{\boldC(U)}$ be the endomorphism on $\K_0(\add_{\calC} T)_\R^*$ whose image is $\boldC(U)^\perp$. In other words, so that the projection onto this subspace is the map $\proj_{\boldC(U)}$ as defined in \cref{eg:orthproj}\eqref{eg:orthproj_proj}.
	
	By definition, the ambient vector space of the fan $\pi_{\boldC(U)}(\starr(\boldC(U)))$ is $\image (\mathrm{P}_{\boldC(U)})$. Using {these $\R$-linear maps} , let $\overline{\Phi}_U$ be the composite $\R$-linear map given by}
	\[
	\begin{tikzcd}
	\K_0(\add_{\calZ_U/[U]}(U^+))_\R \arrow[r, "\simeq"] & \Ksp_0(\add U^+/U)_\R \arrow[r, "\simeq"] & \image(\mathrm{P}_U) \arrow[r,  hookrightarrow] & \Ksp_0(\add U^+)_\R \arrow[d, "\Phi_U"] \\
	& \image (\mathrm{P}_{\boldC(U)}) & \K_0(\add_{\calC}(T))_\R^* \arrow[l, twoheadrightarrow, "{{\proj_{\boldC(U)}}}",swap] & \K_0(\add_{\calC}( T))_\R \arrow[l, "\simeq"]
	\end{tikzcd} \]
	Let $\sigma \in \partfan{\calZ_U/[U]}$ be a cone, which means that there exists a presilting object $X$ in $\calZ_U/[U]$ such that $\sigma = \boldC_{\calZ_U/[U]}(X)$. Again, as $X$ is also presilting in $\calC$, the image of $\sigma$ along the top chain of maps composed with $\Phi_U$ yields the cone $\boldC_{\calC}(X)$ as in \eqref{lem:ZUfan1}. However, since $U \not \in \add(X)$, we have that 
	\[ \boldC_{\calC}(X) \in \overline{\starr}(\boldC(U)) \setminus \starr(\boldC(U)). \]
	Notice that $\mathrm{P}_{\boldC(U)}(\boldC_{\calC}(X)) = \mathrm{P}_{\boldC(U)}(\boldC_{\calC}(X \oplus U))$, which means that $\overline{\Phi}_U(\sigma)$ is a cone in $\pi_{\boldC(U)}(\starr(\boldC(U)))$. The inverse morphism of fans is obtained by swapping the roles of inclusion and projection in the composite chain of morphisms above, and replacing $\Phi_U$ by its inverse. This proves \eqref{lem:ZUfan2}.
\end{proof}

In \Cref{eq:0Aus}, we provided classes of examples of 0-Auslander extriangulated $k$-categories. Building on this, we now give a handful of examples of $\bfg$-vector fans.

\begin{example}\label{ex:fans}
    
    \begin{enumerate}
        \item\label{ex:fans_Preproj} Let $\Pi(A_2)$ denote the preprojective $k$-algebra of Dynkin type $A_2$. The $\bfg$-vector fan of $\Ktwo(\proj \Pi(A_2))$ is displayed in \Cref{fig:fanOfPreproj}. Since $\Pi(A_2)$ is a $k$-algebra {with only finitely many basic $\tau$-rigid objects}, its $\bfg$-vector fan is finite, and also complete in $\K_0(\proj \Pi(A_2))_{\R}\simeq \R^2$ {by} \cite[{Thm.} 4.7]{Asa21}.
    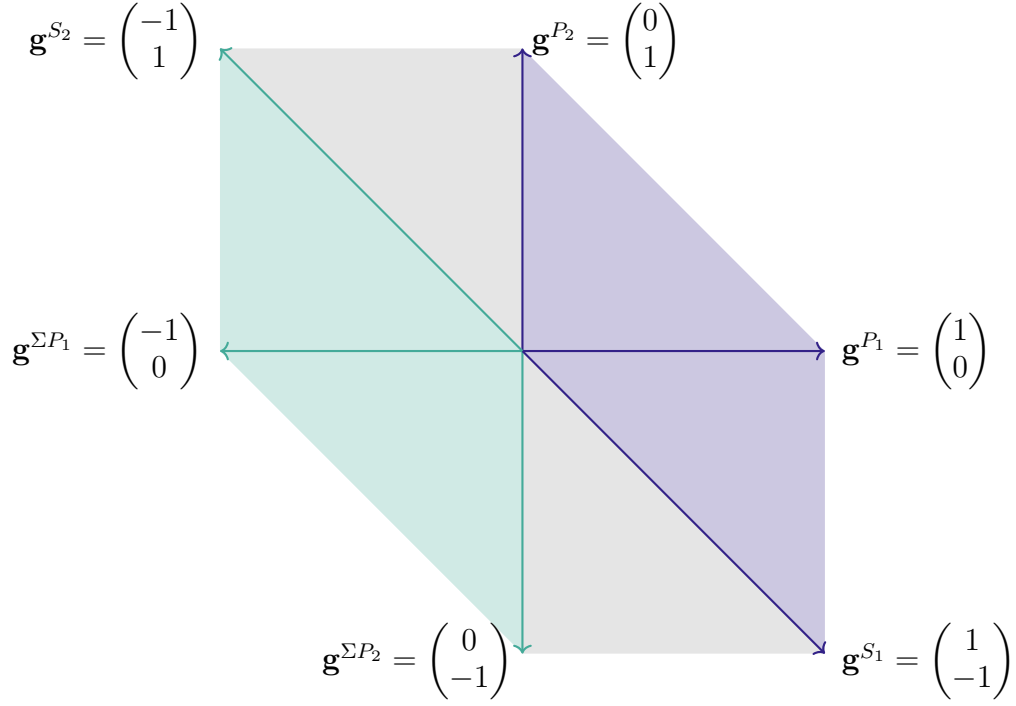
\begin{figure}[ht!]
    \begin{tikzpicture}[font=\Large]
        \fill[friendly_green, opacity=0.25] (0,0) -- (-4,4) -- (-4,0) -- (0,-4) -- (0,0);
        \fill[friendly_deepblue, opacity=0.25] (0,0) -- (4,-4) -- (4,0) -- (0,4) -- (0,0);
        \fill[gray, opacity=0.2] (0,0) -- (0,4) -- (-4,4) -- (0,0);
        \fill[gray, opacity=0.2] (0,0) -- (0,-4) -- (4,-4) -- (0,0);
        \draw[->, thick, friendly_deepblue] (0,0) -- (4,0);
        \draw[->, thick, friendly_green] (0,0) -- (-4,0);
        \draw[->, thick, friendly_deepblue] (0,0) -- (0,4);
        \draw[->, thick, friendly_green] (0,0) -- (0,-4);
        \draw[->, thick, friendly_deepblue] (0,0) -- (4,-4);
        \draw[->, thick, friendly_green] (0,0) -- (-4,4);
        \filldraw[black] (4.1,0) circle (0pt) node[right]{ ${\bfg^{P_1}=\begin{pmatrix} 1 \\ 0
              \end{pmatrix}}$};
        \filldraw[black] (-4.1,0) circle (0pt) node[left]{${\bfg^{\Sigma P_1}=\begin{pmatrix} -1 \\ 0
              \end{pmatrix}}$};
        \filldraw[black] (0,4.1) circle (0pt) node[right]{${\bfg^{P_2}=\begin{pmatrix} 0 \\ 1
              \end{pmatrix}}$};
        \filldraw[black] (0,-4.1) circle (0pt) node[left]{${\bfg^{\Sigma P_2}=\begin{pmatrix} 0 \\ -1
              \end{pmatrix}}$};
        \filldraw[black] (4.1,-4.1) circle (0pt) node[right]{${\bfg^{S_1}=\begin{pmatrix} 1 \\ -1
              \end{pmatrix}}$};
        \filldraw[black] (-4.1,4.1) circle (0pt) node[left]{ ${\bfg^{S_2}=\begin{pmatrix} -1 \\ 1
              \end{pmatrix}}$};
    \end{tikzpicture}
    \caption{The $\bfg$-vector fan of {$\Ktwo(\proj \Pi(A_2))$, where $\Pi(A_2)$ is} the preprojective $k$-algebra of Dynkin type $A_2$.}
    \label{fig:fanOfPreproj}
    \end{figure}
    \item\label{ex:fanOf_A2} Let $\Lambda_{2}$ denote the path $k$-algebra of a Dynkin quiver of type $A_2$, and regard $\mods(\Lambda_2)$ as a 0-Auslander extriangulated $k$-category.
    The $\bfg$-vector fan of $\mods(\Lambda_2)$ can be displayed as in \Cref{fig:fanOfA2}.
    \begin{figure}[ht!]
    \begin{tikzpicture}[font=\Large]
        \fill[gray, opacity=0.2] (0,0) -- (4,0) -- (0,4) -- (0,0);
        \fill[gray, opacity=0.2] (0,0) -- (0,4) -- (-4,4) -- (0,0);
        \draw[->, thick] (0,0) -- (4,0);
        \draw[->, thick] (0,0) -- (-4,4);
        \draw[->, thick] (0,0) -- (0,4);
        \filldraw[black] (4.1,0) circle (0pt) node[right]{ $\bfg^{P_1}=\begin{pmatrix} 1 \\ 0
              \end{pmatrix}$};
        \filldraw[black] (0,4.1) circle (0pt) node[right]{ $\bfg^{P_2}=\begin{pmatrix} 0 \\ 1
              \end{pmatrix}$};
        \filldraw[black] (-4.1,4.1) circle (0pt) node[left]{ $\bfg^{S_2}=\begin{pmatrix} -1 \\ 1
              \end{pmatrix}$};
    \end{tikzpicture}
    \caption{The $\bfg$-vector fan of $\mods(\Lambda_2)$ as an extriangulated category, where $\Lambda_{2}$ is the path $k$-algebra of a Dynkin quiver of type $A_2$.}
    \label{fig:fanOfA2}
    \end{figure}
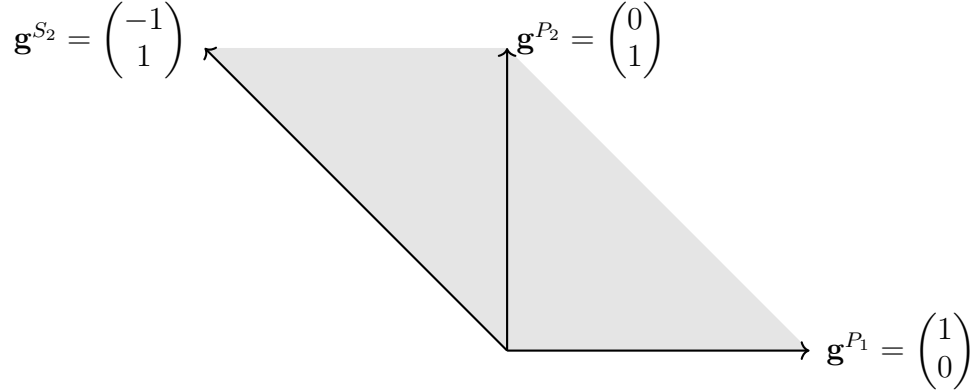
    This is a finite polyhedral fan in $\K_0(\proj \Lambda_{2})_{\R}\simeq \R^2$, but it is not complete.
    \item\label{ex:fans_PreprojZ} 
    We will now build on the example in \eqref{ex:fans_Preproj} above. Let $P_{1}$ denote one of the indecomposable projective $\Pi(A_2)$-modules. In \Cref{fig:fanOfPreproj}, the $\bfg$-vector of $P_{1}$ points horizontally to the right, and the $\bfg$-vector of $\Sigma P_{1}$ point{s} horizontally to the left. By \Cref{prop:siltred}\eqref{prop:siltred_0Aus}, the extension-closed subcategories $\calZ_{P_{1}}$ and $\calZ_{\Sigma P_{1}}$ are 0-Auslander in their own right, and they also satisfy \Cref{setting}. {It follows directly from the definition of these subcategories that their fans are equivalent to the fans $\overline{\mathrm{star}}(\boldC(P_1))$ defined in \cref{eg:orthproj}\eqref{eg:orthproj1}, which naturally embed into the $\bfg$-vector fan.} Their $\bfg$-vector fans are displayed as subfans in \Cref{fig:fanOfPreproj}, coloured \textcolor{friendly_deepblue}{blue} and \textcolor{friendly_green}{green}, respectively. Since both $\calZ_{P_{1}}$ and $\calZ_{\Sigma P_{1}}$ are equivalent to $\mods(\Lambda_2)$, where $\Lambda_2$ is defined as in \eqref{ex:fanOf_A2} above, the $\bfg$-vector fans of $\calZ_{P_{1}}$ and $\calZ_{\Sigma P_{1}}$ are both isomorphic to the fan displayed in \Cref{fig:fanOfA2}.
    \item\label{ex:fans_subA3} Let $\Lambda_3$ denote the path $k$-algebra of the quiver 
$
\begin{tikzcd}
	1 \arrow[r,"a"] & 2 \arrow[r,"b"] & 3.
\end{tikzcd}
$
The Auslander--Reiten quiver of the $k$-category $\mathrm{mod}(\Lambda_3)$ then takes the form
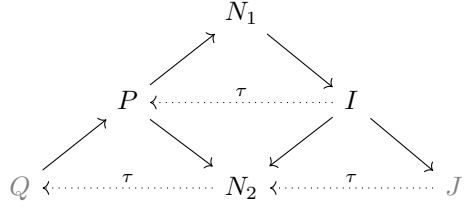
\begin{equation*}
	\begin{tikzcd}
             &                         & N_1 \arrow[rd]                                &                                                                   &                                             \\
             & P \arrow[ru] \arrow[rd] &                                               & I \arrow[rd] \arrow[ld] \arrow[ll, "{\scriptstyle\tau}"', dotted] &                                             \\
	\textcolor{gray}{Q} \arrow[ru] &                         & N_2 \arrow[ll, "{\scriptstyle\tau}"', dotted] &                                                                   & \textcolor{gray}{J} \arrow[ll, "{\scriptstyle\tau}"', dotted]
	\end{tikzcd}
\end{equation*}
where $\tau$ denotes the Auslander--Reiten translation. Let $\calC$ denote the extension-closed subcategory of $\mathrm{mod}(\Lambda_3)$ given by the smallest additive subcategory containing $P$, $N_1$, $N_2$, and $I$. This Quillen exact $k$-category is 0-Auslander as an extriangulated $k$-category, in which $P$ is projective and non-injective, the object $I$ is injective and non-projective, and $N_1$ and $N_2$ are projective-injective. The $\bfg$-vector fan of $\calC$ is displayed in \Cref{fig:FanOfSubA3}. Like in \eqref{ex:fanOf_A2} above, we get a finite polyhedral fan which is not complete.
    \begin{figure}[ht!]
        \centering
        \begin{tikzcd}[font = \Large,column sep=0.5em,row sep=1em,
	ampersand replacement=\&,
	execute at end picture={
    \scoped[on background layer]
    \fill[opacity=0.25, color=friendly_deepblue] (o.center) -- (n1.center) -- (out.center) -- (n2.center) -- cycle;}]
       {} \& |[alias=i]|{} \& {} \&  {}                                                               \&  {}   \& |[alias=out]| {}    \&           \\
	  {\quad\qquad {\bfg^{I}}=\begin{pmatrix} -1 \\ 1 \\ 1
              \end{pmatrix}} \& {} \&   \&   |[alias=n2]|{\quad\qquad\bfg^{N_2}=\begin{pmatrix} 0 \\ 0 \\ 1
              \end{pmatrix}}                                                 \& {}  \&    {}       \&    {}       \\
        {} \& {} \&  {} \&     {}    \& {} \&  |[alias=n1]|{\bfg^{N_1}=\begin{pmatrix} 0 \\ 1 \\ 0
              \end{pmatrix}} \&     {}      \\
      {}   \& {}\arrow[uu,no head,dotted,start anchor=center,end anchor=center] \&  \&                                                               \&  \&           \&           \\
        {}\arrow[ru,no head,dotted,start anchor=center,end anchor=center] \& \& \&  |[alias=o]| {} \arrow[lll,no head,dotted,start anchor=center,end anchor=center] \arrow[rrr, thick,start anchor=center] \arrow[rruu,start anchor=center] \arrow[uuu, thick,start anchor=center] \arrow[lluuu,start anchor=center,end anchor=center] \& {} \&           \&  |[alias=p]|{\bfg^{P}=\begin{pmatrix} 1 \\ 0 \\ 0
              \end{pmatrix}}
\end{tikzcd}
        \caption{The $\bfg$-vector fan of the 0-Auslander $k$-category defined in \Cref{ex:fans}\eqref{ex:fans_subA3}. The nodes display the coordinates in $\R^3$ of the tips of the $\bfg$-vectors, the shaded two-dimensional cone (non-negatively spanned by $\bfg^{N_1}$ and $\bfg^{N_2}$) is the intersection of the two maximal (three-dimensional) cones, and the $\bfg$-vectors $\bfg^{P}$ and $\bfg^{I}$ lie in distinct maximal cones. The dotted lines are visual aids to illustrate how the fan sits in three-dimensional space.}
        \label{fig:FanOfSubA3}
    \end{figure}
    \end{enumerate}
\end{example}

For a finite-dimensional $k$-algebra $\Lambda$, the $\bfg$-vector fan in the sense of \cite[Cor. 6.7(b)]{DIJ2019} is recovered as the $\bfg$-vector fan of $\Ktwo(\proj \Lambda)$, see \cref{eq:0Aus}\eqref{eg:0Aus_K2}. We write $\partfan{\Lambda} \coloneqq \partfan{\Ktwo(\proj \Lambda)}$ from here on. This section will be concluded by proving useful properties holding in our more general setting. We begin with a lemma that will let us construct morphisms between $\bfg$-vector fans based on the existence of commutative diagrams as in \cref{lem:indexFanIso2} or \cref{lem:Findexsquare}. 

\begin{lemma}\label{lem:FanIso}
    Let $\calC$, $\Lambda$ and $T$ be as in \Cref{setting}, and let $\calD$ and $S$ satisfy the same setup.
    \begin{enumerate}
        \item\label{lem:FanIso1} Suppose that we have a commutative square
    \begin{equation}\label{eq:FanIso_commsq}
    \begin{tikzcd}[column sep=6em]
        \presilt({\calC}) \arrow[d,"\bfg{=}\ind_T\otimes 1"] \arrow[r] & \presilt({\calD}) \arrow[d,"\bfg{=}\ind_S\otimes 1"] \\
        \K_0(\add_{\calC}T)_{\R} \arrow[r,"f"] & \K_0(\add_{\calD}S)_{\R}
    \end{tikzcd}
    \end{equation}
    in which the bottom row displays an $\R$-linear transformation $f$.
   Then $f$ induces a morphism of fans
    \begin{equation}\label{eq:FanIo_fan}
        \partfan{\calC} \to \partfan{\calD}.
    \end{equation}
    If the horizontal maps in \eqref{eq:FanIso_commsq} are bijections, then the morphism of $\bfg$-vector fans in \eqref{eq:FanIo_fan} is an isomorphism.
    \item\label{lem:FanIso2} If $G: \calC \to \calD$ is an extriangulated functor {satisfying the assumptions of \cref{lem:indexFanIso2}, then} $G$ induces a morphism of fans 
    $
        \partfan{\calC} \to \partfan{\calD}.
    $
   \item\label{thm:modulecatreduction_fans}
   Suppose that $\calC$ is reduced. {The map $\calC(T,-)$ induces an isomorphism of $\bfg$-vector fans}
    \begin{equation}\label{eq:partfanCtoLambda}
        (\K_0(\add T)_\R, \partfan{\calC}) \to (\K_0(\add \Lambda)_\R, \partfan{\Lambda}).
    \end{equation}
    \end{enumerate}
\end{lemma}
\begin{proof}
    {To prove \eqref{lem:FanIso1}}, {it} is to be verified that the $\R$-linear transformation $f$ along the bottom row of \eqref{eq:FanIso_commsq} satisfies \Cref{def:morfan}, but since the commutativity of \eqref{eq:FanIso_commsq} entails that $f$ sends cones in $\partfan{\calC}$ to cones in $\partfan{\calD}$, this is automatic. It is also clear that the bijectivity of the horizontal morphisms is equivalent to the morphism in \eqref{eq:FanIo_fan} being an isomorphism of fans. The claim in {\eqref{lem:FanIso2} follows immediately from \cref{lem:indexFanIso2} and \eqref{lem:FanIso1}.}
    The proof of \eqref{thm:modulecatreduction_fans} follows directly from \cref{lem:Findexsquare} and \eqref{lem:FanIso1}.
\end{proof}

\begin{lemma}\label{lem:Jredfan}
    Let $\calC$ and $T$ be as in \Cref{setting}. {Let $U$ be a presilting object in $\calC$.}
    \begin{enumerate}
    \item\label{lem:Jredfan1} The functor $\calZ_U \to \calZ_U/[U]$ induces a morphism of $\bfg$-vector fans 
    \[ \partfan{\calZ_U} \to \partfan{\calZ_U/[U]}. \]
    \item\label{lem:Jredfan2} The extriangulated equivalence $\calZ_U/[U] \to \calC/\thick(U)$, see \Cref{prop:siltred}\eqref{prop:siltred_equiv}, induces an isomorphism of $\bfg$-vector fans 
    \[ \partfan{\calZ_U/[U]} \xrightarrow{\simeq}  \partfan{\calC/\thick(U)}. \]
    \item\label{lem:Jredfan_star} 
    There is a morphism of fans 
    \[(\K_0(\add_{\calC} T)_\R,  \overline{\mathrm{star}}(\boldC(U))) \to (\K_0(\add_{\calC/\thick(U)} T)_\R,\partfan{\calC/\thick(U)}) \]
    which may be identified with orthogonal projection onto $\boldC(U)^\perp \subseteq \K_0(\add_{\calC} T)_\R^*$.
    \end{enumerate}
    Let $\mathcal{J}$ be an ideal of $\calC$ generated by morphisms with injective domain and projective codomain, let $N(\mathcal{J})$ denote the largest basic object in $\calC$ such that $[N{(\mathcal{J})}]\subseteq \mathcal{J}$ (note that $N(\mathcal{J})$ must be projective-injective), and define $\overline{\calC}_{\mathcal{J}}\coloneqq \calC/[N(\mathcal{J})]$.
    \begin{enumerate}
    \setcounter{enumi}{3}
    \item\label{lem:Jredfan_iso} The natural functor 
    \begin{equation}\label{eq:NtoJ}
        \overline{\calC}_{\mathcal{J}} \to \calC/\mathcal{J}
    \end{equation}
    induces an isomorphism of $\bfg$-vector fans
    \begin{equation}\label{eq:NtoJ_fan}
        \partfan{\overline{\calC}_{\mathcal{J}}} \to \partfan{\calC/\mathcal{J}}.
    \end{equation}
    \item\label{lem:Jredfan_combo} Combining \eqref{lem:Jredfan_star} and \eqref{lem:Jredfan_iso} above, we deduce that the natural functor $\calC\to \calC/\mathcal{J}$ induces a 
    {morphism of fans}
    \begin{equation*}
        (\K_0(\add_\calC T)_\R, \overline{\mathrm{star}}({\boldC(N(\mathcal{J})}))) \to (\K_0(\add_{\calC/\mathcal{J}} T)_\R, \partfan{\calC/\mathcal{J}})
    \end{equation*}
    given by orthogonal projection (note that $\calC/\mathcal{J} = \calC/(\add N(\mathcal{J})) = \calC / \thick(N(\mathcal{J}))$).
\end{enumerate}
\end{lemma}
\begin{proof}
    We first prove \eqref{lem:Jredfan1}. 
    The functor $\calZ_U \to \calZ_U/[U]$ is the ideal quotient of an extriangulated $k$-category with respect to a class of projective-injective objects, so by \Cref{prop:FGPPP23}\eqref{prop:FGPPP23.3.3} it induces an extriangulated functor. As a result of \Cref{prop:FGPPP23}\eqref{prop:FGPPP23.3.5} and \cref{prop:siltred}\eqref{prop:siltred_0Aus_proj}, it also satisfies the assumptions of \cref{lem:indexFanIso2}. Therefore, the desired morphism of fans exists by \cref{lem:FanIso}\eqref{lem:FanIso2}. 
    To prove \eqref{lem:Jredfan2} simply apply \cref{lem:FanIso}\eqref{lem:FanIso2} again. The assumptions are met because the functor is an equivalence, giving a morphism of fans, as desired. Moreover, it is an isomorphism, again because the functor is an equivalence.

    The claim \eqref{lem:Jredfan_star} follows from the combination of \eqref{lem:Jredfan1} and \eqref{lem:Jredfan2} with \cref{lem:ZUfan}. That is, there is a chain of morphisms of fans
    \[ \begin{tikzcd}
        \overline{\starr}(\boldC(U)) \arrow[r, "\simeq", "\ref{lem:ZUfan}"'] & \partfan{\calZ_U} \arrow[r, "\eqref{lem:Jredfan1}"'] & \partfan{\calZ_U/[U]} \arrow[r, "\eqref{lem:Jredfan2}"', "{\simeq}"] & \partfan{\calC/\thick(U)}{,}
    \end{tikzcd}\]
    where we have suppressed the ambient vector spaces from the notaion. Hence the desired morphism of fans exists. It can be identified with orthogonal projection, because $\partfan{\calZ_U/[U]} \simeq \pi_{\boldC(U)}(\starr(\boldC(U))) \subseteq \K_0(\add_{\calC}T)_\R^*$ by \cref{lem:ZUfan}.
    It is now enough to prove \eqref{lem:Jredfan_iso}, since \eqref{lem:Jredfan_combo} is simply the conjunction of  {\eqref{lem:Jredfan_star} and \eqref{lem:Jredfan_iso}}.

Next, we prove \eqref{lem:Jredfan_iso}. Let $\overline{\mathcal{J}}$ denote the ideal $\mathcal{J}/[N(\mathcal{J})]$ of $\overline{\calC}_{\mathcal{J}}$. The natural functor $\overline{\calC}_{\mathcal{J}} \xrightarrow{} \calC/\mathcal{J}$ is equivalent to the ideal quotient $\overline{\calC}_{\mathcal{J}} \xrightarrow{} \overline{\calC}_{\mathcal{J}}/[\overline{\mathcal{J}}]$. By \Cref{prop:FGPPP23}, the functor $\overline{\calC}_{\mathcal{J}} \xrightarrow{} \calC/\mathcal{J}$ is an extriangulated functor satisfying condition {\eqref{lem:indexIso2a} in \cref{lem:indexFanIso2}}. Moreover, it sends the basic projective silting object in $\overline{\calC}_{\mathcal{J}}$ to a projective silting object in $\calC/\mathcal{J}$. It now follows from \Cref{lem:FanIso}\eqref{lem:FanIso2} that a morphism of fans $\partfan{\overline{\calC}_{\mathcal{J}}} \xrightarrow{} \partfan{\calC/\mathcal{J}}$ is induced.
    In particular, the following square commutes.
    \begin{equation}\label{eq:Jredfan_iso_commsq}
    \begin{tikzcd}[column sep=6em]
        \presilt(\overline{\calC}_{\mathcal{J}}) \arrow[d,"\bfg\coloneqq\ind_T\otimes 1"] \arrow[r] & \presilt({\calC}/{\mathcal{J}}) \arrow[d,"\bfg\coloneqq\ind_T\otimes 1"] \\
        \K_0(\add_{\overline{\calC}_{\mathcal{J}}}T)_{\R} \arrow[r] & \K_0(\add_{{\calC}/{\mathcal{J}}}T)_{\R}
    \end{tikzcd}
    \end{equation}
    It remains to show that the horizontal maps in \eqref{eq:Jredfan_iso_commsq} are bijections.
    By \Cref{prop:FGPPP23}\eqref{prop:FGPPP23.3.5}, the top map is a bijection. 
    Finally, we show that the natural functor in \eqref{eq:NtoJ} induces an isomorphism 
    \begin{equation*}
    \begin{tikzcd}[column sep=4em]
        \K_0(\add_{\overline{\calC}_{\mathcal{J}}}T)_{\R} \arrow[r] & \K_0(\add_{{\calC}/{\mathcal{J}}}T)_{\R}
    \end{tikzcd}
    \end{equation*}
    of $\R$-vector spaces. Since $T$ is projective, both when regarded as an object in $\overline{\calC}_{\mathcal{J}}$ and in ${\calC}/{\mathcal{J}}$, it is equivalent to show that the natural functor in \eqref{eq:NtoJ} induces an equivalence of categories
    \begin{equation*}
        \begin{tikzcd}[column sep=4em]
        \add_{\overline{\calC}_{\mathcal{J}}}T \arrow[r] & \add_{{\calC}/{\mathcal{J}}}T,
    \end{tikzcd}
    \end{equation*}
    whence it suffices to show that the natural functor in \eqref{eq:NtoJ} induces an isomorphism
    \begin{equation}\label{eq:Jredfan_Ends}
    \begin{tikzcd}[column sep=4em]
        \End_{\overline{\calC}_{\mathcal{J}}}(T) \arrow[r] & \End_{{\calC}/{\mathcal{J}}}(T)
    \end{tikzcd}
    \end{equation}
    of endomorphism $k$-algebras. Since both $k$-algebras occurring in \eqref{eq:Jredfan_Ends} can be naturally identified with $\End_{\calC}(T)/[N(\mathcal{J})]$, where $[N(\mathcal{J})]$ is the ideal of $\End_{\calC}(T)$ generated by morphisms factoring through the identity on $N(\mathcal{J})$, we may conclude the proof here.
\end{proof}

As an illustrating example of the above result, consider \Cref{ex:fans}\eqref{ex:fanOf_A2}, where $\calC=\mods(\Lambda_2)$ and $\Lambda_2$ is the path $k$-algebra of a Dynkin quiver of type $A_2$. The $\bfg$-vector fan of $\calC$ is shown in \Cref{fig:fanOfA2}. Consider the indecomposable presilting object $P_2$, whose $\bfg$-vector point{s} upwards in \Cref{fig:fanOfA2}. The $\bfg$-vector fan of $\calC/\thick(P_2)$ can then be drawn in the orthogonal complement of $\bfg^{P_2}$, which would be along the horizontal axis in \Cref{fig:fanOfA2} if we identify the dual space with the original via the dual bases induced by the basis consisting of indecomposable projective objects.

The observations in \Cref{prop:F} allow us to generalise an important result of Asai \cite[Thm. 4.7]{Asa21} to our setting. We encourage the reader to compare the last result of this section with the examples given in \Cref{ex:fans}. 

If $\calC$ satisfies \Cref{setting}, we say that it is \textit{$\bfg$-finite} if, in addition, the $\bfg$-vector fan $\partfan{\calC}$ consists of finitely many cones. Observe that this is equivalent to $\calC$ being silting-finite as an extriangulated $k$-category.

\begin{proposition}\label{prop:Asa21.4.7}
Let $\calC$ be as in \Cref{setting}.
The $\bfg$-vector fan $\partfan{\calC}$ is complete if and only if $\calC$ is $\bfg$-finite and reduced. 
\end{proposition}
 \begin{proof}
    $(\Leftarrow)$ Suppose that $\calC$ is {$\bfg$-finite and} reduced. 
    The assertion is well-known to hold when $\calC = \Ktwo(\proj \Lambda)$, for some finite-dimensional $k$-algebra $\Lambda${, see} \cite[{Thm.} 4.7]{Asa21}.
    In general, \cref{lem:FanIso}\eqref{thm:modulecatreduction_fans} shows that the $\bfg$-vector fan of $\calC$ is isomorphic to the $\bfg$-vector fan of a finite-dimensional $k$-algebra, so we conclude that if $\calC$ is $\bfg$-finite and reduced, then $\partfan{\calC}$ is complete. 
     
    $(\Rightarrow)$ Suppose that $\partfan{\calC}$ is complete. We argue by contradiction. If $\calC$ is reduced but not $\bfg$-finite, its $\bfg$-vector fan is isomorphic to that of $\Ktwo(\proj \Lambda)$, for some finite-dimensional $k$-algebra $\Lambda$, by \cref{lem:FanIso}\eqref{thm:modulecatreduction_fans}. Referring to Asai again, we have that $\partfan{\calC}$ cannot be complete \cite[Thm. 4.7]{Asa21}, a contradiction.
    Suppose now that $\calC$ is not reduced, and let $N$ be a non-zero indecomposable projective-injective object in $\calC$. We show that $\partfan{\calC}$ is not complete by showing that the vector $-\bfg^{N}$ cannot lie inside any cone. Suppose to the contrary that there exists a cone $\mathbf{C}{(R)}$ in $\partfan{\calC}$ containing $-\bfg^{N}$ and let $X=\{\mathbf{c}_1, \cdots ,\mathbf{c}_{\ell} \}$ be the set of $\bfg$-vectors {of the indecomposable direct summands of the presilting object $R$}. Since ${N}$ is projective-injective, there are two possibilities:  If $N$ is a direct summand of $R$, then $\mathbf{c}_i=\bfg^{{N}}$ for some $i$. In this case, both $\bfg^N$ and $-\bfg^N$ lie in $\boldC(R)$, a contradiction to it being strongly convex. Otherwise, $R \oplus N$ is presilting and defines a higher-dimensional $\bfg$-vector cone spanned by $\{ \mathbf{c}_1, \dots, \mathbf{c}_{\ell}, \bfg^{{N}}\}$. In this case we can write $- \bfg^N$ and hence $\bfg^N$ as a linear combination of $\mathbf{c}_1, \dots, \mathbf{c}_{\ell}$, which contradicts the linear independence of the $\bfg$-vectors generating the larger cone.
	 \end{proof}

\section{Partitions and picture categories}\label{sec:part}
Let $(\calC,\E)$ be a 0-Auslander extriangulated category satisfying \cref{setting}. In this section, we show that the thick subcategories of $\calC$ generated by presilting objects induce an admissible partition of the $\bfg$-vector fan $\partfan{\calC}$ in the sense of \cite{Kai23}. The construction of \cite{Kai23} then yields what we call the \textit{picture category} of $\calC$. In \cref{prop:Monica}, we establish a novel relationship between thick subcategories of $\calC$ generated by presilting objects and so-called $\tau$-perpendicular subcategories of the module category $\mods(\Lambda)$, where $\Lambda$ is the endomorphism $k$-algebra of the projective silting object in {$\overline{\calC}$}. This generalises the result \cite[{Prop.} 4.21]{Gar24}, which establishes a similar connection for proper subsets of the respective subcategories in a more restrictive setting. Moreover, this shows that the picture category is equivalent to the $\tau$-cluster morphism category in the case where $\calC = \Ktwo(\proj \Gamma)$ for some finite-dimensional $k$-algebra $\Gamma$. The $\tau$-cluster morphism category was developed in a sequence of papers \cite{IT17, BM18w, BH21} as a tool for studying the \textit{picture group} of $\Gamma$ \cite{ITW16}, which arises as the fundamental group of its classifying space. The question of whether the classifying space of this category is a $K(\pi,1)$ space for the picture group has inspired significant research \cite{BorMot,BorveKaipeltautilt,HI21,HI21p,Kai24} and the category has been transferred to more general contexts \cite{Bor21,Bor24, Kai23}. In this section, we apply the techniques of \cite{Kai23} to unify all other constructions of this category and, in \cref{sec:morpartfan}, of the results known about functorial relationships between different $\tau$-cluster morphism categories.

\begin{definition}[{\cite[§3]{Kai23}}]\label{def:admisspart}
Let $N$ be a finite lattice and let $\Sfan$ be a simplicial polyhedral fan in $N_{\R}$. Let $\sigma \in \Sfan$ be a cone of $\Sfan$. 
    \begin{enumerate}
    \item\label{def:admisspart_1} Let $\pi_{\sigma}\colon N_\R \to \sigma^{\perp}$ denote the orthogonal projection onto the orthogonal complement of $\spann(\sigma)$. Consider the fan $\pi_{\sigma}(\mathrm{star}(\sigma))$ in $\sigma^{\perp}$. The \textit{set of potential identifications with $\sigma$} is defined by the set
  \begin{equation*}
 	 \mathcal{E}_{\sigma} \coloneqq \{\kappa \in \Sfan \sth \sigma^{\perp} = {\kappa}^{\perp} \text{ and } \pi_{\sigma}(\mathrm{star}(\sigma)) = \pi_{{\kappa}}(\mathrm{star}({\kappa})) \}.
 \end{equation*}
A \textit{partition} $\mathfrak{P}$ of $\Sfan$ is given by a family $\{\mathfrak{P}_{\sigma}\}_{\sigma \in \Sfan}$, where $\mathfrak{P}_{\sigma}=\{\mathcal{E}^1_{\sigma},\dots, \mathcal{E}^{m_\sigma}_{\sigma} \}$ is a partition of $\mathcal{E}_{\sigma}$ and $\Pfrak_{\sigma}= \Pfrak_{\kappa}$ when $\mathcal{E}_{\sigma}= \mathcal{E}_{\kappa}$. We write $\sigma_1 \sim \sigma_2$ if $\sigma_1,\sigma_2\in \mathcal{E}^{i}_{\sigma}$ for some $\sigma\in \Sfan$ and $1\leq i\leq m_\sigma$.
    \item A partition $\mathfrak{P}$ of $\Sfan$ is \textit{admissible} if the following implication holds: \\
if $\sigma_1\sim \sigma_2$ and $\pi_{\sigma_1}(\kappa_1) = \pi_{\sigma_2}(\kappa_2)$ for some $\kappa_i \in \mathrm{star}(\sigma_i)$ and $i=1,2$, then $\kappa_1 \sim \kappa_2$.
\item A \textit{partitioned fan} {is a pair $(\Sfan, \mathfrak{P})$ consisting of} a simplicial fan {$\Sfan$} equipped with an admissible partition $\mathfrak{P}$ of its cones. 
\end{enumerate}
\end{definition}

It is established in \cite[§5]{Kai23} that different choices of partitions give rise to structural relationships between the categories. In this article, we are mostly interested in studying one particular partition of the $\bfg$-vector fan.

\begin{definition}\label{def:partitionbythick}
    Let $\calC$ be as in \Cref{setting}.
    We define the \textit{thick partition} on the $\bfg$-vector fan $\partfan{\calC}$ by the following equivalence relation:
    \[ \boldC(U) \sim \boldC(V) \quad \Longleftrightarrow \quad \thick(U) = \thick(V).\]
   We denote the thick partition of $\calC$ by $\mathfrak{P}_{\thick}$.
\end{definition}

\begin{proposition}\label{prop:STTWgeneralisation}
    Let $\calC$ be as in \Cref{setting}.
    Then the thick partition $\mathfrak{P}_{\thick}$ is an admissible partition of $\partfan{\calC}$.
    This is to say that for all presilting objects $U$ and $V$ in $\calC$ such that $\thick(U)=\thick(V)$ {the following hold:}
    \begin{enumerate}
        \item\label{prop:STTWgeneralisation_1} $\boldC(U)^\perp = \boldC(V)^\perp $;
         \item\label{prop:STTWgeneralisation_2} $\pi_{{\boldC(U)}}(\mathrm{star}({\boldC(U)})) = \pi_{{\boldC(V)}}(\mathrm{star}({\boldC(V)})) $;
    \end{enumerate}
    \begin{enumerate}
    \setcounter{enumi}{2}
        \item\label{prop:STTWgeneralisation_3} If $\pi_{\boldC(U)}(\boldC(U')) = \pi_{\boldC(V)}(\boldC(V'))$ for some $\boldC(U') \in \mathrm{star}(\boldC(U))$ and some $\boldC(V') \in \mathrm{star}(\boldC(V))$, then $\boldC(U') \sim_{\mathfrak{P}_{\thick}} \boldC(V')$.
      \end{enumerate}
    The pair $(\partfan{\calC},\mathfrak{P}_{\thick})$ is thus a partitioned fan.
\end{proposition}
\begin{proof}
Without loss of generality, assume that $U$ and $V$ are basic and let $n$ be the number of indecomposable summands of $T$. Fix the basis of $\K_0(\add_{\calC} (T))_\R$ which is induced by the indecomposable direct summands of $T$. Consider the canonical surjective linear map
\[ \psi: \K_0(\add_{\calC}(T))_\R \twoheadrightarrow \K_0(\add_{\calC/\thick(U)}(T))_\R = \K_0(\add_{\calC/\thick(V)}(T))_\R \]
sending $[T_i] {\otimes 1} \mapsto [T_i] {\otimes 1} $. Let $\{U_i\}_{i=1}^r$ be the indecomposable direct summands of $U$. By \cref{lem:IsFan}\eqref{lem:IsFan_linindep}, the elements $\ind_T(U_i)$ are linearly independent, so $\dim_\R \spann( \boldC(U))= r$.

On the other hand, the isomorphism $\calC/\thick(U) \cong \calZ_U/[U]$ in \cref{prop:siltred}\eqref{prop:siltred_equiv} gives
\[ \dim_\R \K_0(\add_{\calC/\thick(U)}(T))_\R = \dim_\R \K_0(\add_{\calZ_U/[U]}(U^+))_\R =  \dim_\R \Ksp_0(\add(U^+/U))_\R = |U^+| - |U| = n-r. \]
It follows that $\dim_\R (\ker (\psi))=r$. From the above, we also get
\[ n- |V| =  |V^+|- |V| = \dim_\R \K(\add_{\calC/\thick(V)} (T))_\R = n-r,\]
whence $|V|=r$. For all $U_i \in \add(U)$ and every conflation $\begin{tikzcd} T_1^{U_i} \arrow[r, tail] &T_0^{U_i} \arrow[r, two heads] & U_i \end{tikzcd}$ with $T_0^{U_i}, T_1^{U_i} \in \add(T)$, there is an isomorphism $T_1^{U_i} \simeq T_0^{U_i}$ in $\calC/\thick(U)$. It follows that $\psi(\boldC(U)) = 0$ and analogously $\psi(\boldC(V)) = 0$. For dimension reasons, we must have that $\spann(\boldC(U)) = \ker(\psi) =  \spann(\boldC(V))$ in $\K_0(\add_{\calC}(T))_\R$, which directly implies $\boldC(U)^\perp = \boldC(V)^\perp$, proving \eqref{prop:STTWgeneralisation_1}.

Let now $\sigma \in \starr(\boldC(U)) \subseteq \partfan{\calC}$. By definition, there exists a presilting object $X$ in $\calC$ such that $\sigma = \boldC(U \oplus X)$. From \cref{prop:siltred}\eqref{prop:siltred_bij} and the equality $\calC/\thick(U) = \calC/\thick(V)$, it follows that there exists a presilting object $Y$ in $\calC$ such that $V \oplus Y$ is presilting and such that $X \simeq Y$ in $\calC/\thick(U) = \calC/\thick(V)$. In particular, this defines a bijection between cones in $\starr(\boldC(U))$ and $\starr(\boldC(V))$. Crucially, because $X \simeq Y$ in $\widehat{\calC} \coloneqq \calC/\thick (U) = \calC/\thick(V)$, we deduce the following consequences of the fact that $\ind_{T}^{\widehat{\calC} } X = \ind_T^{\widehat{\calC}} Y$:
\begin{equation*}
\begin{aligned}
	{} & \Rightarrow \boldC_{\widehat{\calC} }(X) = \boldC_{\widehat{\calC} }(Y) & \text{(in $\K_0(\add_{\widehat{\calC} }(T))_\R$)} \\
	& \Rightarrow  \psi^{-1}(\boldC_{\widehat{\calC} }(X)) = \psi^{-1}( \boldC_{\widehat{\calC} }(Y)) &  \text{(in $\K_0(\add_{\calC}(T))_\R$)} \\
	& \Rightarrow \spann_{\geq 0} ( \boldC(X), \psi^{-1}(0)) = \spann_{\geq 0} ( \boldC(Y), \psi^{-1}(0)) & \text{(in $\K_0(\add_{\calC}(T))_\R$)} \\
	& \Rightarrow \spann_{\geq 0} ( \boldC(X), \spann(\boldC(U))) = \spann_{\geq 0} ( \boldC(Y), \spann(\boldC(Y))) & \text{(in $\K_0(\add_{\calC}(T))_\R$)} \\
	& \Rightarrow \pi_{\boldC(U)}(\spann_{\geq 0} ( \boldC(X), \spann(\boldC(U)))) = \pi_{\boldC(V)}(\spann_{\geq 0} ( \boldC(Y), \spann(\boldC(Y)))) & \text{(in $\K_0(\add_{\calC}(T))_\R^*$)} \\
	& \Rightarrow \pi_{\boldC(U)}(\spann_{\geq 0} ( \boldC(X))) = \pi_{\boldC(V)}(\spann_{\geq 0} ( \boldC(Y))) & \text{(in $\K_0(\add_{\calC}(T))_\R^*$)} \\
	& \Rightarrow \pi_{\boldC(U)}( \boldC(X) ) = \pi_{\boldC(V)}( \boldC(Y)) & \text{(in $\K_0(\add_{\calC}(T))_\R^*$)} \\
\end{aligned}
\end{equation*}
The claim in \eqref{prop:STTWgeneralisation_2} follows from this.

To prove \eqref{prop:STTWgeneralisation_3}, we may assume without loss of generality that $U'$ and $V'$ are presilting in $\calC$. Then they are also presilting in $\widehat{\calC}$ by \Cref{prop:siltred}\eqref{prop:siltred_bij}. By \cref{thm:IsFan}, there is a fan $\partfan{\widehat{\calC}}$ which is isomorphic to $\partfan{\calZ_U/[U]}$ and $\partfan{\calZ_V/[V]}$ by \Cref{prop:siltred}\eqref{prop:siltred_equiv} and to $\pi_{\boldC(U)}(\starr(\boldC(U)))$ by \cref{lem:ZUfan}\eqref{lem:ZUfan2}. Therefore, the assumption $\pi_{\boldC(U)}(\boldC(U')) = \pi_{\boldC(V)}(\boldC(V'))$ implies that $U'$ and $V'$ are isomorphic in $\widehat{\calC}$. Therefore $\thick_{\widehat{\calC}}(U') = \thick_{\widehat{\calC}}(V')$ holds, which implies $\thick_{\calC}(U')=\thick_{\calC}(V')$ using \Cref{lem:Ver2.3.1}. This implies the desired result.
\end{proof}

With the help of the thick partition of the $\bfg$-vector fan, we are now able to introduce the category on which we will focus our attention for the remaining part of this paper.

\begin{definitiontheorem}\label{defthm:piccat}
\,
    \begin{enumerate}
        \item\label{defthm:piccat_general}(\cite[{Def.} 3.4]{Kai23}). Let $N$ be a finite lattice and let $(\Sfan, \mathfrak{P})$ be a partitioned fan in $N_{\R}$.
        	The \textit{category of {the partitioned fan} $(\Sfan, \mathfrak{P})$} is denoted by $\catof(\Sfan,\mathfrak{P})$ and defined as follows:
	\begin{enumerate}
		\item The objects in $\catof(\Sfan,\mathfrak{P})$ are the equivalence classes of the partition {$\mathfrak{P}$}. We denote the equivalence class containing $\sigma$ by $\eqclass{\sigma}$.
		\item 
        Denote a relation $\sigma\leq \rho$ in the poset of cones {in $\Sfan$ under inclusion} by the symbol $\binom{\rho}{\sigma}$.
        Given $\Sfan$-cones $\sigma$ and $\rho$, let $w(\sigma,\rho)$ be the singleton set $\{\binom{\rho}{\sigma} \}$ if $\sigma\leq \rho$ and the empty set otherwise. Then the set of morphisms $\catof(\Sfan,\mathfrak{P})(\eqclass{\sigma},\eqclass{\rho})$ is given by the equivalence classes of objects in the set
		\begin{equation}\label{eq:piccat_mor}
			\bigcup\limits_{\sigma_i\in\eqclass{\sigma}, \rho_i\in \eqclass{\rho}} w(\sigma_i,\rho_i),
		\end{equation}
		where $\binom{\rho_1}{\sigma_1}$ is identified with $\binom{\rho_2}{\sigma_2}$ provided that $\pi_{\sigma_1}(\rho_1)=\pi_{\sigma_2}(\rho_2)$. We denote the equivalence class containing $\binom{\rho}{\sigma}$ by $\sqbinom{\rho}{\sigma}$.
		\item Composition in $\catof(\Sfan,\mathfrak{P})$ is defined by $\sqbinom{\rho}{\kappa} \circ \sqbinom{\kappa}{\sigma} \coloneqq \sqbinom{\rho}{\sigma}$.
	\end{enumerate}
    \item{(\cite[{Prop.} 3.8]{Kai23})} The category $\catof(\Sfan,\mathfrak{P})$ of the partitioned fan $(\Sfan,\mathfrak{P})$ is a well-defined category.
        \item\label{defthm:piccat_0Aus} Let $\calC$ be as in \Cref{setting}. The \textit{picture category} $\Wfrak(\calC)$ is defined to be the category of the partitioned fan $(\partfan{\calC} , \Pfrak_{\thick})$
        We simplify the notation for objects and morphisms in $\Wfrak(\calC)$ by setting 
        $\eqclass{U} \coloneqq \eqclass{\boldC(U)}$  and $ \sqbinom{V}{U}\coloneqq \sqbinom{\boldC(V)}{\boldC(U)}$, where $U\in\presilt(\calC)$ and $V\in\presilt_U(\calC)$.
    \end{enumerate}
\end{definitiontheorem}

Let $\calB$ be a small category. Given a morphism $f$ in $\calB$, we define the \textit{factorisation category} of $f$ as the category $\mathrm{faq}(f)$, where the objects are diagrams $(X  \xrightarrow{f_1} Z \xrightarrow{f_2} Y)$ in $\calC$ with $f_2\circ f_1=f$, and the morphisms are given by commutative diagrams of the form 
\begin{equation*}
    \begin{tikzcd}[row sep=0.025em]
 & Z\arrow[dd]\arrow[rd,"f_2"] & \\
X\arrow[ru,"f_1"]\arrow[rd,"f'_1"']& & Y.\\
 & Z' \arrow[ru,"f'_2"']&
\end{tikzcd}
\end{equation*}
The category $\calB$ is equipped with the structure of {a} \textit{cubical category} if the following axioms hold \cite[§3.1]{Igu14}.
\begin{enumerate}[label=($\mathbf{C \arabic*}$)]
    \item\label{cubical1} 
    Let $\mathcal{M}or(\calB)$ denote the set of morphisms in $\calB$. There is a function
    \begin{equation*}
        \mathrm{rk}\colon \mathcal{M}or(\calB) \to \Z_{\geq 0}
    \end{equation*}
    such that $\mathrm{rk}(g\circ f) = \mathrm{rk}(f)+\mathrm{rk}(g)$ for all composable morphisms $f,g \in \calC$. 
    The non-negative integer $\mathrm{rk}(f)$ is called the \textit{rank} of the morphism $f$ in $\calB$. 
    \item\label{cubical2} If $\mathrm{rk}(f) = d$, then there is an equivalence of categories $\mathrm{faq}(f) \simeq P(d)$, where $P(d)$ is the power poset of a set containing $d$ elements (taking the form of a $d$-dimensional hypercube).
    \item\label{cubical3} The \textit{forgetful functor} $\mathrm{faq}(f) \to \calB$, sending $(X  \xrightarrow{f_1} Z \xrightarrow{f_2} Y)$ to $Z$, is faithful and injective on objects.  
    \item\label{cubical4}
    By \ref{cubical2} and \ref{cubical3} above, every morphism of rank $d$ has $d$ distinct first factors and $d$ distinct last factors.
    Every morphism of rank $d$ is determined by its $d$ first factors. 
    \item\label{cubical5} 	Every morphism of rank $d$ is determined by its $d$ last factors.
\end{enumerate}

As we have shown in \cref{prop:STTWgeneralisation}, the thick partition is an admissible partition of $\partfan{\calC}$, whence the following result describes the properties of the picture category. 
\begin{theorem}[{\cite[{Thm.} 3.16]{Kai23}}]
    Let $\calC$ be as in \cref{setting}. The category $\catof(\partfan{\calC}, \Pfrak_{\thick})$ is a cubical category, where the rank of a morphism $\sqbinom{V}{U}$ {represented by two basic presilting objects} is given by the well-defined non-negative integer $|V|-|U|$, namely the difference of the number of isomorphism classes of indecomposable direct summands.
\end{theorem}

We write $\Wfrak(\Lambda) = \catof(\partfan{\Lambda}, \Pfrak_{\thick})$.
From the definition of cubical categories, we obtain the following result which establishes a new property also for $\tau$-cluster morphism and picture categories.

\begin{lemma}\label{lem:monoepi}
    Let $\calB$ be a small category satisfying \ref{cubical3} above. Then every morphism in $\calB$ is both a monomorphism and an epimorphism. In particular, if $\calC$ is as in \Cref{setting}, then every morphism in the picture category $\tcmc{\calC}$ is both a monomorphism and an epimorphism. 
\end{lemma}
\begin{proof}
    Suppose that there exists a morphism $X\xrightarrow{f} Y$ in $\calB$ {which} is {not} a monomorphism. There then exists a diagram in $\calB$ of the form
    \begin{equation}\label{eq:monoepi_nonmono}
        \begin{tikzcd}
Z \arrow[r, "g_1", shift left] \arrow[r, "g_2"', shift right] & X \arrow[r, "f"] & Y,
\end{tikzcd}
    \end{equation}
    where $f\circ g_1=f\circ g_2$, {but $g_1 \neq g_2$}. Let $h=f\circ g_1=f\circ g_2$, and consider the two objects $(Z \xrightarrow{g_1} X \xrightarrow{f} Y)$ and $(Z \xrightarrow{g_2} X \xrightarrow{f} Y)$ in the factorisation category $\mathrm{faq}(h)$. They are distinct, yet they are sent to the same object by the forgetful functor $\mathrm{faq}(h) \to \calB$ defined in the axiom \ref{cubical3} above.
    This forgetful functor is thus not injective on objects, so we have shown that \ref{cubical3} cannot hold when there exists a {non-}monomorphism in {$\calB$}.
    A dual argument shows that every morphism in $\calB$ is an epimorphism, whence the proof is complete.
\end{proof}

\begin{example}\label{ex:PicCat}
To give examples of picture categories of 0-Auslander extriangulated $k$-categories, we consider the examples {of fans} in \Cref{ex:fans} {and their thick partitions}.
    \begin{enumerate}
        \item Consider the preprojective $k$-algebra of type $A_2$, denoted $\Pi(A_2)$ (see \Cref{ex:fans}\eqref{ex:fans_Preproj}). The thick partition of the $\bfg$-vector fan $\partfan{\Pi(A_2)}$ (see \Cref{fig:fanOfPreproj}) identifies all maximal (two-dimensional) cones {because $\thick(V) = \calC$ for all silting objects $V$.} {It} also imposes that $\boldC(P_1)\sim \boldC(\Sigma P_1)$ and that $\boldC(P_2)\sim \boldC(\Sigma P_2)$ {because these fit into a conflation with 0 as a middle term}. In \Cref{fig:PicCatOfPreproj}, we display the picture category of $\Ktwo(\proj \Pi(A_2))$, or equivalently the $\tau$-cluster morphism category of $\Pi(A_2)$ {(see \cref{prop:samePartition})}.
        \begin{figure}[ht!]
    \begin{tikzcd}[column sep=1.5em,
    execute at end picture={
    \scoped[on background layer]
    execute at end picture={
    \scoped[on background layer]
    \fill[pattern=north west lines, pattern color=friendly_blue] (p2s2.north west) -- (p1p2.north east)  -- (p1p2.south east) -- (p2s2.south west) -- cycle;},
    execute at end picture={
    \scoped[on background layer]
    \fill[pattern=north west lines, pattern color=friendly_blue] (Sp1Sp2.north west) -- (s1Sp2.north east)  -- (s1Sp2.south east) -- (Sp1Sp2.south west) -- cycle;},
    execute at end picture={
    \scoped[on background layer]
    \fill[pattern=north east lines, pattern color=friendly_orange] (s2Sp1.north west) -- (Sp1Sp2.south west)  -- (Sp1Sp2.south east) -- (s2Sp1.north east) -- cycle;},
    execute at end picture={
    \scoped[on background layer]
    \fill[pattern=north east lines, pattern color=friendly_orange] (p1p2.north west) -- (p1s1.south west)  -- (p1s1.south east) -- (p1p2.north east) -- cycle;},
    \fill[color=friendly_green, opacity=0.5] (s1Sp2.north west) -- (s1Sp2.north east)  -- (s1Sp2.south east) -- (s1Sp2.south west) -- cycle;},
    execute at end picture={
    \scoped[on background layer]
    \fill[color=friendly_green, opacity=0.5] (Sp1Sp2.north west) -- (Sp1Sp2.north east)  -- (Sp1Sp2.south east) -- (Sp1Sp2.south west) -- cycle;},
    execute at end picture={
    \scoped[on background layer]
    \fill[color=friendly_green, opacity=0.5] (s2Sp1.north west) -- (s2Sp1.north east)  -- (s2Sp1.south east) -- (s2Sp1.south west) -- cycle;},
    execute at end picture={
    \scoped[on background layer]
    \fill[color=friendly_green, opacity=0.5] (p2s2.north west) -- (p2s2.north east)  -- (p2s2.south east) -- (p2s2.south west) -- cycle;},
    execute at end picture={
    \scoped[on background layer]
    \fill[color=friendly_green, opacity=0.5] (p1p2.north west) -- (p1p2.north east)  -- (p1p2.south east) -- (p1p2.south west) -- cycle;},
    execute at end picture={
    \scoped[on background layer]
    \fill[color=friendly_green, opacity=0.5] (p1s1.north west) -- (p1s1.north east)  -- (p1s1.south east) -- (p1s1.south west) -- cycle;},
    ]
                                                                                                                                                                        &                                                                                                                                    & |[alias=p2s2]|\eqclass{P_2\oplus S_2} &  & |[alias=p2]|\eqclass{P_2} \arrow[rrrr, "{\tiny\sqbinom{P_1\oplus P_2}{P_2}}"'] \arrow[ll, "{\tiny\sqbinom{P_2\oplus S_2}{P_2}}" ]                                                                                                                                                                                                                                                                                                                                                                                                                                                                                                                                                         &  &                                &                                                                                                                                    & |[alias=p1p2]|\eqclass{P_1\oplus P_2}                                                                                                       \\
                                                                                                                                                                        & |[alias=s2]|\eqclass{S_2} \arrow[ru, "{\tiny\sqbinom{P_2\oplus S_2}{S_2}}"'] \arrow[ld, "{\tiny\sqbinom{S_2\oplus \Sigma P_1}{S_2}}"] &                         &  &                                                                                                                                                                                                                                                                                                                                                                                                                                                                                                                                                                                                                                                                                       &  &                                &                                                                                                                                    &                                                                                                                               \\
|[alias=s2Sp1]|\eqclass{S_2\oplus \Sigma P_1}                                                                                                                                          &                                                                                                                                    &                         &  &                                                                                                                                                                                                                                                                                                                                                                                                                                                                                                                                                                                                                                                                                       &  &                                &                                                                                                                                    &                                                                                                                               \\
                                                                                                                                                                        &                                                                                                                                    &                         &  &                                                                                                                                                                                                                                                                                                                                                                                                                                                                                                                                                                                                                                                                                       &  &                                &                                                                                                                                    &                                                                                                                               \\
|[alias=Sp1]|\eqclass{\Sigma P_1} \arrow[uu, "{\tiny\sqbinom{S_2\oplus \Sigma P_1}{\Sigma P_1}}" description] \arrow[dddd, "{\tiny\sqbinom{\Sigma P_1\oplus \Sigma P_2}{\Sigma P_1}}" description] &                                                                                                                                    &                         &  & |[alias=o]|\eqclass{0} \arrow[rrrr, "{\tiny\sqbinom{P_1}{0}}" description] \arrow[rrrddd, "{\tiny\sqbinom{S_1}{0}}" description] \arrow[dddd, "{\tiny\sqbinom{\Sigma P_2}{0}}" description] \arrow[llll, "{\tiny\sqbinom{\Sigma P_1}{0}}" description] \arrow[uuuu, "{\tiny\sqbinom{P_2}{0}}" description] \arrow[llluuu, "{\tiny\sqbinom{S_2}{0}}" description] \arrow[lluuuu, "{\tiny\sqbinom{P_2\oplus S_2}{0}}" description,color=gray,opacity=0.4] \arrow[rrrruuuu, "{\tiny\sqbinom{P_1\oplus P_2}{0}}" description,color=gray,opacity=0.4] \arrow[rrrrdd, "{\tiny\sqbinom{P_1\oplus S_1}{0}}" description,color=gray,opacity=0.4] \arrow[rrdddd, "{\tiny\sqbinom{S_1\oplus \Sigma P_2}{0}}" description,color=gray,opacity=0.4] \arrow[lllldddd, "{\tiny\sqbinom{\Sigma P_1\oplus \Sigma P_2}{0}}" description,color=gray,opacity=0.4] \arrow[lllluu, "{\tiny\sqbinom{S_2\oplus \Sigma P_1}{0}}" description,color=gray,opacity=0.4] &  &                                &                                                                                                                                    & |[alias=p1]|\eqclass{P_1} \arrow[uuuu, "{\tiny\sqbinom{P_1\oplus P_2}{P_1}}" description] \arrow[dd, "{\tiny\sqbinom{P_1\oplus S_1}{P_1}}" description] \\
                                                                                                                                                                        &                                                                                                                                    &                         &  &                                                                                                                                                                                                                                                                                                                                                                                                                                                                                                                                                                                                                                                                                       &  &                                &                                                                                                                                    &                                                                                                                               \\
                                                                                                                                                                        &                                                                                                                                    &                         &  &                                                                                                                                                                                                                                                                                                                                                                                                                                                                                                                                                                                                                                                                                       &  &                                &                                                                                                                                    & |[alias=p1s1]|\eqclass{P_1\oplus S_1}                                                                                                       \\
                                                                                                                                                                        &                                                                                                                                    &                         &  &                                                                                                                                                                                                                                                                                                                                                                                                                                                                                                                                                                                                                                                                                       &  &                                & |[alias=s1]|\eqclass{S_1} \arrow[ld, "{\tiny\sqbinom{S_1\oplus \Sigma P_2}{S_1}}"' ] \arrow[ru, "{\tiny\sqbinom{P_1\oplus S_1}{S_1}}" ] &                                                                                                                               \\
|[alias=Sp1Sp2]|\eqclass{\Sigma P_1\!\!\oplus\!\! \Sigma P_2}                                                                                                                                   &                                                                                                                                    &                         &  & |[alias=Sp2]|\eqclass{\Sigma P_2} \arrow[llll, "{\tiny\sqbinom{\Sigma P_1\oplus \Sigma P_2}{\Sigma P_2}}"'] \arrow[rr, "{\tiny\sqbinom{S_1\oplus\Sigma P_2}{\Sigma P_2}}"]                                                                                                                                                                                                                                                                                                                                                                                                                                                                                                                &  & |[alias=s1Sp2]|\eqclass{S_1\oplus \Sigma P_2} &                                                                                                                                    &                                                                                                                              
\end{tikzcd}
    \caption{The picture category of the preprojective $k$-algebra of Dynkin type $A_2$. The faint morphisms are of (maximal) rank 2, and all other morphisms shown are of rank 1. The six \textcolor{friendly_green}{green} objects are those given by cones of silting objects, so these are to be identified.
    The \textcolor{friendly_blue}{blue} parts are to be identified, and so are the \textcolor{friendly_orange}{orange} parts.}
    \label{fig:PicCatOfPreproj}
    \end{figure}
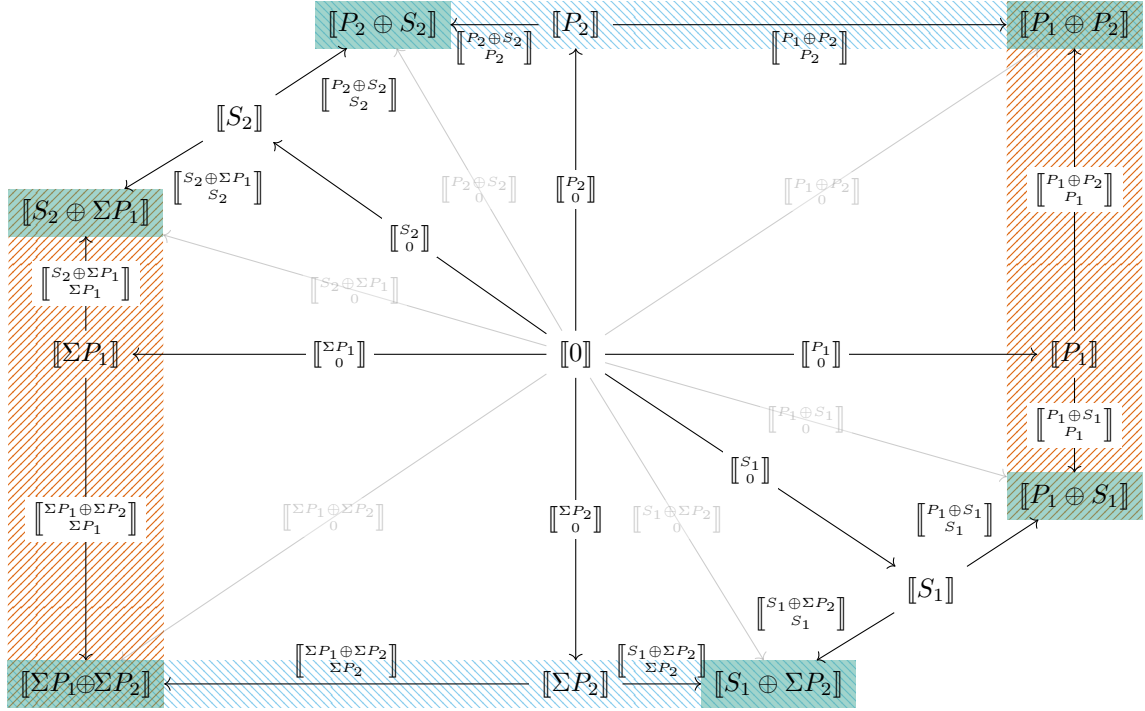
        \item\label{ex:PicCat_A2} In \Cref{ex:fans}\eqref{ex:PicCat_A2}, we considered the path $k$-algebra of a Dynkin quiver of type $A_2$, denoted $\Lambda_{2}$, and regarded the module {category} $\mods(\Lambda_2)$ as a 0-Auslander extriangulated $k$-category. The $\bfg$-vector fan 
 $\partfan{\mods(\Lambda_2)}$ of $\mods(\Lambda_2)$ was then displayed in \Cref{fig:fanOfA2}. The thick partition of $\partfan{\mods(\Lambda_2)}$ only identifies the two maximal (two-dimensional) cones.
        The picture category of $\mods(\Lambda_2)$ is displayed in \Cref{fig:PicCatA2}.
        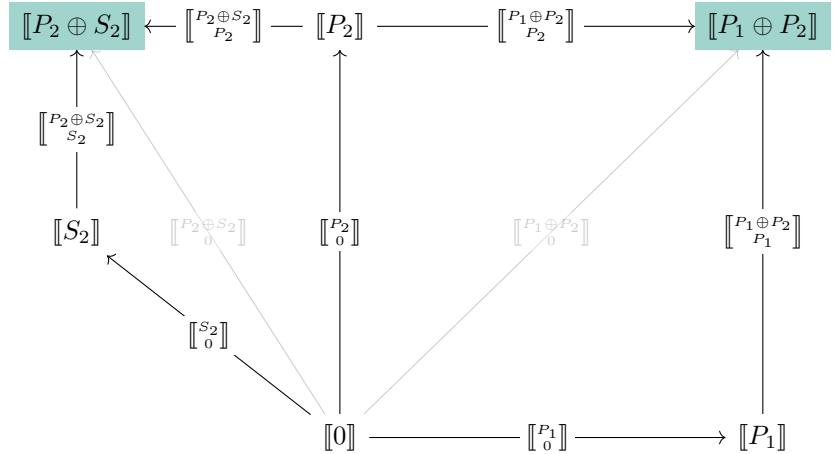
\begin{figure}[ht!] 
    \begin{tikzcd}[row sep=6em,column sep=6em,
    execute at end picture={
    \scoped[on background layer]
    \fill[color=friendly_green,opacity=0.5] (p2s2.south east) -- (p2s2.south west)  -- (p2s2.north west) -- (p2s2.north east) -- cycle;},
    execute at end picture={
    \scoped[on background layer]
    \fill[color=friendly_green,opacity=0.5] (p1p2.south east) -- (p1p2.south west)  -- (p1p2.north west) -- (p1p2.north east) -- cycle;},
    ]
|[alias=p2s2]|{\eqclass{P_2\oplus S_2}}                                             & |[alias=p2]|\eqclass{P_2} \arrow[rr, "{\tiny \sqbinom{P_1\oplus P_2}{P_2}}" description] \arrow[l, "{\tiny \sqbinom{P_2\oplus S_2}{P_2}}" description]                                                                                                                                                &  & |[alias=p1p2]|{\eqclass{P_1\oplus P_2}}  \\
|[alias=s2]|\eqclass{S_2} \arrow[u, "{\tiny \sqbinom{P_2\oplus S_2}{S_2}}" description] &                                                                                                                                                                                                                                                                           &  &                          \\
                                                                    & |[alias=o]|\eqclass{0} \arrow[uu, "{\tiny \sqbinom{P_2}{0}}" description] \arrow[rr, "{\tiny \sqbinom{P_1}{0}}" description] \arrow[lu, "{\tiny \sqbinom{S_2}{0}}" description] \arrow[rruu, "{\tiny {\sqbinom{P_1\oplus P_2}{0}}}" description, color=gray,opacity=0.4] \arrow[luu, "{\tiny {\sqbinom{P_2\oplus S_2}{0}}}" description, color=gray,opacity=0.4] &  & |[alias=p1]|\eqclass{P_1} \arrow[uu,"{\tiny \sqbinom{P_1\oplus P_2}{P_1}}" description]
\end{tikzcd}
    \caption{The picture category of $\mods(\Lambda_2)$, where $\Lambda_{2}$ is the path $k$-algebra of a Dynkin quiver of type $A_2$. The faint morphisms are of (maximal) rank 2, and all other morphisms shown are of rank 1.
    The \textcolor{friendly_green}{green} objects (in the upper corners) are to be identified.}
    \label{fig:PicCatA2}
    \end{figure}
    \item\label{ex:PicCatSubA3} Consider the 0-Auslander extriangulated $k$-category $\calC$ defined in \Cref{ex:fans}\eqref{ex:fans_subA3}. Similarly to the example in \eqref{ex:PicCat_A2} above, the thick partition of the $\bfg$-vector fan $\partfan{\calC}$ only identifies the two maximal (three-dimensional) cones.
        The picture category of $\calC$ is displayed in \Cref{fig:PicCatSubA3}.
    \begin{figure}[ht!]
        \begin{tikzcd}[scale=0.2,row sep=6em, column sep=0.5em,/tikz/column 4/.style={column sep=2em},/tikz/column 5/.style={column sep=0.0em},font=\Large,
	execute at end picture={
    \scoped[on background layer]
    \fill[color=friendly_green, opacity=0.5] (in1n2.south east) -- (in1n2.south west)  -- (in1n2.north west) -- (in1n2.north east) -- cycle;},
    execute at end picture={
    \scoped[on background layer]
    \fill[color=friendly_green, opacity=0.5] (pn1n2.south east) -- (pn1n2.south west)  -- (pn1n2.north west) -- (pn1n2.north east) -- cycle;}]
                                 & |[alias=in1n2]|\eqclass{I\oplus N_1\oplus N_2} \arrow[<-,ld, "{\tiny \sqbinom{I\oplus N_1\oplus N_2}{I\oplus N_2}}"' description] \arrow[<-,dd, "{\tiny \sqbinom{I\oplus N_1\oplus N_2}{I\oplus N_1}}" description, near end] &                                                                                                                                                    & \eqclass{N_1\oplus N_2} \arrow[ll, "{\tiny \sqbinom{I\oplus N_1\oplus N_2}{N_1\oplus N_2}}"'] \arrow[rr, "{\tiny \sqbinom{P\oplus N_1\oplus N_2}{N_1\oplus N_2}}"] \arrow[<-,ld, "{\tiny \sqbinom{N_1\oplus N_2}{N_2}}"' description,end anchor={90}] \arrow[<-,dd, "{\tiny \sqbinom{N_1\oplus N_2}{N_1}}" description, near end, xshift=0.5em]       &                                  &  |[alias=pn1n2]| \eqclass{P\oplus N_1\oplus N_2} \arrow[<-,ld, "{\tiny \sqbinom{P\oplus N_1\oplus N_2}{P\oplus N_2}}"' description,start anchor={200}] \arrow[<-,dd, "{\tiny \sqbinom{P\oplus N_1\oplus N_2}{P\oplus N_1}}" description, near end, xshift=-1.25em] \\
|[alias=in2]| \eqclass{I\oplus N_2} \arrow[<-,dd,"{\tiny \sqbinom{I\oplus N_2}{I}}" description,crossing over, near start] &                                                                                            & \eqclass{N_2}  \arrow[rr, "{\tiny \sqbinom{P\oplus N_2}{N_2}}" description,crossing over, near end] \arrow[ll, "{\tiny \sqbinom{I\oplus N_2}{N_2}}"' description,crossing over, near start]        &                                                                                                                                             & |[alias=pn2]|\eqclass{P\oplus N_2} &                                                                                            \\
                                 & |[alias=in1]|\eqclass{I\oplus N_1} \arrow[<-,ld, "{\tiny \sqbinom{I\oplus N_1}{I}}" description]           &                                                                                                                                                    & \eqclass{N_1} \arrow[ll, "{\tiny \sqbinom{I\oplus N_1}{N_1}}"' description, near start] \arrow[rr, "{\tiny \sqbinom{P\oplus N_1}{N_1}}" description, near start] \arrow[<-,ld, "{\tiny \sqbinom{N_1}{0}}" description,start anchor={300}] &                                  &  |[alias=pn1]|\eqclass{P\oplus N_1} \arrow[<-,ld, "{\tiny \sqbinom{P\oplus N_1}{P}}" description, start anchor={200}]           \\
|[alias=i]| \eqclass{I}  &                                                                                            & \eqclass{0} \arrow[uu, "{\tiny \sqbinom{N_2}{0}}" description,crossing over, near end]\arrow[rr, "{\tiny \sqbinom{P}{0}}"' description] \arrow[ll, "{\tiny \sqbinom{I}{0}}" description] &                                                                                                                                             & |[alias=p]|{\displaystyle \eqclass{P}} \arrow[uu,crossing over, near end,"{\tiny \sqbinom{P\oplus N_2}{P}}" description] &                                                                          
\end{tikzcd}
    \caption{The picture category of $\calC$, as defined in \Cref{ex:fans}\eqref{ex:fans_subA3} (see also \Cref{ex:PicCat}\eqref{ex:PicCatSubA3}). Only the morphisms of rank 1 are displayed. All subdiagrams commute. The \textcolor{friendly_green}{green} objects (in the upper corners at the back) are to be identified.}
    \label{fig:PicCatSubA3}
    \end{figure}
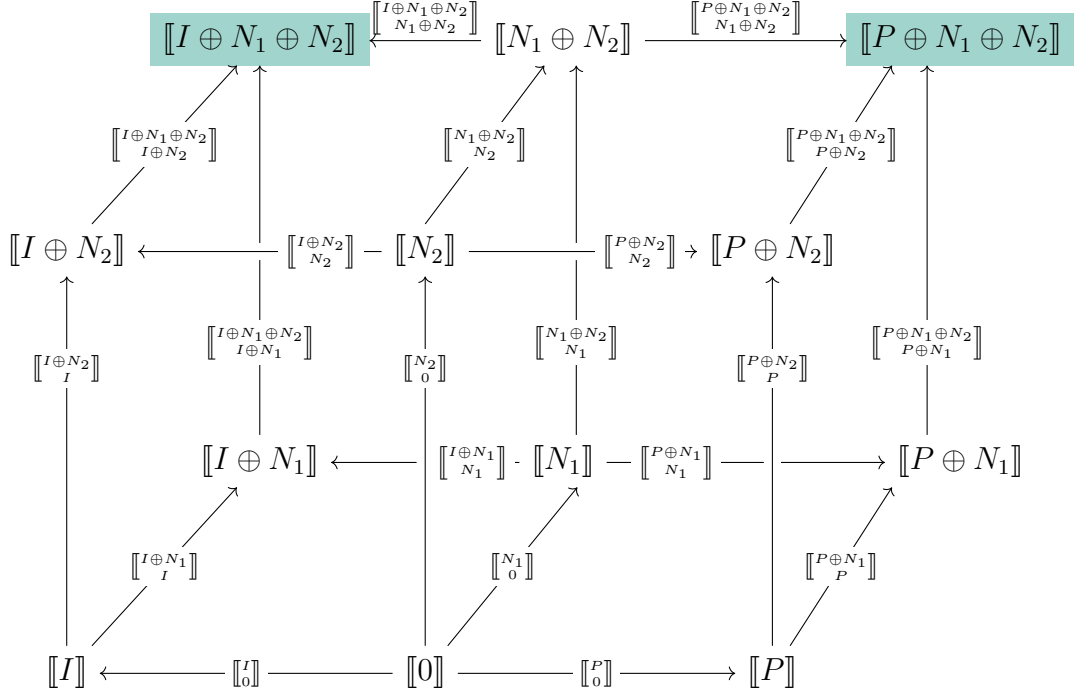
    \end{enumerate}
\end{example}

In order to illustrate how our notion of picture category {relates to} constructions in the literature, we will need to define semistability for 0-Auslander extriangulated $k$-categories, taking inspiration from M. Garcia \cite{Gar24}. 
 \Cref{prop:Monica} below is a generalisation of Garcia's work  \cite[{Prop.} 3.14]{Gar23} \cite[Prop. 4.21]{Gar24}, as well as Yurikusa's \cite{Yur}. 

\begin{definition} 
Let $\calC$ and $T$ be as in \Cref{setting}, and let  $L \in\mods(\Lambda)$. Consider the functor $\overline{(-)} \colon \calC \to \mods(\Lambda)$ of \Cref{lem:GNP23.3.11}. 
If an object $X\in\calC$ admits an $\E$-triangle
\begin{equation}\label{eq:confX}
        \begin{tikzcd}
            T_1 \arrow[r, tail, "{x}"] & T_0 \arrow[r,two heads] & X \arrow[r, dashed] & {},
        \end{tikzcd}
\end{equation}
with $T_0$ and $T_1$ projective, we say that $X$ is \textit{$L$-semistable} if the induced $k$-homomorphism
\begin{equation*}
    \Hom_{\Lambda}\big(\overline{T_0},L\big) \xrightarrow{{\Hom_\Lambda(\overline{x},L)}} \Hom_{\Lambda}\big(\overline{T_1},L \big)
\end{equation*}
is an isomorphism of $k$-vector spaces (cf. \cite[Def. 3.12]{Gar23}). Using a long-exact sequence argument (see \Cref{lem:LES}) or otherwise, one shows that the semistability property is independent of the choice of $\E$-triangle for $X$. 
For a subcategory $\mathcal{X}$ of $\calC$, let $\calW_{\calC}(\mathcal{X})$ denote the full subcategory of $\mods(\Lambda)$ {consisting of} $\Lambda$-modules $L$ such that all objects in $\mathcal{X}$ are $L$-semistable. Given a full subcategory $\mathcal{Y}$ of $\mods(\Lambda)$, let $\calT_{\calC}(\mathcal{Y})$ denote the full subcategory of $\calC$ spanned by the objects that are $L$-semistable for all $L\in\mathcal{Y}$.
\end{definition}

Let $\Lambda$ be a finite-dimensional $k$-algebra. A full subcategory of $\mods(\Lambda)$ is \textit{wide} if it is closed under extensions, kernels, and cokernels.
Given a $\tau$-rigid pair $(M,P)$ in $\mods(\Lambda)$, the \textit{$\tau$-perpendicular category} of $(M,P)$ is defined by $J(M,P)\coloneqq M^{\perp} \cap {^{\perp}(\tau M})\cap P^{\perp}$, where $\tau$ denotes the Auslander--Reiten translation, and ${^{\perp}(-)}$ and ${(-)^{\perp}}$ are Hom-perpendicular subcategories in $\mods(\Lambda)$. This is a \textit{wide subcategory} of $\mods(\Lambda)$ \cite[{Thm.} 4.12(a)]{DIRRT17}. {The} wide subcategories arising in this way are called \textit{$\tau$-perpendicular wide subcategories} of $\mods(\Lambda)$ \cite[{Def.} 3.2]{BH21}.

\begin{theorem}\label{prop:Monica}
    Let $\calC$, $T$, $N$ and $\Lambda$ be as in \Cref{setting}, and let $U$ be a presilting object in $\calC$.
    \begin{enumerate}
        \item\label{prop:Monica_W} If $\mathcal{X}$ is a full subcategory of $\calC$, then $\calW_{\calC}(\mathcal{X})$ is a wide subcategory of $\mods(\Lambda)$.
        \item\label{prop:Monica_T} If $\mathcal{Y}$ is a full subcategory of $\mods(\Lambda)$, then $\calT_{\calC}(\mathcal{Y})$ is a thick subcategory of $\calC$.
        \item\label{prop:Monica_Galois} Let $\fullsubc{\calC}$ denote the poset of full subcategories of $\calC$ {under inclusion} and {let} $\fullsubc \mods(\Lambda)$ denote the poset of full subcategories of $\mods(\Lambda)$ {under inclusion}. The maps 
    \begin{align*}
        \calW_{\calC}: \fullsubc {\calC} &\to \fullsubc \mods(\Lambda), \\
        \calT_{\calC}: \fullsubc \mods(\Lambda) &\to \fullsubc {\calC}
    \end{align*}
    form an antitone Galois connection, which is to say that these maps are inclusion-reversing and that the following holds for all $\mathcal{D}\in \fullsubc \calC$ and $\mathcal{V}\in \fullsubc \mods(\Lambda)$:
    \begin{equation}\label{eq:Monica_Galois}
       \mathcal{V}\subseteq \calW_{\calC}(\mathcal{D}) \Longleftrightarrow \mathcal{D}\subseteq \calT_{\calC}(\mathcal{V}).
    \end{equation}
    \item\label{prop:Monica_Wbis} If $\mathcal{X}$ is {a full subcategory of $\calC$}, then $\calW_{\calC}(\mathcal{X})=\calW_{\calC}(\thick(\mathcal{X}))$.
    \item\label{prop:Monica_W=J} We have that $\calW_{\calC}(\{U\})=J(\overline{U},\overline{\Omega I_U})$ as (wide) subcategories of $\mods(\Lambda)$, where $I_U$ is the maximal injective non-projective direct summand of $U$.  
    \item\label{prop:Monica_Z}
    We have that $ \calT_{\calC}\calW_{\calC}(\{U\}) = \thick(U\oplus N)$ as (thick) subcategories of $\calC$.
        \item\label{prop:Monica_maps} 
    The maps $\calW_{\calC}$ and $\calT_{\calC}$ restrict to mutually inverse inclusion-reversing bijections 
    \[\begin{tikzcd} \left\{ \begin{varwidth}{20em} \begin{center} thick subcategories of $\mathcal{C}$ generated by  presilting objects of the form $X \oplus N$ \end{center} \end{varwidth} \right\} \arrow[r, bend left, "\calW_{\calC}"] & \left\{ \begin{varwidth}{12em} \begin{center} $\tau$-perpendicular subcategories of $\mods(\Lambda)$ \end{center} \end{varwidth} \right\}. \arrow[l, bend left, "\calT_{\calC}"]  \end{tikzcd} \] 
    In particular, if $\calC$ is reduced, the maps $\calW_{\calC}$ and $\calT_{\calC}$ restrict to mutually inverse inclusion-reversing bijections between $\tau$-perpendicular wide subcategories of $\mods(\Lambda)$ and thick subcategories of $\calC$ of the form $\thick(X)$ for some presilting object $X\in \calC$.
    \end{enumerate}
\end{theorem}
\begin{proof}
    In order to prove \eqref{prop:Monica_W}, it suffices to consider the case where $\mathcal{X}$ contains a single object $X$. Indeed, in the general case we have
    \begin{equation}\label{eq:unionObj}
        \calW_{\calC}(\mathcal{X}) = \bigcap_{X\in\mathcal{X}} \calW_{\calC}(\{X\}),     
    \end{equation}
    so we may use the fact that the set of wide subcategories of $\mods(\Lambda)$ is closed under arbitrary intersection. {(cf. \cite[Proof of Prop. 3.14]{Gar23}).} 

    Let $X$ be an object in $\calC$. We show that $\calW_{\calC}(\{X\})$ is a wide subcategory of $\mods(\Lambda)$, i.e. that it is closed under extensions, kernels, and cokernels in $\mods(\Lambda)$. In order to establish closure under extensions, we fix a short-exact sequence in $\mods(\Lambda)$
    \begin{equation}\label{eq:Monica_W_exactseq}
        \begin{tikzcd}
            L \arrow[r,tail,"i"] & E \arrow[r,two heads,"p"] & M.
        \end{tikzcd}
    \end{equation}
    Fix also {an $\E$-triangle} for $X$ as in \eqref{eq:confX}. Since $\overline{T_0}$ and $\overline{T_1}$ are projective $\Lambda$-modules, we may induce a morphism of short-exact sequences as follows:
    \begin{equation}\label{eq:Monica_W_morofexactseq}
        \begin{tikzcd}
0 \arrow[r] \arrow[d, equal] & {\Hom_{\Lambda}\big(\overline{T_0},L\big)} \arrow[r, "i\circ -", tail] \arrow[d] & {\Hom_{\Lambda}\big(\overline{T_0},E\big)} \arrow[r, "p\circ -", two heads] \arrow[d] & {\Hom_{\Lambda}\big(\overline{T_0},M\big)} \arrow[d] \arrow[r] & 0 \arrow[d, equal] \\
0 \arrow[r]                    & {\Hom_{\Lambda}\big(\overline{T_1},L\big)} \arrow[r, "i\circ -", tail]           & {\Hom_{\Lambda}\big(\overline{T_1},E\big)} \arrow[r, "p\circ -", two heads]           & {\Hom_{\Lambda}\big(\overline{T_1},M\big)} \arrow[r]           & 0                   
\end{tikzcd}
    \end{equation}
    We may now use the Four Lemma \cite[{Lem.} I.3.2]{Mac67} to conclude that if two of the vertical morphisms in \eqref{eq:Monica_W_morofexactseq} are isomorphisms, so is the third. In particular, $\calW_{\calC}(\{X\})$ is closed under extensions.

    Take a morphism $f: L \to M$ in $\calW_{\calC}(\{X\})$, then the left-exact sequence 
    \[ 0 \to \ker(f) \to L \to M \]
    gives rise to a commutative diagram:
    \begin{equation*}
        \begin{tikzcd}
0 \arrow[r] \arrow[d, equal] & {\Hom_{\Lambda}\big(\overline{T_0},\ker(f) \big)} \arrow[r, "i\circ -", tail] \arrow[d] & {\Hom_{\Lambda}\big(\overline{T_0},L\big)} \arrow[r, "p\circ -"] \arrow[d] & {\Hom_{\Lambda}\big(\overline{T_0},M\big)} \arrow[d] \\
0 \arrow[r] & {\Hom_{\Lambda}\big(\overline{T_1},\ker(f)  \big)} \arrow[r, "i\circ -", tail]           & {\Hom_{\Lambda}\big(\overline{T_1},L\big)} \arrow[r, "p\circ -"]           & {\Hom_{\Lambda}\big(\overline{T_1},M\big)}               
\end{tikzcd}
    \end{equation*}
    We apply the Four Lemma to conclude that if the right-most two vertical arrows are isomorphisms, then so is the other nonzero vertical arrow. This establishes closure under kernels. Closure under cokernels is a consequence of the fact that we have two short exact sequences
    \[ 0 \to \ker(f) \to M \to \image(f) \to 0, \quad\text{and} \quad 0 \to \image(f) \to M \to \coker(f). \]
    Indeed, we have argued that if $M$ and $N$ are in $\calW_{\calC}(\{X\})$, then so is $\ker(f)$ and that if two terms in a short exact sequence are in $\calW_{\calC}(\{X\})$, then so is the third. The result follows.

    Similarly, in order to prove \eqref{prop:Monica_T}, it suffices to consider the case where $\mathcal{Y}$ contains a single object $Y$. Indeed, in the general case we have
    \begin{equation*}
        \calT_{\calC}(\mathcal{Y}) = \bigcap_{Y\in\mathcal{Y}} \calW_{\calC}(\{Y\}),     
    \end{equation*}
    so we may use the fact that the set of thick subcategories of $\calC$ is closed under arbitrary intersection. Let 
    \begin{equation*}
        \begin{tikzcd}
            L \arrow[r,tail,"i"] & E \arrow[r,two heads,"p"] & M \arrow[r, dashed] & \phantom{x}
        \end{tikzcd}
    \end{equation*}
    be an $\E$-triangle. Using the Horseshoe Lemma (\cref{lem:horseshoe}), we may induce a commutative diagram
    \begin{equation}\label{eq:thickconfla}
        \begin{tikzcd}
    0 \arrow[r] \arrow[d, equal] & {\Hom_{\Lambda}\big(\overline{T_0^M}, Y \big)} \arrow[r, "i\circ -", tail] \arrow[d] & {\Hom_{\Lambda}\big(\overline{T_0^E}, Y\big)} \arrow[r, "p\circ -", two heads] \arrow[d] & {\Hom_{\Lambda}\big(\overline{T_0^L},Y\big)} \arrow[d] \arrow[r] & 0 \arrow[d, equal]  \\
    0 \arrow[r] & {\Hom_{\Lambda}\big(\overline{T_1^M},Y  \big)} \arrow[r, "i\circ -", tail]           & {\Hom_{\Lambda}\big(\overline{T_1^E},Y\big)} \arrow[r, "p\circ -", two heads]           & {\Hom_{\Lambda}\big(\overline{T_1^L},Y \big)}    \arrow[r] & 0           
\end{tikzcd}
    \end{equation}
    where $T_i^L, T_i^E$ and $T_i^M$ are projective objects in $\calC$ which fit into a conflation with $L, E$ and $M$ as in \eqref{eq:confX} respectively. As in \eqref{prop:Monica_W}, the Four Lemma states that if two of the nonzero vertical arrows are isomorphisms, then so is the third. This means that $\calW_{\calC}(\{Y\})$ is closed under cones, cocones and extensions.
    
    We are left with showing that $\calW_{\calC}(\{Y\})$ is closed under direct summands. For this, let $M = M' \oplus M'' \in \calW_{\calC}(\{Y\})$ and notice that $\E(M,Y) = 0$ implies $\E(M',Y) = 0=\E(M'', Y)$. Consider a conflation
    \begin{equation*}
        \begin{tikzcd}
            M' \arrow[r,tail,"i"] & M \arrow[r,two heads,"p"] & M'' \arrow[r, dashed] & \phantom{x}
        \end{tikzcd}
    \end{equation*}
    and consider the induced diagram as in \eqref{eq:thickconfla}. If the vertical map in the middle is an isomorphism, its cokernel is zero, hence $\E(M,Y) = 0$. It follows from the above discusison that the other two vertical maps are epimorphisms. It follows from the Four Lemma that if the middle map is an isomorphism (in particular injective), then the right map is additionally injective, hence an isomorphism. Since this argument is symmetric, we have shown that $\calW_{\calC}(\{Y\})$ is a thick subcategory.
    
    We move on to prove \eqref{prop:Monica_Galois}. Since it is obvious that $\calW_{\calC}$ and $\calT_{\calC}$ reverse inclusions, only \eqref{eq:Monica_Galois} is left to prove. 
    {The two statements may be formalised as 
    \[ \forall L \in \mathcal{V}, \text{ $X$ is $L$-semistable } \forall X \in \mathcal{D} \quad \text{ and } \quad  \forall X \in \mathcal{D}, \text{ $X$ is $L$-semistable } \forall L \in \mathcal{V},\]}
    respectively. These assertions are logically equivalent, whence \eqref{eq:Monica_Galois} holds.

    Next, we prove \eqref{prop:Monica_Wbis}, following Garcia \cite[{Prop.} 3.14(4)]{Gar23}. Since $\calW_{\calC}$ reverses inclusions, the inclusion $\calW_{\calC}(\mathcal{X})\supseteq \calW_{\calC}(\thick(\mathcal{X}))$ is immediate. One has  $\mathcal{X}\subseteq\calT_{\calC}{(}\calW_{\calC}(\mathcal{X}) {)}$, as a result of \eqref{eq:Monica_Galois} (setting {$\calD = \mathcal{X}$ and} $\mathcal{V}=\calW_{\calC}(\mathcal{X})$). Further, since $\calT_{\calC}{(}\calW_{\calC}(\mathcal{X}){)}$ is a thick subcategory of $\calC$ {by \eqref{prop:Monica_T}}, it contains $\thick(\mathcal{X})$. This gives the second inclusion of
    \begin{equation*}
        \calW_{\calC}(\mathcal{X}) \subseteq \calW_{\calC}{(}\calT_{\calC}{(}\calW_{\calC}(\mathcal{X}){))} \subseteq \calW_{\calC}(\thick(\mathcal{X})),
    \end{equation*}
    where the first inclusion arises from \eqref{eq:Monica_Galois} {by setting $\mathcal{V} = \calW_{\calC}(\mathcal{X})$ and $\calD = \calT_{\calC} \calW_{\calC}(\mathcal{X})$}.
    We have proved that \eqref{prop:Monica_Wbis} holds.

   We prove \eqref{prop:Monica_W=J}. In this part of the proof we will fix an $\E$-triangle
   \begin{equation}\label{eq:confU}
        \begin{tikzcd}
            T_1 \arrow[r, tail, "{u}"] & T_0 \arrow[r,two heads] & U \arrow[r, dashed] & {},
        \end{tikzcd}
\end{equation}
in $\calC$ with the property that the right-exact sequence induced by $\overline{(-)}$ in $\mods (\Lambda)$ is of the form
\[ (\overline{T}_1)' \oplus \overline{\Omega I_U} \xrightarrow{(p \phantom{x} 0)} \overline{T}_0 \to \overline{U} \to 0\]
    where $p$ determines a minimal projective presentation of $\overline{U}$ in $\mods(\Lambda)$ and $I_U$ is the maximal injective non-projective direct summand of $U$. This exists for the following reason: Let $U \simeq U' \oplus I_U'$, where $I_U'$ is the maximal injective direct summand of $U$. Then we may take any $\E$-triangle 
   \begin{equation*}
        \begin{tikzcd}
            T_1'\arrow[r, tail] & T_0' \arrow[r,two heads] & U' \arrow[r, dashed] & {}
        \end{tikzcd}
\end{equation*}
in $\calC$. This $\E$-triangle is sent to a minimal projective presentation of $\overline{U}$ by $\overline{(-)}$ (see the proof of \cite[Prop. 4.5]{PZ24}). Moreover, we write $I_U' = N_U  \oplus I_U$ with $N_U$ projective-injective and take an $\E$-triangle 
\begin{equation*}
        \begin{tikzcd}
            \Omega I_U \arrow[r, tail] & N' \oplus N_U \arrow[r,two heads] & I_U' \arrow[r, dashed] & {},
        \end{tikzcd}
\end{equation*}
for $I_U$ as in \cref{lem:GNP23.3.12}, where $N'$ is projective-injective. The direct sum of these $\E$-triangles is the desired one. 

Let $\widetilde{(-)}$ denote the functor $\Hom_{\Ktwo(\proj(\Lambda))}(\Lambda,-) \colon \Ktwo(\proj (\Lambda)) \to \mods(\Lambda)$. By \cite[Thm. 3.2]{AIR2014} (see also \Cref{prop:F}\eqref{prop:PZ24.4.5}), it induces a bijection $\widetilde{G}: \presilt (\Lambda) \to \taurigidpair(\Lambda)$. 
Recall the map $F: \presilt(\calC) \to \presilt (\Lambda)$ from \Cref{prop:F}\eqref{thm:modulecatreduction_presilt}. As a consequence of \cite[Thm. 3.2]{AIR2014}, we obtain $\widetilde{G}(F(U)) = (\overline{U}, \overline{\Omega I_U})$.

   In the special case where $\calC=\Ktwo(\proj\Lambda)$ and $\Lambda=\End_{\calC}(T)$, we denote the maps in the Galois connection {of} \eqref{prop:Monica_Galois} by $\calW_{\Lambda}$ and $\calT_{\Lambda}$. For a full category $\mathcal{X}$ of the bounded derived category $\Db(\mods(\Lambda))$, we will make use of the \textit{perpendicular subcategories} of $\mathcal{X}$, defined by
\begin{align*}
    \mathcal{X}^{\perp_{\Z}} &\coloneqq \{Y\in \Db(\mods(\Lambda)) \sth \Hom_{\Db(\mods(\Lambda))}(X,\Sigma^i Y)=0 \quad \forall X\in\mathcal{X}\quad \forall i\in\Z \}.
\end{align*}

We first prove that $\calW_{\calC}(\{ U\}) = \calW_{\Lambda}(\{F(U)\})$. Let
\begin{equation}\label{eq:confQ}
        \begin{tikzcd}
            Q_1 \arrow[r, tail, "v"] & Q_0 \arrow[r,two heads] & F(U) \arrow[r, dashed] & {},
        \end{tikzcd}
\end{equation}
be a conflation in $\Ktwo(\proj \Lambda)$ with $Q_0, Q_1 \in \add(\Lambda)$. Because $\widetilde{G}(F(U))=(\overline{U}, \overline{\Omega I_U})$, we have $\widetilde{F(U)} = \overline{U}$. As above, it follows from \cite[Prop. 4.5]{PZ24} that applying $\widetilde{(-)}$ to \cref{eq:confQ} gives a projective presentation
\[ (\widetilde{Q_1})' \oplus \overline{\Omega I_U} \xrightarrow{(q \phantom{x} 0)} \widetilde{Q_0} \to \overline{U} \to 0\]
where $q$ determines a minimal projective presentation of $\overline{U}$ in $\mods (\Lambda)$. That means, we may identify $p$ and $q$. Consequently, ${\Hom_{\Lambda}(p,L)}$ is an isomorphism if and only if ${\Hom_{\Lambda}(q,L)}$ is an isomorphism, whence $\Hom_{\Lambda}(\overline{u},L)$ is an isomorphism if and only if $\Hom_{\Lambda}(\widetilde{v}, L)$ is an isomorphism. So, $\calW_{\calC}(\{U\}) = \calW_{\Lambda}(\{F(U)\})$.
We conclude that
\begin{equation}\label{eq:Gar24Bor21w}
\begin{tikzcd}[column sep=2em]
\calW_{\calC}(\{U\}) \arrow[r, equal] & {\calW_{{\Lambda}}(\{F(U)\})} \arrow[r, "{\text{\cite{Gar24}}}", "\text{4.18}"', equal]& (F(U))^{\perp_{\Z}}\cap \mods(\Lambda) \arrow[r, "{\text{\cite{Bor21}}}", "\text{2.9(1)}"', equal] & J(\widetilde{G}(F(U))) \arrow[r, equal]  & J(\overline{U},\overline{\Omega I_U}),
\end{tikzcd}
\end{equation}
showing that \eqref{prop:Monica_W=J} holds.
   
   In order to prove \eqref{prop:Monica_Z}, it will be useful to know that the bottom half of the following diagram is commutative:
    \begin{equation}\label{eq:GaloisTower}
    	\begin{tikzcd}[row sep=4em]
	\fullsubc \calC \arrow[rr, bend right=15,"\calW_{\calC}"'] &  & {\fullsubc \mods(\Lambda)} \arrow[ll, bend right=15,"\calT_{\calC}"'] \\
\thicksubc_U \calC \arrow[d,"\varphi"] \arrow[rr, bend right=15,"\calW_{\calC}"'] \arrow[u,hook] &  & {\fullsubc J(\overline{U},\overline{\Omega I_{U}})} \arrow[d, "H"] \arrow[ll, bend right=15,"\calT_{\calC}"']  \arrow[u,hook]  \\
\thicksubc (\calC/\thick(U)) \arrow[rr, bend right=15,"\calW_{\calC/\thick(U)}"']   &  & \fullsubc \mods(C_U) \arrow[ll, bend right=15,"\calT_{\calC/\thick(U)}"']                                           
\end{tikzcd}
    \end{equation}
    Note that the upper half commutes as a result of \eqref{prop:Monica_T}, \eqref{prop:Monica_Galois}, and \eqref{prop:Monica_W=J} above.
      Here, we denote the set of thick subcategories of $\calC$ containing $U$ by $\thicksubc_U \calC$. As $\calC/\thick(U)$ and $U^+$ satisfy \Cref{setting} (see \Cref{prop:siltred}), we can use \eqref{prop:Monica_Galois} above to form the Galois connection in the bottom row. The poset isomorphism $\varphi$ is induced by the localisation $\calC \to \calC/\thick(U)$, see \cref{lem:Ver2.3.1}, and $H$ is induced by the functor
      \begin{equation}\label{eq:functorF}
       H = \Hom_{\Lambda}(\overline{U^+},-)\colon \mods(\Lambda) \to \mods(C_U),
      \end{equation}
      where $C_U$ is the $k$-algebra $\End_{\Lambda}(\overline{U^+})/[\overline{U}]$. The functor $H$ restricts to an exact equivalence 
      \begin{equation}\label{eq:equivF}
       H = \Hom_{\Lambda}(\overline{U^+},-)\colon J(\overline{U},\overline{\Omega I_{U}}) \to \mods(C_U)
      \end{equation}
      by \cite[Thm. 3.8]{Jas15} \cite[Thm. 4.12(b)]{DIRRT17}.
      Hence, the map $H$ in \eqref{eq:GaloisTower} is also a poset isomorphism. {More precisely, the codomain of the map $\calW_{\calC/\thick(U)}$ is the poset of wide subcategories of $\mods(\End_{\calC/\thick(U)}(U^+)/[V])$, where $V$ is the maximal basic injective object in $\calC/\thick(U)$, but by \Cref{lem:endoinred} we find this poset to be isomorphic to that of wide subcategories in $\mods{(}C_U{)}$. {Similarly, we argue that the domain of $\calT_{\calC/\thick(U)}$ is equivalent to $\fullsubc(\mods C_U)$.}}
      
      As soon as we know that the diagram in \eqref{eq:GaloisTower} commutes, we will deduce the claim in \eqref{prop:Monica_Z} as follows:{      \begin{align*}
     	 \calT_{\calC}\calW_{\calC}(\{U\}) &=  \varphi^{-1}\calT_{\calC/\thick(U)}H \big(\calW_{\calC}(\{U\})\big) \\
	  &=  \varphi^{-1}\calT_{\calC/\thick(U)}H \big(J(\overline{U},\overline{\Omega I_U}) \big) & \text{(\eqref{prop:Monica_W=J} above)} \\
	  &= \varphi^{-1}\calT_{\calC/\thick(U)}\big( \mods(C_U) \big) \\
	  &=  \varphi^{-1}(\thick_{\calC/\thick(U)}(N)) & \text{(\Cref{lem:Ver2.3.1})} \\
	  &= \thick(U\oplus N).
      \end{align*}
      }
      
      To show that the bottom half of \eqref{eq:GaloisTower} commutes, there are two squares to consider. One involves $\calT$s and the other involves $\calW$s. We address the two cases in that order. For the first, it suffices to fix an object $M\in J(\overline{U},\overline{\Omega I_U})$ and establish commutativity for the full subcategory {$\{ M\}$} of $J(\overline{U},\overline{\Omega I_U})$ (see \eqref{eq:unionObj} and the surrounding discussion). Given an object $X$ in $\calC$, we fix {an $\E$-triangle} as in \cref{eq:confX}, and we let $L(x)$ denote the image of $x$ in $\calC/\thick(U)$ under the localisation functor $L\colon \calC \xrightarrow{} \calC/\thick(U)$. We {thus need to compare}
      \begin{align*}
      \calT_{\calC}(\{M\}) &= \{X \in \calC \sth \Hom_{\Lambda}(\overline{x},M) \text{ is an isomorphism} \}, \\
      \calT_{\calC/\thick(U)}(\{HM\}) &=  \{X \in \calC/\thick(U) \sth \Hom_{C_U}\left({\calC \over \thick(U)}(U^+, L(x)),HM\right) \text{ is an isomorphism} \}. 
      \end{align*}
       \Cref{prop:Bongartz} gives {an $\E$-triangle}        
        \begin{equation}\label{eq:3.9exchange'}
      \begin{tikzcd}
      	T_i  \arrow[r,tail,"\alpha_i"] & Y_i  \arrow[r,two heads] & U'_i \arrow[r,dashed] & {}
	\end{tikzcd}
      \end{equation}
         for $i=0,1$, where the objects $T_i$ are as in \eqref{eq:confX}, $\alpha_i$ is a left $\add(U^+)$-approximation {of $T_i$} and $U'_i$ is in $\add(U)$. Let $Y_M \in \calC$ be an object such that $Y_M$ is sent to $M$ under the equivalence $\calC/[\calI] \to \mods (\Lambda)$ of \cref{lem:GNP23.3.11}. By \cref{lem:LESmodinj}, there are long exact sequences
       \[ \frac{\calC}{[\calI]}(U_i', Y_M) \to \frac{\calC}{[\calI]}(Y_i, Y_M) \to \frac{\calC}{[\calI]}(T_i, Y_M) \to \E(U_i', Y_M) \]
       for $i=0,1$. 
       After applying the equivalence of categories in \cref{lem:GNP23.3.11} we obtain long exact sequences
       \begin{equation}\label{eq:3.9longexact} \Hom_{\Lambda}(\overline{U_i'}, M) \to \Hom_{\Lambda}(\overline{Y_i}, M) \to \Hom_{\Lambda}(\overline{T_i}, M) \to \E(U_i', Y_M) \end{equation}
       for $i=0,1$. We write $U_i' = N_{U_i'} \oplus I_{U_i'} \oplus U_i''$ with $N_{U_i'}$ the maximal projective direct summand and $I_{U_i'}$ the maximal injective non-projective direct summand. Recall that $M \in J(\overline{U},\overline{\Omega I_U}) \subseteq \mods(\Lambda)$ and ${\overline{U_i'}} \in \add(\overline{U})$. We get that the two outer terms in \eqref{eq:3.9longexact} vanish for the following reasons:
       \begin{itemize}
       		\item $ \Hom_{\Lambda}(\overline{U_i'}, M) = 0$ because $M \in \overline{U}^\perp$;
		\item $\E(N_{U_i'}, Y_M) = 0$ because $N_{U_i'}$ is projective;
		\item We have $\E(U_i'', Y_M) = 0 $ as a consequence of $M \in {}^\perp \tau \overline{U}$. Indeed, fixing an $\E$-triangle
        \begin{equation*}
            \begin{tikzcd}
      	     {T_1''}  \arrow[r,tail,"{u''}"] & {T_0''}  \arrow[r,two heads] & {U_i''} \arrow[r,dashed] & {}
	       \end{tikzcd}
        \end{equation*}
        with $T_0''$ and $T_1''$ projective, it follows from \cref{lem:LESmodinj} that $\E(U_i'', Y_M) = 0$ if and only if $-\circ [u'']$ is surjective (because $T''_0$ is projective). As is in the proof of \cite[Prop. 4.5]{PZ24}, the $\E$-triangle above is sent to a minimal projective presentation of $\overline{U_i''}$ because $U_i''$ has no injective direct summands. Then, \cite[Prop. 2.4]{AIR2014} states that $\Hom_{\Lambda}(M,\tau \overline{U_i''}) = 0$ implies the surjectivity of $-\circ [u'']$.
		\item Because $M \in \overline{\Omega I_U}^\perp$ we have $\frac{\calC}{[\calI]}(\Omega I_{U_i'}, Y_M) = \Hom(\overline{\Omega I_{U_i'}}, M) = 0$. From an $\E$-triangle involving $I_{U_i'}$ and $\Omega I_{U_i'}$ as in \cref{lem:GNP23.3.12} we obtain a surjection $\calC( \Omega I_{U_i'}, Y_M) \to \E(I_{U_i'}, Y_M) \to 0$. By the proof of the exactness at $\E(Z,U)$ in \cref{lem:LESmodinj}, we have that the surjection factors through $\frac{\calC}{[\calI]}(\Omega I_{U_i'}, Y_M) =0$ from which $\E(I_{U_i'}, Y_M)=0$ follows.
       \end{itemize}
       Putting together these observations yields $\E(U_i', Y_M) =0$. Consequently, the morphism $\alpha_i$ induces an isomorphism
      \begin{equation}\label{eq:alphainducediso}
      \begin{tikzcd}[column sep=3em]
            \Hom_{\Lambda}(\overline{Y_i},M) \arrow[r,"(\alpha_i\circ -)^\ast"] &  \Hom(\overline{T_i},M).
            \end{tikzcd}
      \end{equation}
    Next, using that fact that $\alpha_1$ {is a} left $\add(U^+)$-approximation, there exists a morphism $x'$ such that the following diagram commutes
	\begin{equation}\label{eq:xtox'}
	\begin{tikzcd} 
	T_1 \arrow[d,tail, "\alpha_1", swap] \arrow[r, tail, "x"]& T_0\arrow[d,tail, "\alpha_0"] \\
	Y_1 \arrow[r, "x'", dashed] & Y_0 
	\end{tikzcd} {.}
	\end{equation}
	By applying $\Hom_{\Lambda}({\overline{(-)}},M)$ to the commutative diagram in \eqref{eq:xtox'}, we obtain a commutative diagram of $k$-vector spaces
	\[
	\begin{tikzcd} [column sep=8em]
	\Hom_{\Lambda}(\overline{T_0},M) \arrow[r, "{\Hom_{\Lambda}(\overline{x},M)}"]& \Hom_{\Lambda}(\overline{T_1},M)  \\
	\Hom_{\Lambda}(\overline{Y_0},M) \arrow[r, "{\Hom_{\Lambda}(\overline{x'},M)}"] \arrow[u, "{(\alpha_0\circ -)^\ast}"] & \Hom_{\Lambda}(\overline{Y_1},M)  \arrow[u, "{(\alpha_1\circ -)^\ast}"]  	\end{tikzcd}
	\]
	where the vertical homomorphisms are isomorphisms {as established in \eqref{eq:alphainducediso}}. It follows that $\Hom_\Lambda(\overline{x}, M)$ is an isomorphism if and only if $\Hom_\Lambda(\overline{x'}, M)$ is an isomorphism.
	
Let $f_U\colon \mods{(}\Lambda{)} \rightarrow \overline{U}^{\perp}$ denote the natural functor into the torsion-free class $\overline{U}^{\perp}$ of $\mods{(}\Lambda{)}$ (see \cite{Dickson66}). Since $M\in {\overline{U}^\perp}$, we have {that} {$\Hom_{\Lambda}(\overline{Y_i}, M) \simeq \Hom_{\Lambda}(f_U(\overline{Y_i}), M)$ from which it follows that} $\Hom_{\Lambda}(\overline{x'}, M)$ is an isomorphism if and only if $ \Hom_{\Lambda}(f_U\overline{x'}, M)$ is.
Since the bijection between silting objects and support $\tau$-tilting pairs in \Cref{prop:F}\eqref{prop:PZ24.4.5} sends the Bongartz completion $U^+$ to the Bongartz completion of $(\overline{U},\overline{\Omega I_U})$ in a $\tau$-tilting theoretic sense \cite[Cor. 4.14]{PZ24}, we have that ${f}_U(\overline{U^+})$ is a projective generator of the abelian category $J(\overline{U}, \overline{\Omega I_U})$ \cite[Lem. 4.9]{BM18t}. 
Since $J(\overline{U}, \overline{\Omega I_U})$ is a full subcategory of $\mods(\Lambda)$ and $f_U(\overline{Y_0}), f_U(\overline{Y_1}), M \in J(\overline{U}, \overline{\Omega I_U})$ it contains the morphism $f_U\overline{x'}$. Therefore, the map $\Hom_{\Lambda}({{f}_U}\overline{x'}, M)$ is an isomorphism if and only if $\Hom_{J(\overline{U}, \overline{\Omega I_U})}(f_U\overline{x'}, M)$ is an isomorphism.
It follows from this discussion that
	\[  \calT_{\calC}(\{M\}) = \{ X \in \calC \sth \Hom_{J(\overline{U},\overline{\Omega I_U})}(f_U\overline{x'},M) \text{ is an isomorphism}\}. \]
By an analogous result of \cite[Prop. 4.15]{Jas15}, the equivalence $H$ in \eqref{eq:equivF}, namely
\begin{equation*}
	H\colon J(\overline{U},\overline{\Omega I_U}) \to \mods(C_U),
\end{equation*}
sends the morphism $\Hom_{J(\overline{U},\overline{\Omega I_U})}(f_U \overline{x'},M)$ to $\Hom_{C_U}\left({\calC \over \thick(U)}(U^+,L(x')),HM\right)$.
Consequently, one is an isomorphism if and only if the other one is.
Since the cones of the vertical morphisms in \eqref{eq:xtox'} are in $\thick(U)$, see {\cref{eq:3.9exchange'}}, we have that $L(x')$ and $L(x)$ are isomorphic as morphisms in ${\calC \over \thick(U)}$, so {$\Hom_{C_U}\left({\calC \over \thick(U)}(U^+,L(x')),HM\right)$} is an isomorphism if and only if $\Hom_{C_U}\left({\calC \over \thick(U)}(U^+,L(x)),HM\right)$ is an isomorphism. By stringing together the arguments about $\Hom$-spaces, we conclude that the square involving the $\calT$s commutes.

To establish the commutativity of the square involving the $\calW$s, it suffices to fix an object $X\in \calC$ and establish commutativity for the full subcategory of $\calC$ spanned by $U$ and $X$ (again, see \eqref{eq:unionObj} and the surrounding discussion). We {thus need to compare}
      \begin{align*}
      \calW_{\calC}(\{U,X\}) &= \{M\in J(\overline{U},\overline{\Omega I_U}) \sth \Hom_{\Lambda}(\overline{x},M) \text{ is an isomorphism } \}, \\
      \calW_{\calC/\thick(U)}(\{X\}) &= \{M \in \mods(C_U) \sth \Hom_{C_U}\left({\calC\over \thick(U)}(U^+,{L(x)}), M \right) \text{ is an isomorphism } \}. 
      \end{align*}
       {In this way, t}he commutativity of the square follows from the same argument as above. We hereby conclude the proof of \eqref{prop:Monica_Z}.
       
        Finally, we prove \eqref{prop:Monica_maps}.
    Let $\mathrm{Im}(\calW_{\calC})$ and $\mathrm{Im}(\calT_{\calC})$ denote the images of the inclusion-reversing maps defined in \eqref{prop:Monica_Galois} above. We then have a {Galois correspondence} \cite[{Thm.} 2]{Ore44}
    \begin{align*}
        \calW_{\calC}: \mathrm{Im}(\calT_{\calC}) &\to \mathrm{Im}(\calW_{\calC}) \\
        \calT_{\calC}: \mathrm{Im}(\calW_{\calC}) &\to \mathrm{Im}(\calT_{\calC}),
    \end{align*}
    {that is,} a pair of mutually inverse inclusion-reversing maps.
     By \eqref{prop:Monica_W}, the poset $\mathrm{Im}(\calW_{\calC})$ is a subposet of that of wide subcategories of $\mods(\Lambda)$, and by \eqref{prop:Monica_T}, the poset $\mathrm{Im}(\calT_{\calC})$ is a subposet of that of thick subcategories of $\calC$.  
     It now suffices to show the { following}
     \begin{itemize}
     	\item $\calW_{\calC}$ sends a thick subcategory $\thick(U)$, where $U$ is {presilting and} {such that $\add(U)$ contains all projective-injective objects in $\calC$,} to $\tau$-perpendicular wide subcategories,
	\item $\calT_{\calC}$ does the opposite.
     \end{itemize}
     {Note that the functor $\overline{(-)}\colon \calC \to \mods(\Lambda)$ sends all projective-injective objects to zero, whence the projective-injective objects in $\calC$ are $L$-semistable for any $L$ in $\mods(\Lambda)$.}
     By the bijectivity of the maps {between the images} in the Galois correspondence, {it suffices to show that thick subcategories generated by presilting objects are sent to $\tau$-perpendicular subcategories, and vice versa}. Specifically, given  $U\in\presilt(\calC)$, we will show that
    \begin{align}
        \calW_{\calC}(\thick(U))=J(\overline{U},\overline{\Omega I_U}), \label{eq:Monica_WJ} 
        \\
         {\calT_{\calC}(J(\overline{U},\overline{\Omega I_U}))=\thick_{}(U {\oplus N})} \label{eq:Monica_Tt}
    \end{align}
    
    {Using our results above,}
\begin{equation*}\label{eq:Gar24Bor21w'}
\begin{tikzcd}[column sep=2em]
\calW_{{\calC}}(\thick(U)) \arrow[r, "{\text{\eqref{prop:Monica_Wbis}}}", "\text{above}"', equal] & \calW_{{{\calC}}}(\{U\}) \arrow[r, equal] \arrow[r, equal,"{\eqref{prop:Monica_W=J}}", "\text{{above}}"']  & J(\overline{U},\overline{\Omega I_U}).
\end{tikzcd}
\end{equation*}
We have shown that \eqref{eq:Monica_WJ} holds.
We conclude the proof by deducing \eqref{eq:Monica_Tt}:
\begin{equation*}
\begin{tikzcd}
\calT_{\calC}\big(J(\overline{U},\overline{\Omega I_U})\big) \arrow[r, "\eqref{eq:Monica_WJ}", equal]  &  {\calT_{\calC}(\calW_{\calC}(\thick(U))) }  \arrow[r, "{{\text{\eqref{prop:Monica_Wbis}}}}", "\text{{above}}"', equal]  & \calT_{\calC}{(} \calW_{\calC}(\{U\}){)}\arrow[r, "{{\text{\eqref{prop:Monica_Z}}}}", "\text{{above}}"', equal]  & \thick(U {\oplus N}).
\end{tikzcd}
\end{equation*}
The proof is complete.
\end{proof}

Having proved \Cref{prop:Monica}, we are ready to show that we have generalised the notion of $\tau$-cluster morphism category, in the sense of \cite{IT17,BH21,BM18w}. 
Schroll--Tattar--Treffinger--Williams define an equivalent category $\mathfrak{T}(\Lambda)$ \cite[{Thm. 1.1}]{STTW23}, which the second named-author recovers in terms of the category of a partitioned fan {of} $(\partfan{\Lambda},\mathfrak{P}_{\mathrm{WAC}})$ \cite[§6.3]{Kai23}. The definition of $\mathfrak{P}_{\mathrm{WAC}}$ identifies two cones $\boldC(U), \boldC(U') \in \partfan{\Lambda}$ whenever $J(\overline{U}, \overline{\Omega I_U})= J(\overline{U'}, \overline{\Omega I_{U'}})$. Thus we recover the $\tau$-cluster morphism category as follows.

\begin{proposition}\label{prop:samePartition}
    \label{thm:modulecatreduction_piccat}
    \, 
    
    \begin{enumerate}
    	\item \label{thm:modulecatreduction_piccat1}  Let $\Lambda$ be a finite-dimensional $k$-algebra. Then the picture category $\Wfrak(\Lambda)$ is equivalent to the $\tau$-cluster morphism category $\mathfrak{T}(\Lambda)$ of $\Lambda$.
	\item \label{thm:modulecatreduction_piccat2} Let $\calC$ and $\Lambda$ be as in \Cref{setting}, and suppose furthermore that $\calC$ is reduced. 
    Consider the isomorphism of fans
    \begin{equation*}
        \partfan{\calC} \to \partfan{\Lambda},
    \end{equation*}
    put forward in \Cref{lem:FanIso}\eqref{thm:modulecatreduction_fans}. This morphism preserves and reflects the thick partitions of the relevant fans and thus gives rise to an equivalence of picture categories
     \[\Wfrak(\calC) \to \Wfrak(\Lambda).\]
    \end{enumerate}
\end{proposition}
\begin{proof}
	To prove \eqref{thm:modulecatreduction_piccat1} {l}et $\sigma, \sigma' \in \partfan{\Lambda}$. Then we may assume that $\sigma = \boldC(U)$ and $\sigma' = \boldC(U')$ for some $U,U' \in \presilt(\Lambda)$. By \Cref{prop:F}\eqref{prop:PZ24.4.5}, both $U$ and $U'$ correspond to $\tau$-rigid pairs $(\overline{U}, \overline{\Omega I_U})$ and $(\overline{U'}, \overline{\Omega I_{U'}})$ respectively. Then we obatin
    \begin{align*}
        \boldC(U) \sim_{\Pfrak_{\thick}} \boldC(U') & \Longleftrightarrow \thick(U) = \thick(U') \\
        & \Longleftrightarrow \calT_{\calC} (J(\overline{U}, \overline{\Omega I_U})) = \calT_{\calC}( J(\overline{U'}, \overline{\Omega I_{U'}})) &\text{(by \cref{prop:Monica}\eqref{prop:Monica_W=J}\eqref{prop:Monica_Z})} \\
        & \Longleftrightarrow \calW_{\calC}\calT_{\calC} (J(\overline{U}, \overline{\Omega I_U})) = \calW_{\calC} \calT_{\calC}( J(\overline{U'}, \overline{\Omega I_{U'}})) \\
        & \Longleftrightarrow J(\overline{U}, \overline{\Omega I_U}) =  J(\overline{U'}, \overline{\Omega I_{U'}}) & \text{(by \cref{prop:Monica}\eqref{prop:Monica_maps})} \\
        & \Longleftrightarrow \boldC(U) \sim_{\Pfrak_{\mathrm{WAC}}} \boldC(U').
    \end{align*}
    This shows that $\Pfrak_{\thick} = \Pfrak_{\mathrm{WAC}}$ and hence 
    \[\Wfrak(\Lambda) = \catof(\partfan{\Lambda}, \Pfrak_{\thick}) = \catof(\partfan{\Lambda}, \Pfrak_{\mathrm{WAC}}) = \mathfrak{T}(\Lambda).\]
    
    In order to prove \eqref{thm:modulecatreduction_piccat2} {recall the bijection} 
    \[ F\colon \presilt(\calC) \to \presilt( \Lambda) \]
    defined in \Cref{prop:F}\eqref{thm:modulecatreduction_presilt}. Consider the composite map of posets
    \begin{equation*}
        \begin{tikzcd}
            \thicksubc{\calC} \arrow[r,"\calW_{\calC}"] & \wide \Lambda \arrow[r,"\calT_{\Lambda}"] & \thicksubc {\Ktwo(\proj{\Lambda})}
        \end{tikzcd}
    \end{equation*}
    from thick subcategories of $\calC$, through wide subcategories of $\mods(\Lambda)$, to thick subcategories of $\Ktwo(\proj{\Lambda})$.
    By \Cref{prop:Monica}\eqref{prop:Monica_maps},
    this composite map sends $\thick_{\calC}(U)$ to $\thick_{\Lambda}(FU)$ whenever $U$ is a presilting object in $\calC$. Since the maps $\calW_{\calC}$ and $\calT_{\Lambda}$ are bijections when restricted as imposed in \Cref{prop:Monica}\eqref{prop:Monica_maps}, the following occurs: if $U$ and $V$ are presilting objects in $\calC$, then we have $\thick_{\calC}(U)=\thick_{\calC}(V)$ if and only if $\thick_{\Lambda}(FU)=\thick_{\Lambda}(FV)$. This proves that the thick partitions are preserved and reflected by the fan isomorphism in \Cref{lem:FanIso}\eqref{thm:modulecatreduction_fans}, establishing that there is an equivalence as abstract partitioned fans $(\partfan{\calC}, \mathfrak{P}_{\thick})=(\partfan{\Lambda}, \mathfrak{P}_{\thick})$ and thus their categories of partitioned fans coincide, which completes the proof.
\end{proof}

We have now, implicitly, seen the first example of the notion of a morphism of partitioned fans, which we will formally introduce and study in \Cref{sec:morpartfan}. To conclude this section, we connect \Cref{defthm:piccat} with a previous construction of the first-named author.

In previous work of the first-named author \cite[§6]{Bor24}, picture categories were defined for 0-Auslander extriangulated $k$-categories admitting an exact dg enhancement. Our final result of this section, namely \Cref{prop:sameasBor}, shows that our construction above is compatible with the dg categorical construction.

X. Chen defines the notion of \textit{exact dg $k$-category} by formulating the axioms of exact  $k$-categories in a dg categorical context \cite{Che23thesis,Che26}. If a dg $k$-category $\A$ is equipped with an exact structure, then the homology category $H^0\A$ is naturally equipped with an extriangulation \cite[{Thm.} 4.26]{Che24}. We say that an exact dg $k$-category is \textit{0-Auslander} if the extriangulation on the homology category $H^0\A$ is 0-Auslander \cite{Che23}.

Let $\A$ be a 0-Auslander exact dg $k$-category with $k$-cofibrant mapping complexes. Suppose also that the 0-Auslander extriangulated $k$-category $H^0\A$ satisfies \Cref{setting}, in particular that it contains a silting object. This assumption on $\A$ is less general than what is needed to define the picture category $\tcmc{\A}$ of $\A$ {in \cite{Bor24}}. It is, however, necessary to consider this special case in order to prove \Cref{prop:sameasBor} below. Before proving \Cref{prop:sameasBor}, we briefly recall how $\tcmc{\A}$ is defined.

Let $\mathbf{Hqe}_{k}$ denote the localisation of the category of small dg $k$-categories at its quasi-equivalences. For a full dg subcategory $\mathscr{S}\subseteq \A$, {let $S$ denote the class of morphisms in $H^0\A$ that are obtained as composite morphisms of morphisms from the following two classes: inflations with cones in $\thick(H^0\mathscr{S})$ or deflations with co-cones in $\thick(H^0\mathscr{S})$.}
Let ${L_{S}}\A$ denote the dg {localisation} of $\A$ with respect to $S$. If $H^0\mathscr{S}$ is the additive hull of an object $X\in H^0\A$, we simplify the notation to ${L_X\A}$. Two dg {localizations} ${L_{S_1}\A}$ and ${L_{S_2}\A}$ are said to be \textit{equivalent} if the dg localisation morphisms $\A \xrightarrow{L_{S_1}} {L_{S_1}\A}$ and $\A \xrightarrow{L_{S_2}} {L_{S_2}\A}$ are isomorphic in the under-category of $\A$ in $\mathbf{Hqe}_{k}$. The equivalence class containing $\A \xrightarrow{L_{S}} {L_{S}\A}$ will be denoted $[\A \xrightarrow{L_{S}} {L_{S}\A}]$.

A presilting object in $\A$ is simply defined to be a presilting object in the extriangulated $k$-category $H^0\A$. 
If $V$ is a basic presilting object of $H^0\A$, the dg localisation morphism $\A \xrightarrow{\varpi_{V}} {L_V\A}$ induces a bijection
\begin{equation}\label{eq:dgsiltred}
    \presilt_V(\A) \xrightarrow{\varpi_{V}} \presilt({L_V\A}),
\end{equation}
where the domain is the set of basic presilting objects in $\A$ containing $V$ as a direct summand, and the codomain is the set of all basic presilting objects in ${L_V\A}$ \cite[Cor. 5.16(iv)]{Bor24}. 

The class of objects in the \textit{picture category} of $\A$, denoted $\tcmc{\A}$, is given by equivalence classes of dg quotients $[\A \xrightarrow{\varpi_{U}} {L_U\A}]$, where $U$ is a basic\footnote{It is not necessary to impose that $U$ be basic; any localisation with respect to a presilting object is clearly equivalent to the localisation with respect to a basic presilting object. We find it convenient to consistently choose basic representatives, so that the proof of \Cref{prop:sameasBor} becomes easier to conduct.} presilting object in $H^0\A$. A morphism in $\tcmc{\A}$ is given by an equivalence class of localisations $[{L_U\A} \xrightarrow{\varpi_{V}} {L_V(L_U\A)}]$, where $V$ is a basic presilting object in ${L_U\A}$. The composite of two consecutive morphism is defined by
\begin{equation}\label{eq:comp_tcmc}
    [{L_V(L_U\A)} \xrightarrow{\varpi_W} {L_W(L_V(L_U\A))} ]\circ [{L_U\A} \xrightarrow{\varpi_V} {L_V(L_U\A)}]=[{L_U\A} \xrightarrow{\varpi_{\varpi_V^{-1}W}} {L_{\varpi_V^{-1}W}(L_U\A)}]
\end{equation}
where $\varpi_V^{-1}W$ denotes the preimage of $W$ along the bijection in \eqref{eq:dgsiltred}. This composition rule can be shown to be well-defined \cite[§6.2]{Bor24} and associative \cite[{Prop.} 6.3]{Bor24}, whence the picture category $\tcmc{\A}$ is a well-defined category. 

\begin{proposition}\label{prop:sameasBor}
    Let $\A$ be a 0-Auslander exact dg $k$-category with $k$-cofibrant mapping complexes, such that the 0-Auslander extriangulated $k$-category $H^0\A$ satisfies \Cref{setting}, and let $\tcmc{\A}$ denote its picture category in the sense of the first-named author \cite[§6]{Bor24}. Let $\Wfrak(H^0\A)$ denote the picture category of $H^0\A$, as defined in \Cref{defthm:piccat}\eqref{defthm:piccat_0Aus}. There is a well-defined equivalence of categories
    \begin{equation}\label{eq:cat_fan=tcmc}
    \begin{tikzcd}
        \tcmc{H^0\A} \arrow[r] & \tcmc{\A}
    \end{tikzcd}
    \end{equation}
sending the object $\eqclass{{U}}$ to $[\A \xrightarrow{\varpi_U} {L_U\A}]$ and the morphism $\sqbinom{W}{V}$ to $[ {L_V\A} \xrightarrow{\varpi_W}  {L_W\A}]$, where $\varpi_V$ is as in \eqref{eq:dgsiltred}.
\end{proposition}
\begin{proof}
    We first show that our proposed functor encodes well-defined maps of objects and morphisms.
    In \Cref{def:partitionbythick}, we imposed that the thick partition of $\partfan{H^0\A}$ identify cones of the presilting object ${U_1}$ and ${U_2}$ when $\thick(U_1)=\thick(U_2)$. Thus, if the cones of ${U_1}$ and ${U_2}$ are identified, the dg localisation morphisms $\A\xrightarrow{\varpi_{U_1}}  {L_{U_1}\A}$ and $\A\xrightarrow{\varpi_{U_2}}  {L_{U_2}\A}$ are equivalent, which shows that the map of objects is well-defined. To show that the map of morphisms is well-defined, we give ourselves four basic presilting object $U_1$, $U_2$, $V_1$, and $V_2$ in $H^0\A$, such that $U_1$ is a direct summand of $V_1$ and $U_2$ is a direct summand of $V_2$. 
    By the definition of morphisms in $\Wfrak(H^0\A)$ (see \eqref{eq:piccat_mor}), we have that $\thick(U_1)=\thick(U_2)$ and that $\thick(V_1)=\thick(V_2)$.
    The morphisms $\sqbinom{V_1}{U_1}$ and $\sqbinom{V_2}{U_2}$ coincide precisely when $V_1$ and $V_2$ are isomorphic in $\calC/\thick(U_1)$ (or indeed $\calC/\thick(U_2)$. 
    If this is the case, the dg localisations $[ {L_{U_1}\A} \xrightarrow{\varpi_{V_1}} {L_{V_1}}({L_{U_1}\A} )]$ and $[  {L_{U_2}\A} \xrightarrow{\varpi_{V_2}} {L_{V_2}}({L_{U_2}\A})]$, must be equivalent, so the map of morphisms is well-defined.

    Next, we show that we have defined a functor in \eqref{eq:cat_fan=tcmc}. It is clear that identity morphisms are sent to identity morphisms. Given a 
    composite $\sqbinom{W}{V}\circ\sqbinom{V}{U} = \sqbinom{W}{U}$ in $\tcmc{H^0\A}$, it is to be shown that the following is met in $\tcmc{\A}$: $$[{L_V\A} \xrightarrow{\varpi_W} {L_W\A}]\circ [{L_U\A} \xrightarrow{\varpi_V} {L_V\A}]=[{L_U\A} \xrightarrow{\varpi_W} {L_W\A}],$$ but this is straightforward from the definition of the composition rule in $\tcmc{\A}$, as recalled above. 

    Having shown that we have a well-defined functor in \eqref{eq:cat_fan=tcmc}, we conclude the proof by showing that it is an equivalence of categories. We show that functor in \eqref{eq:cat_fan=tcmc} is fully faithful and dense.
    Density is clear from the definitions. 
    Let $\eqclass{U}$ and $\eqclass{V}$ be objects in $\Wfrak(H^0\A)$.
    The induced map of Hom-sets
    \begin{equation}\label{eq:sameasBor_HomMap}
    \begin{tikzcd}
        \Wfrak(H^0\A)(\eqclass{U},\eqclass{V}) \arrow[r] & \tcmc{\A}([\A \xrightarrow{\varpi_U} {L_U\A}],[\A \xrightarrow{\varpi_V} {L_V\A}])
    \end{tikzcd}
    \end{equation}
    will now be shown to be bijective. Fixing $U\in\eqclass{U}$ and $V\in\eqclass{V}$, one can interpret the domain as the subset of $\presilt_U(H^0\A)$ consisting of the objects
    $V'$ such that $\thick(V')=\thick(V)$ {\cite[Cor. 3.6]{Kai23}}. The codomain can be interpreted as the subset of $\presilt_U(H^0\A)$ consisting of the objects $V'$ such that the dg localisations ${L_U\A} \xrightarrow{\varpi_{V'}} {L_{V'}}({L_U\A})$ and ${L_U\A} \xrightarrow{\varpi_{V}} {L_V}({L_U\A})$ are equivalent. Under these interpretations, the map in \eqref{eq:sameasBor_HomMap} can be taken to be the identity map, whence the proof is complete.
\end{proof}

\begin{remark}
	Let $\calC$ and $\Lambda$ be as in \cref{setting} and assume $\calC$ is reduced. Recently, the ``$\tau$-cluster morphism category'' of $\calC$ was introduced in \cite{Nonis2026} and poses an alternative generalisation of {$\mathfrak{T}(\Lambda)$} to the picture category $\Wfrak(\calC)$. Let $\mathfrak{M}(\calC)$ denote the $\tau$-cluster morphiusm category of $\calC$ as introduced in \cite[{§}7]{Nonis2026}. It follows from the fact that $\calZ_U = \calZ_V$ if and only if $U \simeq V$ for two presilting objects $U,V$ in the reduced category $\calC$, that $\mathfrak{M}(\calC)$ is equivalent to the poset category of the $\bfg$-vector fan $\partfan{\calC}$. Viewing the poset category as a partitioned fan with trivial partition, and since $\partfan{\calC}= \partfan{\Lambda}$, the existence of a dense, faithful functor $\mathfrak{M}(\calC) \to \Wfrak(\Lambda)$, see \cite[Thm. 8.3]{Nonis2026}, is a consequence of \cite[Thm. 1.3(a)]{Kai23}. That this functor is a discrete fibration whose induced equivalence relation on $\mathfrak{M}(\Lambda)$ yields {$\mathfrak{T}(\Lambda)$}, see \cite[Thm. 8.3, Thm. 8.8]{Nonis2026}, follows from the fact that the cones of $\starr(\sigma)$ are in bijection with the cones of the fan $\pi_{\sigma}(\starr(\sigma))$ for all $\sigma \in \partfan{\Lambda}$ and underpins the construction of $\Wfrak(\Lambda)$ from $\partfan{\Lambda}$ in \cite{STTW23}. In other words, the new results on $\tau$-cluster morphism categories of \cite{Nonis2026} also fit well into the framework of partitioned fans. 
\end{remark}

\section{Morphisms of partitioned fans and faithful group functors}\label{sec:morpartfan}

Recall the definition of a morphism of fans in \Cref{def:morfan}. It is given by a linear transformation of the underlying {$\R$-vector} spaces, such that the image of each cone in the domain is contained in a cone in the codomain. We will soon introduce morphisms of partitioned fans. In order to deduce some applications concerning picture categories, we prove a useful lemma.

\begin{lemma}\label{lem:morpartfan}
    (cf. \cite[Second exercise on p. 42]{Ful93})
    Let $N_1$ and $N_2$ be finite lattices and let 
    \[f\colon ({(N_1)_{\R}},\Sfan_1) \xrightarrow{}({(N_2)_{\R}},\Sfan_2)\]
    be a morphism of fans. For any $\sigma\in \Sfan_1$, the poset of cones in $\Sfan_2$ containing $f(\sigma)$ admits a unique minimal element, denoted $\sigma^f$.
\end{lemma}
\begin{proof}
    Since the conclusion clearly holds when $f(\sigma)=\{0\}$, we may safely assume that this is not the case. 
    Under this additional assumption, we first prove that the {relative} interior $f(\sigma)^{\circ}$ of $f(\sigma)$ must be contained in the {relative} interior of some cone ${\kappa}_{\sigma}\in\Sfan_2$. Let ${\kappa}$ be a non-zero cone in $\Sfan_2$ containing $f(\sigma)$. If $f(\sigma)^{\circ}$ is not contained in ${\kappa}^{\circ}$, then $f(\sigma)$ must be contained in the boundary of ${\kappa}$ {since $f$ is a morphism of fans}, which would put $f(\sigma)$ inside a face ${\kappa}_1$ of ${\kappa}$. Repeating the same procedure will produce a chain of cones in $\Sfan_2$
    \begin{equation*}
        {\kappa} > {\kappa}_1 > {\kappa}_2 > \cdots,
    \end{equation*}
    where each term is a face of the previous, and each term contains $f(\sigma)$. If $f(\sigma)^{\circ}$ is not contained in the relative interior of any cone, this chain must terminate at the cone $\{0\}$, whence $f(\sigma)=\{0\}$. We have reached a contradiction, whence $f(\sigma)^{\circ}$ must be contained in the relative interior of some cone ${\kappa}_{\sigma}\in\Sfan_2$.
    The cone ${\kappa}_{\sigma}$ is uniquely determined, for the simple reason that the relative interiors of cones in $\Sfan_2$ cannot intersect {by the definition of a fan}. 
    Any cone strictly contained in ${\kappa}_{\sigma}$ is contained in a face of ${\kappa}_{\sigma}$, so it is outside the relative interior of ${\kappa}_{\sigma}$, whence it follows that ${\kappa}_{\sigma}$ must be the minimal element of the poset of cones in $\Sfan_2$ containing $f(\sigma)$.
\end{proof}

If $f$ is as in \cref{lem:morpartfan} and $\sigma,\kappa\in \Sfan_1$, then $\sigma \subseteq \kappa$ if and only if $\sigma^f \subseteq \kappa^f$. As a first application, the assertion in \cref{lem:morpartfan} will be used to define a notion of morphism of partitioned fans.

\begin{definition}\label{def:morpartfan}
    Let $N_1$ and $N_2$ be finite lattices and let $\Sfan_i$ be a fan in $(N_i)_\R$ for $i=1,2$. Let $(\Sfan_1,\mathfrak{P}_1)$ and $(\Sfan_2,\mathfrak{P}_2)$ be partitioned fans. A \textit{morphism of partitioned fans} from $(\Sfan_1,\mathfrak{P}_1)$ to $(\Sfan_2,\mathfrak{P}_2)$ is given by a morphism of fans $f: ((N_1)_\R,\Sfan_1) \to ((N_2)_\R, \Sfan_2)$ such that
    if $\sigma_1 \sim_{\Pfrak_1} \sigma_2$, then $\sigma_1^f \sim_{\Pfrak_2} \sigma_2^f$, where $\sigma_1^f$ and $\sigma_2^f$ are as in \Cref{lem:morpartfan}. We say that 
    \begin{enumerate}
        \item $f$ is \textit{injective} if $\sigma_1^f = \sigma_2^f$ implies $\sigma_1 = \sigma_2$,
        \item $f$ is \textit{surjective} if for every cone ${\kappa}\in\Sfan_2$ there exists a $\sigma\in\Sfan$ such that $\sigma^f={\kappa}$,
        \item $f$ is an \textit{isomorphism} if it admits an inverse morphism of partitioned fans.
    \end{enumerate} 
    We sometimes write $f:((N_1)_\R, \Sfan_1, \Pfrak_1) \to ((N_2)_\R, \Sfan_2, \Pfrak_2)$ for a morphism of partitioned fans. 
\end{definition}

\begin{remark}\label{eg:partmorphgivesfunctor_Lambda} 
    In \Cref{prop:samePartition}, we have implicitly seen examples of isomorphisms of partitioned fans. For instance, in \Cref{prop:samePartition}\eqref{thm:modulecatreduction_piccat1}, we established an isomorphism of partitioned fans
    \[(\K_0(\add \Lambda)_\R,\partfan{\Lambda}, \Pfrak_{\thick})\xrightarrow{} (\K_0(\add \Lambda)_\R, \partfan{\Lambda}, \Pfrak_{\textnormal{WAC}}).\]
    Similarly, in \Cref{prop:samePartition}\eqref{thm:modulecatreduction_piccat2}, we established an isomorphism of partitioned fans
    \[ (\K_0(\add T)_\R, \partfan{\calC}, \Pfrak_{\thick}^{\calC}) \to (\K_0(\add \Lambda)_\R, \partfan{\Lambda}, \Pfrak_{\thick}^{\Ktwo(\proj \Lambda)}),\]
    where $\Pfrak_{\thick}^\calC$ is the thick partition induced by $\calC$ and $\Pfrak_{\thick}^{\Ktwo(\proj \Lambda)}$ is the thick partition induced by $\Ktwo(\proj \Lambda)$. 
\end{remark}

{We} give some further important examples of morphisms of partitioned fans. Given partitions $\Pfrak$ and $\Pfrak'$ on the same fan $\Sfan$, we say that $\Pfrak$ is \textit{finer} than $\Pfrak'$ if each constituent partition $\Pfrak_{\sigma}$ is \textit{finer} than $\Pfrak'_{\sigma}$, where $\sigma\in \Sfan$.

\begin{lemma}\label{lem:finerpartmorph}
    Let $N$ be a finite lattice and let $\Pfrak_1$ and $\Pfrak_2$ be two {admissible} partitions of a fan $\Sfan$ in $N_\R$ such that $\Pfrak_1$ is finer than $\Pfrak_2$. Then the identity transformation $\textnormal{id}: N_\R \to N_\R$ induces an injective morphism of partitioned fans $(N_\R, \Sfan, \Pfrak_1) \to (N_\R, \Sfan, \Pfrak_2)$. 
\end{lemma}
\begin{proof}
    Let $\sigma_1, \sigma_2 \in \Sfan$ be such that $\sigma_1 \sim_{\Pfrak_1} \sigma_2$. The identity transformation $\textnormal{id}: N_\R \to N_\R$ trivially satisfies $\sigma_i^{\textnormal{id}} = \sigma_i$ for $i=1,2${, making {the assignment} injective}. Since $\Pfrak_1$ is finer than $\Pfrak_2$ we get, by definition, that
    \[ \sigma_1 \sim_{\Pfrak_1} \sigma_2 \quad \Longrightarrow \quad \sigma_1^{\textnormal{id}} = \sigma_1 \sim_{\Pfrak_2} \sigma_2 = \sigma_2^{\textnormal{id}}.\]
    Therefore $\textnormal{id}: N_\R \to N_\R$ induces a morphism of partitioned fans, as required. 
\end{proof}

Recall that a \textit{subfan} of a fan $\Sfan$ is given by a subset $\Sfan'$ of $\Sfan$ which is closed under taking faces, that is, forms a fan itself. It is clear that the identity transformation on the underlying $\R$-vector space $N_\R$ of $\Sfan$ then induces a morphism of fans $(N_\R, \Sfan') \to (N_\R,\Sfan)$. 

\begin{lemma}\label{lem:subpartmorph}
    Let $N$ be a finite lattice and Let $(\Sfan,\Pfrak)$ be a partitioned fan in $N_\R$. Let $\Sfan'$ be a subfan of $\Sfan$, and let $\Pfrak'$ denote {the} partition of $\Sfan'$ obtained by restricting $\Pfrak$.
    The identity transformation $\textnormal{id}\colon N_\R\xrightarrow{} N_\R$ induces an injective morphism of partitioned fans 
    $(N_\R, \Sfan',\Pfrak') \xrightarrow{} (N_\R, \Sfan,\Pfrak)$.
\end{lemma}
\begin{proof}
    {It is obvious that $\Pfrak'$ is an admissible partition of $\Sfan'$, which makes the statement well-defined.} It is immediate from the definitions that if $\sigma_1 \sim_{\Pfrak'} \sigma_2$, then $\sigma_1 \sim_{\Pfrak} \sigma_2$, so
    \begin{equation*}
        \sigma_1^{\textnormal{id}} = \sigma_1 \sim_{\Pfrak} \sigma_2 = \sigma_2^{\textnormal{id}},
    \end{equation*}
    which shows {that} $\textnormal{id}$ induce{s} a morphism of partitioned fans. Injectivity follows from that fact that $\sigma^{\textnormal{id}}=\sigma$ for all $\sigma\in\Sfan'$. 
\end{proof}

A similar result holds for the orthogonal projection of the star of a cone.

\begin{lemma}\label{lem:projmorpart}
    Let $N$ be a finite lattice and let $(\Sfan,\Pfrak)$ be a partitioned fan in $N_\R$. Let $\sigma$ be a cone in $\Sfan$. Consider the partitioned subfan $({\overline{\starr}(\sigma)},\Pfrak')$ (see \Cref{lem:subpartmorph}). Consider the orthogonal projection $\pi_{\sigma}\colon N_\R \xrightarrow{} \spann(\sigma)^{\perp}$ and the induced morphism of fans $(N_\R, \overline{\mathrm{star}}(\sigma)) \xrightarrow{}( \sigma^\perp, \pi_{\sigma}(\starr(\sigma)))$, see \Cref{eg:orthproj}\eqref{eg:orthproj_proj}.
    \begin{enumerate}
        \item\label{lem:projmorpart1} There is an isomorphism of posets 
        $\Pi_{\sigma}: {\mathrm{star}}(\sigma)\xrightarrow{}\pi_{\sigma}(\starr(\sigma))$, with both sets ordered by inclusion. 
        \item\label{lem:projmorpart2} We can define an admissible partition $\Pfrak_{\sigma}$ on $\pi_{\sigma}(\starr(\sigma))$ which identifies ${\kappa}_1$ and ${\kappa}_2$ whenever the inverse images under the bijection {$\Pi_{\sigma}$} are identified by $\Pfrak'$.
        \item\label{lem:projmorpart3} The orthogonal projection $\pi_{\sigma}\colon N_\R \xrightarrow{} \spann(\sigma)^{\perp}$ induces a morphism of partitioned fans 
        \[(N_\R, \overline{\mathrm{star}}(\sigma),\Pfrak')\xrightarrow{}(\sigma^\perp, \pi_{\sigma}(\starr(\sigma)),\Pfrak_{\sigma}).\]
        \item\label{lem:projmorpart4} There is a fully faithful functor
	\[  H_{\sigma}: \catof(\pi_{\sigma}(\starr(\sigma)),\Pfrak_{\sigma}) \to \catof(\overline{\mathrm{star}}(\sigma),\Pfrak') .\]
    \end{enumerate}
\end{lemma}
\begin{proof}
    The claim in \eqref{lem:projmorpart1} {is easy to show, see \cite[p. 52]{Ful93}}.
    For \eqref{lem:projmorpart2}, {first} l{et $\kappa_1, \kappa_2 \in \pi_{\sigma}({\starr}(\sigma))$. By \eqref{lem:projmorpart1}, there exist uniquely defined $\nu_1, \nu_2 \in \starr(\sigma)$ such that $\kappa_i = \pi_{\sigma}(\nu_i)$ for $i=1,2$. Assume that $\nu_1 \sim_{\Pfrak} \nu_2$. By definition, $\spann_{N_\R}(\nu_1) = \spann_{N_\R}(\nu_2)$ which implies $\pi_{\sigma}(\spann_{N_\R}(\nu_1)) = \pi_{\sigma}(\spann_{N_\R}(\nu_2))$ because $\pi_{\sigma}$ is a linear transformation, which also implies that thus }
    \[ \spann_{\sigma^\perp}(\pi_{\sigma}(\nu_1))= \spann_{\sigma^\perp}(\pi_{\sigma}(\nu_2)). \]
    {Similarly $\nu_1 \sim_{\Pfrak} \nu_2$ implies $\pi_{\nu_1}(\starr(\nu_1)) = \pi_{\nu_2}(\starr(\nu_2))$ {by definition}. It follows from basic linear algebra (see \cite[Lem. 3.2]{Kai23}) that $\proj_{\nu_1} \circ \pi_{\sigma}(\starr(\nu_1)) = \proj_{\nu_2} \circ \pi_{\sigma}(\starr(\nu_2))$, using the notation of \Cref{eg:orthproj}\eqref{eg:orthproj_proj}. From this it follows that }
    \[ \proj_{\nu_1} ( \starr_{\pi_{\sigma}(\starr(\sigma))}( \pi_{\sigma}(\nu_1))) = \proj_{\nu_2} ( \starr_{\pi_{\sigma}(\starr(\sigma))}( \pi_{\sigma}(\nu_2))). \]
    The restriction of $\proj_{\nu_i}: M_\R \to \nu_i^\perp$ to the subspace $\sigma^\perp \subseteq M_\R$ is $\proj_{\pi_{\sigma}(\nu_i)}$ for $i=1,2$. It follows that 
    \begin{equation}\label{eq:projprojequiv} \proj_{\pi_{\sigma}(\nu_1)} ( \starr_{\pi_{\sigma}(\starr(\sigma))}(\pi_{\sigma}(\nu_1))) = \proj_{\pi_{\sigma}(\nu_2)} ( \starr_{\pi_{\sigma}(\starr(\sigma))}(\pi_{\sigma}(\nu_2))). \end{equation}
    From linear algebra, we know that for a subspace $U$ of a vector space $V$ with a fixed basis $B$, there is an isomorphism between the dual of the quotient space $(V/U)^*$, with respect to the basis induced by $B$, and $U^\perp \subseteq V^*$. Substituting in $V = \sigma^\perp$, and $U = \proj_{\pi_{\sigma}(\nu_1)}(\sigma^\perp) = \nu_i^\perp$, we conclude from \cref{eq:projprojequiv}, that 
    \[ \pi_{\kappa_1}(\starr_{\pi_{\sigma}(\starr(\sigma))}(\kappa_1)) = \pi_{\kappa_2}(\starr_{\pi_{\sigma}(\starr(\sigma))}(\kappa_2))\]
   as required.  {This shows that $\Pfrak_{\sigma}$ satisfies the conditions of \cref{def:admisspart}\eqref{def:admisspart_1} for a partition of $\overline{\starr}(\sigma)$. We are left with showing that it is admissible. For this, let $\nu_i, \kappa_i \in \starr(\sigma)$ be such that $\nu_i \subseteq \kappa_i$ for $i=1,2$. Assume that $\pi_{\sigma}(\nu_1) \sim_{\Pfrak_{\sigma}} \pi_{\sigma}(\nu_2)$ and that
   \begin{equation}\label{eq:projinheritpart} \pi_{ \pi_{\sigma}(\nu_1)}(\pi_{\sigma}(\kappa_1)) = \pi_{\pi_{\sigma}(\nu_2)}(\pi_{\sigma}(\kappa_2)). \end{equation}
   We need to show that $\pi_{\sigma}(\kappa_1) \sim_{\Pfrak_{\sigma}} \pi_{\sigma}(\kappa_2)$, which is equivalent to showing $\kappa_1 \sim_{\Pfrak'} \kappa_2$. From the identification $\pi_{\sigma}(\nu_1) \sim_{\Pfrak_{\sigma}} \pi_{\sigma}(\nu_2)$ we obtain, by the definition of $\Pfrak_\sigma$, that $\nu_1 \sim_{\Pfrak'} \nu_2$. The equality in \cref{eq:projinheritpart} gives rise to a chain of implications
   \begin{align*}
   	\pi_{ \pi_{\sigma}(\nu_1)}(\pi_{\sigma}(\kappa_1)) = \pi_{\pi_{\sigma}(\nu_2)}(\pi_{\sigma}(\kappa_2)) & \Rightarrow \proj_{\pi_{\sigma}(\nu_1)}(\pi_{\sigma}(\kappa_1)) = \proj_{\pi_{\sigma}(\nu_2)}(\pi_{\sigma}(\kappa_2)) \\
	& \Rightarrow \proj_{\nu_1} (\pi_{\sigma}(\kappa_1)) = \proj_{\nu_2} (\pi_{\sigma}(\kappa_2)) \\
	& \Rightarrow \pi_{\nu_1}(\kappa_1) = \pi_{\nu_2}(\kappa_2) \\
	& \Rightarrow \kappa_1 \sim_{\Pfrak'} \kappa_2
   \end{align*}
   as required. The first three implications follow from the same linear algebraic arguments used around \cref{eq:projprojequiv} and the final implication is a consequence of $\nu_1 \sim_{\Pfrak'} \nu_2$ and the admissibility of $\Pfrak'$. This shows that $\Pfrak_{\sigma}$ is admissible. }
     It is clear from the definition of $\Pfrak_{\sigma}$ that we have defined a morphism of partitioned fans. 

     To conclude, we prove \eqref{lem:projmorpart4}. We define the functor $H_{\sigma}$ {using the bijection $\Pi_{\sigma}$ from \eqref{lem:projmorpart1} as follows}:
    \begin{align*}
        \eqclass{{\kappa}} &\mapsto \eqclass{\Pi_{\sigma}^{-1}(\kappa)}, \\
        \sqbinom{{\kappa}}{\rho} &\mapsto \sqbinom{\Pi_{\sigma}^{-1}(\kappa)}{\Pi_{\sigma}^{-1}(\rho)}.
    \end{align*}
    {This assignment is well-defined because $\Pi_{\sigma}^{-1}$ is a poset isomorphism and the equivalence relation on the domain is induced from the equivalence relation on the codomain, see \eqref{lem:projmorpart2}.
    Indeed, $H_{\sigma}$ is easily checked to be a functor because, by \eqref{lem:projmorpart1}, $H_{\sigma}$ {is induced by} a subposet inclusion. It is fully faithful as an immediate consequence of the bijectivity of $\Pi_{\sigma}$.}
\end{proof}

{We now illustrate how the combinatorial notion of morphism of partitioned fans has powerful functorial consequences for the corresponding categories. Subsequently this is used to recover many known properties of $\tau$-cluster morphism categories and picture categories.}

\begin{theorem}\label{thm:partmorphgivesfunctor} Let $N_1$ and $N_2$ be finite lattices and let $\Sfan_i$ be a fan in $(N_i)_\R$ for $i=1,2$. Let $f: ((N_1)_\R, \Sfan_1,\mathfrak{P}_1) \to ((N_2)_\R, \Sfan_2,\mathfrak{P}_2)$ be a morphism of partitioned fans.
    \begin{enumerate}
        \item\label{thm:partmorphgivesfunctor_welldef} The morphism $f$ induces a functor $\catof f\colon\catof(\Sfan_1,\mathfrak{P}_1)\xrightarrow{}\catof(\Sfan_2,\mathfrak{P}_2)$, sending the object $\eqclass{\sigma}$ to $\eqclass{\sigma^f}$, and the morphism $\sqbinom{{\kappa}}{\sigma}$ to $\sqbinom{{\kappa}^f}{\sigma^f}$. 
        \item\label{thm:partmorphgivesfunctor_inj} {If $f$ is injective, then the functor $\catof f$ is faithful.}
        \item\label{thm:partmorphgivesfunctor_surj} {If $f$ is surjective, then the functor $\catof f$ is dense.}
        \item\label{thm:partmorphgivesfunctor_adjbis_R2} Suppose that $\catof f$ admits a right adjoint $R$. Then all objects in $\catof{(}\Sfan_1{, \Pfrak_1)}$ of the form $\eqclass{\mu}$, where $\mu$ is a maximal cone in $\Sfan_1$, are in the essential image of $R$.
    \end{enumerate}
\end{theorem}
\begin{proof}
    Throughout the proof, we simplify the notation by setting $\catof_i\coloneqq \catof(\Sfan_i,\Pfrak_i)$ for $i=1,2$.

    The well-definedness of the maps of objects is a direct consequence of \Cref{def:morpartfan}. For morphisms, suppose that $\sqbinom{{\kappa}_1}{\sigma_1}=\sqbinom{{\kappa}_2}{\sigma_2}$ as morphisms in $\catof_1$. This entails that $\pi_{\sigma_1}({\kappa}_1) = \pi_{\sigma_2}({\kappa}_2)$, where $\pi_{\sigma_i}$ denotes the orthogonal projection onto $\sigma_i^{\perp}$. {In order to show that the map of morphisms is well-defined, it is to be shown that $\pi_{\sigma_1^f}({\kappa}_1^f) = \pi_{\sigma_2^f}({\kappa}_2^f)$ as cones in $\spann(\sigma_1^f)^{\perp}$.} Having assumed that $\pi_{\sigma_1}({\kappa}_1) = \pi_{\sigma_2}({\kappa}_2)$, the linearity of $f$ implies that $\pi_{f(\sigma_1)}(f({\kappa}_1)) = \pi_{f(\sigma_2)}(f({\kappa}_2))$. It is possible to choose a point $p_1$ in the {relative} interior of $f({\kappa}_1)$, and a point $p_2$ in the {relative} interior of $f({\kappa}_2)$, such that $\pi_{f(\sigma_1)}(p_1) = \pi_{f(\sigma_2)}(p_2)$. Since $f(\sigma_i)\subseteq \sigma_i^f$ for $i=1,2$, we have that {$\spann(f(\sigma_i)) = \spann(\sigma_i^f)$ for $i=1,2$, from which} $\pi_{\sigma_1^f}(p_1) = \pi_{\sigma_2^f}(p_2)$ {follows}. Since $p_i$ is in the {relative}  interior of ${\kappa}_i^f$ by construction, it follows that the cones $\pi_{\sigma_1^f}({\kappa}_1^f)$ and $ \pi_{\sigma_2^f}({\kappa}_2^f)$ in 
    $((\sigma_1^f)^{\perp}$ intersect in their {relative} interior{s}, so they are equal, as required for the map of morphisms to be well-defined. To complete the proof of \eqref{thm:partmorphgivesfunctor_welldef}, we show that we have defined a functor. Given composable morphisms $\sqbinom{{\kappa}}{\sigma}$ and $ \sqbinom{\sigma}{\rho}$ in $\catof_1$, we have
    $
        \sqbinom{{\kappa}^f}{\sigma^f}\circ \sqbinom{\sigma^f}{\rho^f} = \sqbinom{{\kappa}^f}{\rho^f}
    $
    as morphisms in $\catof_2$, so \eqref{thm:partmorphgivesfunctor_welldef} holds.

    We now prove \eqref{thm:partmorphgivesfunctor_inj}. 
    {A set of representatives of $\Hom_{\catof_1}(\eqclass{\sigma_1}, \eqclass{{\kappa}_1})$ is given by the morphisms $\binom{{\kappa}'}{\sigma_1}$, where ${\kappa}' \sim_{\Pfrak_1} {\kappa}_1$ by \cite[Cor. 3.6]{Kai23}.  Assume for the sake of contradiction that $\catof f$ is not faithful. Then there exist distinct morphisms $\binom{{\kappa}_1}{\sigma_1} \not \sim_{\Pfrak_1} \binom{{\kappa}_2}{\sigma_1}$, where ${\kappa}_1 \sim {\kappa}_2$ are distinct identified cones, such that $\binom{{\kappa}_1^f}{\sigma_1^f} \sim_{\Pfrak_2} \binom{{\kappa}_2^f}{\sigma_1^f}$. By definition, this means {$\pi_{\sigma_1^f}(\kappa_1^f) = \pi_{\sigma_1^f}(\kappa_2^f)$, but this implies} ${\kappa}_1^f = {\kappa}_2^f$ by \cref{lem:projmorpart}\eqref{lem:projmorpart1}, a contradiction to the injectivity of $f$. Therefore, $\catof f$ is faithful whenever $f$ is injective.} 

    We move on to prove
    \eqref{thm:partmorphgivesfunctor_surj}. We have assumed that for all cones ${\kappa}$ in $\Sfan_2$, there exists a cone $\sigma$ in $\Sfan_1$ such that $\sigma^f={\kappa}$. In particular, we have $\eqclass{\sigma^f}=\eqclass{{\kappa}}$. Since ${\kappa}$ was arbitrarily chosen, the claim follows.

To prove \eqref{thm:partmorphgivesfunctor_adjbis_R2}, it suffices to point out the any morphism of the form $\eqclass{\mu}\xrightarrow{} \eqclass{\rho}$, where $\mu$ is maximal cone, is an identity morphism. Indeed, it follows that the unit $\eqclass{\mu}\xrightarrow{\eta_{\eqclass{\mu}}} R\circ\catof f\eqclass{\mu}$ is an isomorphism, which puts $\eqclass{\mu}$ in the essential image of $R$.
\end{proof}

{We deduce that certain extriangulated functors between 0-Auslander extriangulated categories induce morphisms of partitioned fans. }

\begin{corollary}\label{cor:partmorphgivesfunctor}
    Let ${G}\colon \calC_1\xrightarrow{}\calC_2$ be a functor of 0-Auslander extriangulated $k$-categories, both satisfying \Cref{setting}. {Assume that $G$ sends presilting objects to presilting objects and maps any projective silting object in $\calC_1$ to a projective silting object in $\calC_2$ ({{so that $G$ satisfies} the assumptions imposed in \Cref{lem:FanIso}\eqref{lem:FanIso2}}). Then the following hold:}
    \begin{enumerate}
        \item\label{cor:partmorphgivesfunctor_welldef} {The functor} $G$ induces a morphism of partitioned fans \[(\K_0(\add T_1)_\R,\partfan{\calC_1}, (\Pfrak_{\thick})_1)\xrightarrow{} (\K_0(\add T_2)_\R,\partfan{\calC_2}, (\Pfrak_{\thick})_2{)},\] and thus, by \Cref{thm:partmorphgivesfunctor}\eqref{thm:partmorphgivesfunctor_welldef}, a functor of picture categories $\Wfrak {G}\colon\Wfrak(\calC_1)\xrightarrow{}\Wfrak(\calC_2)$.
        \item\label{cor:partmorphgivesfunctor_inj} If ${G}$ induces an injective map of presilting objects, then the functor $\Wfrak {G}$ is faithful.
        \item\label{cor:partmorphgivesfunctor_surj} If ${G}$ induces a surjective map of presilting objects, then the functor $\Wfrak {G}$ is dense.
    \end{enumerate}
\end{corollary}
\begin{proof}
    It suffices to prove \eqref{cor:partmorphgivesfunctor_welldef}, since the remaining assertions are then deduced from \Cref{thm:partmorphgivesfunctor}\eqref{thm:partmorphgivesfunctor_inj}--{\eqref{thm:partmorphgivesfunctor_surj}}. Having shown in \Cref{lem:FanIso} that {$G$} defines {the desired} morphism of $\bfg$-vector fans, it is only a matter of showing that the partitions are respected.
    Given two presilting objects $P_1,P_2\in\presilt(\calC_1)$ such that $\thick(P_1)=\thick(P_2)$, then
    \begin{equation*}
        \thick(G(P_1)) = \thick(G(\thick(P_1))) = 
\thick(G(\thick(P_2))) = \thick(G(P_2)),
    \end{equation*}
 {where the first an last equalities hold for the following reason: because $G$ is an extriangulated functor {we have $G(\thick {(}P_i{)} \subseteq \thick(G(P_i))$ and applying $\thick(-)$ to both sides gives one inclusion, the other follows from the fact that $\thick(G(P_i))$ is the smallest thick subcategory containing $G(P_i))$ for $i=1,2$.}} {This completes} completes the proof.
\end{proof}

{We return to the setting of \cref{lem:projmorpart}.}

\begin{lemma}\label{lem:projmoradj}
	In the setting of \cref{lem:projmorpart}, the functor 
	\[\catof \pi_{\sigma}: \catof(\overline{\starr}(\sigma), \Pfrak') \to \catof( \pi_{\sigma}(\starr (\sigma)), \Pfrak_\sigma), \]
	which is induced by the morphism of partitioned fans in \cref{lem:projmorpart}\eqref{lem:projmorpart3} via \cref{thm:partmorphgivesfunctor}\eqref{thm:partmorphgivesfunctor_welldef} is left adjoint to the fully faithful functor $H_{\sigma}$ of \cref{lem:projmorpart}\eqref{lem:projmorpart4}.
\end{lemma}
\begin{proof}
	Write $\catof_1 \coloneqq \catof(\overline{\starr}(\sigma)$ and $\catof_2 \coloneqq \catof( \pi_{\sigma}(\starr (\sigma)), \Pfrak_\sigma)$. For $\eqclass{\nu} \in \catof_1$ and $\eqclass{\rho} \in \catof_2$, we are to set up a natural bijection
\begin{equation}\label{eq:adjbijection}
    \catof_2(\catof \pi_{\sigma} \eqclass{{\nu}},\eqclass{\rho}) \cong
    \catof_1(\eqclass{{\nu}},H_{\sigma} \eqclass{\rho}).
\end{equation}
Since $\catof \pi_{\sigma} \eqclass{{\nu}} = \eqclass{{\pi_{
    \sigma}}{\nu}}$, we simplify the left hand side to $\catof_2(\eqclass{{\pi_{\sigma}}{(}{\nu}{)}},\eqclass{\rho})$. On the other hand, one may express the right hand side by $\catof_1(\eqclass{{\nu}},\eqclass{{\Pi_{\sigma}^{-1}}(\rho)})$. A set of representatives in $\catof_2(\eqclass{{\pi_{
    \sigma}}{(}{\nu}{)}},\eqclass{\rho})$ can be given by the symbols $\binom{\rho'}{\pi_{\sigma}{(}{\nu}{)}}$, where $\rho' \sim_{\Pfrak_{_\sigma}} \rho$ \cite[Cor. 3.6]{Kai23}, and one expresses the elements of $\catof_1(\eqclass{{\nu}},H_{\sigma}\eqclass{\rho})$ {using the symbols $\binom{\rho''}{\nu}$, where $\rho'' \sim_{\Pfrak} \Pi_{\sigma}^{-1} (\rho) $} in a similar manner. 
    
    The bijection $\Pi_{\sigma}$ restricts to a bijection between $\starr(\nu) \subseteq \catof_2$ and $\starr(\pi_{\sigma}(\nu)) \subseteq \catof_2$. Since the partition $\Pfrak_{\sigma}$ {is induced by} the partition $\Pfrak$, it is straightforward to use the bijection $\Pi_{\sigma}$ to set up the bijection in \eqref{eq:adjbijection}, which completes the proof. 
\end{proof}

We deduce some useful results for picture categories of 0-Auslander extriangulated categories.

\begin{lemma}\label{lem:redadj}
	Let $\calC$ be as in \Cref{setting}. Let $N' \in \calC$ be a projective-injective object and let $P \in \presilt \calC$.
	\begin{enumerate}
		\item\label{lem:redadj1} There is an adjoint pair of functors
		\begin{equation*}
    		\begin{tikzcd}
       		 \Wfrak(\calC/[N'])\arrow[r,hook,bend right, "R"] & \Wfrak(\calC)\arrow[l, "L",swap],
   		 \end{tikzcd}
   		 \end{equation*} 
		 where $R$ is fully faithful.
		\item\label{lem:redadj2} There is an adjoint pair of functors
		\begin{equation*}
    		\begin{tikzcd}
       		 \Wfrak(\calZ_P/[P])\arrow[r,hook,bend right, "R"] & \Wfrak(\calZ_P)\arrow[l, "L",swap],
   		 \end{tikzcd}
   		 \end{equation*} 
		 where $R$ is fully faithful.
		 \item\label{lem:redadj3} There is a fully faithful functor
		 \[ \Wfrak(\calZ_P/[P]) \to \Wfrak(\calC). \]
	\end{enumerate}
\end{lemma}
\begin{proof}
We first show \eqref{lem:redadj1}. Since $N'$ is projective-injective in $\calC$, it is a direct summand of every silting object, see \cref{rem:silt}\eqref{rem:proj-inj-silt}. Since every presilting object in $\calC$ has a Bongartz completion, which is a silting object, this implies that $\overline{\starr}(\boldC(N')) = \partfan{\calC}$. 
	Similarly, because $N'\oplus U$ is prestilting for all presilting objects $U$ in $\calC$, we have an equivalence of categories $\calZ_{N'}/[N'] \simeq \calC/[N']$. Therefore, \cref{lem:ZUfan}\eqref{lem:ZUfan2} gives an isomorphism of fans 
	\[ \partfan{\calC/[N']} \simeq \partfan{\pi_{\boldC(N')}(\starr(\boldC(N')))}. \]
	By \cref{prop:FGPPP23} and \cref{lem:Ver2.3.1}, the thick subcategories of $\calC/[N']$ are induced by those of $\calC$. Using the notation of \cref{lem:projmorpart}\eqref{lem:projmorpart2}, the admissible partition $\Pfrak_{\thick}^{\calC/[N']}$ of $\partfan{\calC/[N']}$ is given by $(\Pfrak_{\thick}^{\calC})_{\boldC(N')}$. The claim in \eqref{lem:redadj1} follows directly from \cref{lem:projmoradj} and the definitions.
    
    The claim in \eqref{lem:redadj2} is an immediate consequence of \eqref{lem:redadj1} because $\calZ_P$ is a 0-Auslander extriangulated subcategory satisfying \Cref{setting} itself, and $P$ is projective-injective in $\calZ_P$, see \cref{prop:siltred}\eqref{prop:siltred_0Aus_proj}.

   To show \eqref{lem:redadj3}, we take the composite of the fully faithful functor $R$ in \eqref{lem:redadj2} and the faithful functor $\catof \mathrm{id}: \Wfrak(\calZ_P) \to \Wfrak(\calC)$ induced by the injective morphism of fans in \cref{lem:subpartmorph} via \cref{thm:partmorphgivesfunctor}\eqref{thm:partmorphgivesfunctor_welldef}\eqref{thm:partmorphgivesfunctor_inj}. It is not difficult to see, using \cite[Cor. 3.6]{Kai23}, that the composite functor is full even though the functor $\catof \mathrm{id}$ is not generally full.
\end{proof}

\begin{remark}\label{rem:generlises}
\Cref{cor:partmorphgivesfunctor} and \cref{lem:redadj} shed new light on previous work of the authors:
    \begin{enumerate}
        \item\label{eg:partmorphgivesfunctor_red}
    \Cref{lem:redadj}\eqref{lem:redadj3} recovers the behaviour of picture categories \cite[§6]{Bor24}, and indeed that of $\tau$-cluster morphism categories \cite[{Prop.} 6.14]{BH21} when $\calC = \Ktwo(\proj \Lambda)$.
        \item\label{rem:generlises_Kai23} Let $\Pfrak_1, \Pfrak_2$ be admissible partitions of a fan $\Sfan$ in $N_\R$. By \cref{lem:finerpartmorph} and \cref{thm:partmorphgivesfunctor}\eqref{thm:partmorphgivesfunctor_inj}, if $\Pfrak_1$ is a finer partition than $\Pfrak_2$, there exists a faithful functor $\catof(\Sfan, \Pfrak_1) \to \catof(\Sfan, \Pfrak_2)$, as shown previously by the second-named author \cite[Thm. 1.3(1)]{Kai23}.
        \item\label{rem:generlises_Bor24} 
        It is a consequence of \cref{lem:redadj}\eqref{lem:redadj1} that the classifying spaces of $\Wfrak(\calC)$ and $\Wfrak(\calC/[N'])$ are homotopy equivalent \cite[Corollary 1 in §1]{Qui73}. This assertion has been shown by the first-named author \cite[{Prop.} 6.6]{Bor24} under the assumption that $\calC$ admits an exact dg enhancement.
     
     \item\label{rem:generlises_Kai24} Let $\Lambda$ be a finite-dimensional $k$-algebra and $I$ an ideal of $\Lambda$.
     The derived tensor functor
     \begin{equation*}
         -\derotimes{\Lambda} \Lambda/I \colon \Ktwo(\proj \Lambda) \xrightarrow{} \Ktwo(\proj \Lambda/I)
     \end{equation*}
     is full, so it induces a map 
     \begin{equation*}
         -\derotimes{\Lambda} \Lambda/I \colon \presilt(\Lambda) \xrightarrow{} \presilt(\Lambda/I),
     \end{equation*}
     and in turn a morphism of fans 
    $
        \partfan{\Lambda} \xrightarrow{} \partfan{\Lambda/I}
    $
    {by \cref{lem:FanIso}\eqref{lem:FanIso1}}. Moreover, since
     \begin{equation*}
        \thick(P_1) = \thick(P_2) \Longrightarrow \thick(P_1\derotimes{\Lambda} \Lambda/I) = \thick(P_2\derotimes{\Lambda} \Lambda/I) 
    \end{equation*}
    for $P_1,P_2\in \presilt(\Lambda)$, the derived tensor functor induces a morphism of partitioned fans
    \begin{equation*}
        (\K_0(\add \Lambda)_\R, \partfan{\Lambda}, \Pfrak_{\thick}) \xrightarrow{} (\K_0(\add \Lambda/I)_\R, \partfan{\Lambda/I}, \overline{\Pfrak}_{\thick}).
    \end{equation*} \Cref{cor:partmorphgivesfunctor}\eqref{cor:partmorphgivesfunctor_welldef} then induces a functor $\Wfrak(\Lambda)\xrightarrow{}\Wfrak(\Lambda/I)$, which was previously constructed by the second-named author \cite[{Thm.} 1.3]{Kai24}. 
    \end{enumerate}
\end{remark}

We conclude this section by relating the existence of faithful group functors from picture categories of 0-Auslander extriangulated categories to the existence of faithful group functors from $\tau$-cluster morphism categories. The results below draw on previous work \cite{BorveKaipeltautilt}. In short, we generalise our results from finite-dimensional $k$-algebras to 0-Auslander extriangulated $k$-categories. First, the following lemma provides a fan-theoretic perspective on some results of \cite{BorveKaipeltautilt}.

\begin{proposition}\label{prop:BK}
    Let $K:k$ be a field extension and let and $\Lambda$ be a finite-dimensional $k$-algebra. We denote the scalar extension $\Lambda\otimes_{k}K$ by  $\Lambda^K$.
    \begin{enumerate}
        \item\label{prop:BK1} The functor $- \otimes_k K\colon \mods(\Lambda)\xrightarrow{} \mods(\Lambda^K)$ and its corresponding linear map ${\K}_0(\add(\Lambda)) \to \K_0({\add}(\Lambda^K))$ induce a{n injective} morphism of fans
        \begin{equation}
            \label{eq:BK1}
            \partfan{\Lambda} \to \partfan{\Lambda^K}.
        \end{equation}
    \item\label{prop:BK2} If $K:k$ is MacLane separable, the injective morphism in \eqref{eq:BK1} can be upgraded to an injective morphism of partitioned fans
    \begin{equation}
            \label{eq:BK2}
            (\partfan{\Lambda}, \Pfrak_{\thick}) \to (\partfan{\Lambda^K}, \Pfrak_{\thick}^K).
        \end{equation}
    \end{enumerate}
\end{proposition}
\begin{proof} \, 

    \eqref{prop:BK1} For any field extension $K:k$, the functor $- \otimes_k K$ induces an injective morphisms of $\bfg$-vector fans by \cite[Thm. 2.18]{BorveKaipeltautilt}. This can also be deduced from \cref{lem:FanIso}\eqref{lem:FanIso1} using \cite[Prop. 6.6, Cor. 6.7]{DIJ2019} and the injectivity of $- \otimes_k K$ on objects.
    
    \eqref{prop:BK2} Let $U$ and $V$ be two presilting objects in $\Ktwo(\proj \Lambda)$. {If $K:k$ is MacLane separable}, {then we have that
    \begin{align*}
        \boldC(U) \sim_{\Pfrak} \boldC(U) & \Leftrightarrow \thick(U) = \thick(V) \\
        & \Leftrightarrow J(\overline{U},\overline{\Omega I_U}) = J(\overline{V},\overline{\Omega I_V}) & \text{(by \cref{prop:Monica}\eqref{prop:Monica_maps})}\\
        & \Rightarrow J((\overline{U})^K,(\overline{\Omega I_U})^K) = J((\overline{V})^K,(\overline{\Omega I_V})^K) & \text{(by \cite[Prop. 4.9]{BorveKaipeltautilt})} \\
        & \Rightarrow J(\overline{U^K},\overline{\Omega I_{U^K}}) = J(\overline{V^K},\overline{\Omega I_{V^K}}) & \text{(by \cite[Lem. 2.7]{BorveKaipeltautilt})} \\
        & \Leftrightarrow \thick(U^K) = \thick(V^K) & \text{(by \cref{prop:Monica}\eqref{prop:Monica_maps})} \\
        & \Leftrightarrow \boldC(U^K) \sim_{\Pfrak^K} \boldC(V^K)
    \end{align*}
    See also the proof of \cref{prop:samePartition}, for the second and the penultimate implication. This completes the proof.}
\end{proof}

Let $\calC$ be a 0-Auslander extriangulated $k$-category satisfying \Cref{setting}. We say that the picture category $\Wfrak(\calC)$ \textit{admits a faithful group functor} if there exists a group $G$ and a faithful functor $\Wfrak(\calC)\xrightarrow{}G$, where we regard $G$ as a groupoid with a single object. In previous work \cite{BorMot,BorveKaipeltautilt}, we showed that picture categories of finite-dimensional $k$-algebras admit faithful group functors whenever $k$ is perfect. As a result, the classifying space of such a picture category is a $\mathrm{CAT}(0)$ space precisely when it satisfies \textit{compatibility of last factors} \cite{BH21,Igu14}, a combinatorial condition on the set of semibricks. By \Cref{thm:faithfulgroupfunctor} below, we may generalise these assertions to a broader context.

\begin{theorem}\label{thm:faithfulgroupfunctor}
    Let $\calC$ and $\Lambda$ be as in \Cref{setting}. Then $\Wfrak(\calC)$ admits a faithful group functor if and only if $\Wfrak(\Lambda)$ admits a faithful group functor. 
\end{theorem}

\begin{proof}
    The proof builds on \Cref{prop:samePartition}\eqref{thm:modulecatreduction_piccat2} and \eqref{rem:generlises_Bor24} \cref{lem:redadj}\eqref{lem:redadj1}. {Let $N$ be the maximal basic projective-injective object in $\calC$.} We consider the diagram of functors
    \begin{equation}\label{eq:faithfulgroupfunctor_toLambda}
    \begin{tikzcd}
        \Wfrak(\calC) \arrow[r,bend left] & \arrow[l,bend left] \Wfrak(\calC/[N])\arrow[r,"\simeq"] & \Wfrak(\Lambda),
    \end{tikzcd}
    \end{equation}
    where the former result provides the equivalence on the right, and the latter provides the adjunction on the left. It suffices to prove that both adjoints are faithful, but this is a consequence of \Cref{lem:monoepi}; since every morphism in $\Wfrak(\calC)$ is a monomorphism {and an epimorphism}, the unit (resp. counit) will induce a monomorphism (resp. an epimorphism) for all objects in $\Wfrak(\calC)$ and this is equivalent to the adjoint functors being faithful. We have proved the result.
\end{proof}

This allows for the extension of various results which have established such a faithful functor for the $\tau$-cluster morphism category of finite-dimensional algebras to the picture categories of 0-Auslander extriangulated categories. For an overview of the most general results in this direction, see \cite{BorveKaipeltautilt,Kai23,Treffinger26}.

\bibliographystyle{ourIEEEstyle.bst}
\bibliography{fieldext}

@article{Yur,
	year = 2018,
	volume = {23},
	pages = {35--47},
	author = {T. Yurikusa},
	title = {Wide subcategories are semistable},
	journal = {Doc. Math.}
}

@article{Kai24,
title = {$\tau$-cluster morphism categories of factor algebras},
journal = {J. Algebra},
volume = {701},
pages = {739-778},
year = {2026},
author = {Maximilian Kaipel},
}

@misc{Wan25,
      title={The {G}rothendieck group of an extriangulated category}, 
      author={Li Wang},
      year={2025},
      eprint={2508.12545},
      archivePrefix={arXiv},
      primaryClass={math.RT},
      note={arXiv:2508.12545}, 
}

@article{DIJ2019,
title = "$\tau$-tilting finite algebras, bricks, and $g$-vectors",
author = "L. Demonet and O. Iyama and G. Jasso",
year = "2019",
pages = "852--892",
journal = "Int. Math. Res. Not.",
volume = "2019",
issue = "3",
}

@article{AIR2014,
    title={$\tau$-tilting theory},
    volume={150},
    number={3},
    journal={Compos. Math.},
    author={T. Adachi and O. Iyama and I. Reiten},
    year={2014},
    pages={415–452}
}

@misc{Gar24,
  title={On $g$-finiteness in the category of projective presentations},
  author={Garcia, Monica},
  note={arXiv:2406.04134},
  year={2024}
}

@book{Mac67,
	title={Homology},
	author={MacLane, Saunders},
	year={2012},
	publisher={Springer Science \& Business Media}
}

@book{GZ12,
  title={Calculus of fractions and homotopy theory},
  author={Gabriel, Peter and Zisman, Michel},
  volume={35},
  year={2012},
  publisher={Springer Science \& Business Media}
}

@misc{BorMot,
    author       = {Erlend D. {B{\o}rve}},
    title        =  {Picture groups and green sequences from the perspective of {Ringel}--{Hall} algebras},
    note= {In preparation}
  }

@misc{Bor24,
  title={Silting reduction and picture categories of 0-{A}uslander extriangulated categories},
  author={B{\o}rve, Erlend D.},
  note={arXiv:2405.00593},
  year={2024}
}

@article{Kla22,
title = {n-extension-closed subcategories of n-exangulated categories},
journal = {Bull. Sci. Math.},
volume = {211},
pages = {103846},
year = {2026},
author = {Carlo Klapproth},
}

@article{Oga22,
  title={Abelian Categories from Triangulated Categories via {N}akaoka--{P}alu’s Localization},
  author={Ogawa, Yasuaki},
  journal={Appl. Categ. Struct.},
  volume={30},
  number={4},
  pages={611--639},
  year={2022},
  publisher={Springer}
}

@article{AT23,
 author = {Adachi, Takahide and Tsukamoto, Mayu},
 title = {An assortment of properties of silting subcategories of extriangulated categories},
 fjournal = {Bulletin des Sciences Math{\'e}matiques},
 journal = {Bull. Sci. Math.},
 issn = {0007-4497},
 volume = {203},
 pages = {},
 note = {Id/No 103647},
 year = {2025},
}

@article{Che26,
  title={Exact dg categories {I}: Foundations},
  author={Chen, Xiaofa},
  journal={Adv. Math.},
  volume={489},
  pages={110809},
  year={2026},
}

@article{PZ24,
 author = {Pan, Jixing and Zhu, Bin},
 title = {Silting interval reduction and 0-{Auslander} extriangulated categories},
 fjournal = {Journal of Pure and Applied Algebra},
 journal = {J. Pure Appl. Algebra},
 volume = {229},
 number = {6},
 pages = {32},
 note = {Id/No 107978},
 year = {2025},
}

@misc{Che24,
  title={Iyama--{S}olberg correspondence for exact dg categories},
  author={Chen, Xiaofa},
  note={arXiv:2401.02064},
  year={2024}
}

@article{ZZ21,
  title={Grothendieck groups in extriangulated categories},
  author={Zhu, Bin and Zhuang, Xiao},
  journal={J. Algebra},
  volume={574},
  pages={206--232},
  year={2021},
}

@article{OS23,
 author = {Ogawa, Yasuaki and Shah, Amit},
 title = {A resolution theorem for extriangulated categories with applications to the index},
 fjournal = {Journal of Algebra},
 journal = {J. Algebra},
 volume = {658},
 pages = {450--485},
 year = {2024},
}

@article{GNP21,
  title={Positive and negative extensions in extriangulated categories},
  author={Gorsky, Mikhail and Nakaoka, Hiroyuki and Palu, Yann},
  year={2026},
  volume = {313},
  journal ={Math. Z.},
  pages ={89},
}

@article{BTHSS23,
  title={The category of extensions and a characterisation of $n$-exangulated functors},
  author={Bennett-Tennenhaus, Raphael and Haugland, Johanne and Sand{\o}y, Mads Hustad and Shah, Amit},
  journal={Math. Z.},
  volume={305},
  number={3},
  pages={44},
  year={2023},
}

@article{HLN17,
  title={$n$-exangulated categories ({I}): Definitions and fundamental properties},
  author={Herschend, Martin and Liu, Yu and Nakaoka, Hiroyuki},
  journal={J. Algebra},
  volume={570},
  pages={531--586},
  year={2021},
}

@article{GHKK18,
  title={Canonical bases for cluster algebras},
  author={Gross, Mark and Hacking, Paul and Keel, Sean and Kontsevich, Maxim},
  journal={J. Am. Math. Soc.},
  volume={31},
  number={2},
  pages={497--608},
  year={2018}
}

@article{BY14,
  title={Ordered exchange graphs},
  author={Br{\"u}stle, Thomas and Yang, Dong},
  journal={Advances in Representation Theory of Algebras},
  pages={135--193},
  year={2014},
  publisher={European Mathematical Society Publishing House}
}

@article{NOS22,
  title={Localization of extriangulated categories},
  author={Nakaoka, Hiroyuki and Ogawa, Yasuaki and Sakai, Arashi},
  journal={J. Algebra},
  volume={611},
  pages={341--398},
  year={2022},
}

@article{AT22,
  title={Hereditary cotorsion pairs and silting subcategories in extriangulated categories},
  author={Adachi, Takahide and Tsukamoto, Mayu},
  journal={J. Algebra},
  volume={594},
  pages={109--137},
  year={2022},
}

@article{Igu14,
 author = {Igusa, Kiyoshi},
 title = {A category of noncrossing partitions},
 fjournal = {Applied Categorical Structures},
 journal = {Appl. Categ. Struct.},
 volume = {33},
 number = {6},
 pages = {41},
 note = {Id/No 40},
 year = {2025},
}

@article{BH21,
      title={$\tau$-perpendicular wide subcategories}, 
      author={A. B. Buan and E. J. Hanson},
      year={2023},
    journal={Nagoya Math. J.},
volume = {252},
    pages ={959--984},
}

@ARTICLE{B-TS20,
	AUTHOR = {Bennett-Tennenhaus, Raphael and Shah, Amit},
	TITLE = {Transport of structure in higher homological algebra},
	JOURNAL = {J. Algebra},
	FJOURNAL = {Journal of Algebra},
	VOLUME = {574},
	YEAR = {2021},
	PAGES = {514--549},
}

@article {NP19,
	AUTHOR = {Nakaoka, Hiroyuki and Palu, Yann},
	TITLE = {Extriangulated categories, {H}ovey twin cotorsion pairs and
		model structures},
	JOURNAL = {Cah. Topol. G\'{e}om. Diff\'{e}r. Cat\'{e}g.},
	FJOURNAL = {Cahiers de Topologie et G\'{e}om\'{e}trie Diff\'{e}rentielle Cat\'{e}goriques},
	VOLUME = {60},
	YEAR = {2019},
	NUMBER = {2},
	PAGES = {117--193},
}

@article{FG06,
 author = {Fock, Vladimir and Goncharov, Alexander},
 title = {Moduli spaces of local systems and higher {Teichm{\"u}ller} theory},
 fjournal = {Publications Math{\'e}matiques},
 journal = {Publ. Math., Inst. Hautes {\'E}tud. Sci.},
 volume = {103},
 pages = {1--211},
 year = {2006},
}

@article{GHK15,
  title={Birational geometry of cluster algebras},
  author={Gross, Mark and Hacking, Paul and Keel, Sean},
  journal={Algebr. Geom.},
  volume={2},
  number={2},
  pages={137--175},
  year={2015},
  publisher={Foundation Compositio Mathematica}
}

@book{Ful93,
  title={Introduction to toric varieties},
  author={Fulton, William},
  number={131},
  year={1993},
  publisher={Princeton university press}
}

@article{Ore44,
 author = {Ore, {\O}ystein},
 title = {Galois connexions},
 fjournal = {Transactions of the American Mathematical Society},
 journal = {Trans. Am. Math. Soc.},
 volume = {55},
 pages = {493--514},
 year = {1944},
}

@book{Zie95,
  author     = {Ziegler, G\"unter M.},
  title      = {Lectures on polytopes},
  series     = {Graduate Texts in Mathematics},
  volume     = {152},
  publisher  = {Springer-Verlag, New York},
  year       = {1995},
  pages      = {x+370},
  isbn       = {0-387-94365-X},
}

@article {Kra15,
	AUTHOR = {Krause, Henning},
	TITLE = {Krull--{S}chmidt categories and projective covers},
	JOURNAL = {Expo. Math.},
	FJOURNAL = {Expositiones Mathematicae},
	VOLUME = {33},
	YEAR = {2015},
	NUMBER = {4},
	PAGES = {535--549},
}

@article{DK08,
  title={On the combinatorics of rigid objects in 2--{C}alabi--{Y}au categories},
  author={Dehy, Raika and Keller, Bernhard},
  journal={	Int. Math. Res. Not.},
  volume={2008},
  pages={rnn029},
  year={2008},
}

@article{Asa21,
  title={The wall-chamber structures of the real {G}rothendieck groups},
  author={S. Asai},
  journal={Adv. Math.},
volume = {381},
pages = {Paper No. 107615},
year = {2021},
}

@article {AI12,
	AUTHOR = {Aihara, Takuma and Iyama, Osamu},
	TITLE = {Silting mutation in triangulated categories},
	JOURNAL = {{J. Lond. Math. Soc., II. Ser.}},
	FJOURNAL = {Journal of the London Mathematical Society. Second Series},
	VOLUME = {85},
	YEAR = {2012},
	NUMBER = {3},
	PAGES = {633--668},
}

@article {IY08,
		AUTHOR = {Iyama, Osamu and Yoshino, Yuji},
		TITLE = {Mutation in triangulated categories and rigid
			{C}ohen-{M}acaulay modules},
		JOURNAL = {Invent. Math.},
		FJOURNAL = {Inventiones Mathematicae},
		VOLUME = {172},
		YEAR = {2008},
		NUMBER = {1},
		PAGES = {117--168},
	}

@article {BM18t,
		AUTHOR = {Buan, Aslak B. and Marsh, Bethany Rose},
		TITLE = {{$\tau$}-exceptional sequences},
		JOURNAL = {J. Algebra},
		FJOURNAL = {Journal of Algebra},
		VOLUME = {585},
		YEAR = {2021},
		PAGES = {36--68},
	}

@article {BM18w,
		AUTHOR = {Buan, Aslak B. and Marsh, Bethany Rose},
		TITLE = {A category of wide subcategories},
		JOURNAL = {Int. Math. Res. Not.},
		FJOURNAL = {International Mathematics Research Notices. IMRN},
		YEAR = {2021},
		NUMBER = {13},
		volume={2021},
		PAGES = {10278--10338},
	}

@article {PPPP23,
    AUTHOR = {Padrol, Arnau and Palu, Yann and Pilaud, Vincent and
              Plamondon, Pierre-Guy},
     TITLE = {Associahedra for finite-type cluster algebras and minimal
              relations between {$g$}-vectors},
   JOURNAL = {Proc. Lond. Math. Soc. (3)},
  FJOURNAL = {Proceedings of the London Mathematical Society. Third Series},
    VOLUME = {127},
      YEAR = {2023},
    NUMBER = {3},
     PAGES = {513--588},
}

@article{FGPPP23,
  title={Extriangulated ideal quotients, with applications to cluster theory and gentle algebras},
  author={Fang, Xin and Gorsky, Mikhail and Palu, Yann and Plamondon, Pierre-Guy and Pressland, Matthew},
  journal={J. Pure Appl. Algebra},
  pages={108370},
  year={2026},
}

@article{Kai23,
author = {Kaipel, M.},
title = {The category of a partitioned fan},
journal = {J. Lond. Math. Soc., II. Ser.},
volume = {111},
number = {2},
pages = {Paper No. e70071},
year = {2025}
}

@article {Jas15,
    author = {G. Jasso},
    title = {Reduction of $\tau$-tilting modules and torsion pairs},
    journal = {Int. Math. Res. Not.},
    volume = {2015},
    number = {16},
    pages = {7190-7237},
    year = {2015}
}

@article {IY18,
	AUTHOR = {Iyama, Osamu and Yang, Dong},
	TITLE = {Silting reduction and {C}alabi--{Y}au reduction of
		triangulated categories},
	JOURNAL = {Trans. Amer. Math. Soc.},
	FJOURNAL = {Transactions of the American Mathematical Society},
	VOLUME = {370},
	YEAR = {2018},
	NUMBER = {11},
	PAGES = {7861--7898},
}

@incollection {KV88,
	AUTHOR = {Keller, B. and Vossieck, D.},
	TITLE = {Aisles in derived categories},
	NOTE = {Deuxi\`eme Contact Franco-Belge en Alg\`ebre (Faulx-les-Tombes,
		1987)},
	JOURNAL = {Bull. Soc. Math. Belg. S\'{e}r. A},
	FJOURNAL = {Bulletin de la Soci\'{e}t\'{e} Math\'{e}matique de Belgique. S\'{e}rie A},
	VOLUME = {40},
	YEAR = {1988},
	NUMBER = {2},
	PAGES = {239--253},
	ISSN = {0037-9476},
	MRCLASS = {16A46 (16A64)},
	MRNUMBER = {976638},
}

@article {IJY14,
	AUTHOR = {Iyama, Osamu and Jørgensen, Peter and Yang, Dong},
	TITLE = {Intermediate co-{$t$}-structures, two-term silting objects,
		{$\tau$}-tilting modules, and torsion classes},
	JOURNAL = {Algebra Number Theory},
	FJOURNAL = {Algebra \& Number Theory},
	VOLUME = {8},
	YEAR = {2014},
	NUMBER = {10},
	PAGES = {2413--2431},
}

@article {Ver96,
	AUTHOR = {Verdier, Jean-Louis},
	TITLE = {Des cat\'{e}gories d\'{e}riv\'{e}es des cat\'{e}gories ab\'{e}liennes},
	NOTE = {With a preface by Luc Illusie,
		Edited and with a note by Georges Maltsiniotis},
	JOURNAL = {Ast\'{e}risque},
	FJOURNAL = {Ast\'{e}risque},
	NUMBER = {239},
	YEAR = {1996},
	PAGES = {xii+253 pp. (1997)},
	ISSN = {0303-1179},
	MRCLASS = {18E30 (18-03 18E35)},
	MRNUMBER = {1453167},
	MRREVIEWER = {Amnon Neeman},
}

@misc{ITW16,
      title={Picture groups of finite type and cohomology in type $A_n$}, 
      author={Kiyoshi Igusa and Gordana Todorov and Jerzy Weyman},
      year={2016},
      eprint={1609.02636},
      archivePrefix={arXiv},
      primaryClass={math.RT},
      note={arXiv:1609.02636}, 
}

@misc{IT17,
      title={Signed exceptional sequences and the cluster morphism category}, 
      author={Kiyoshi Igusa and Gordana Todorov},
      year={2017},
      eprint={1706.02041},
      archivePrefix={arXiv},
      primaryClass={math.RT},
      note={arXiv:1706.02041}, 
}

@article{BMR07,
  title={Cluster-tilted algebras},
  author={Buan, Aslak B. and Marsh, Bethany R. and Reiten, Idun},
  journal={Trans. Am. Math. Soc.},
  volume={359},
  number={1},
  pages={323--332},
  year={2007}
}

@article {BMRRT06,
	AUTHOR = {Buan, Aslak B. and Marsh, Bethany and Reineke, Markus and
		Reiten, Idun and Todorov, Gordana},
	TITLE = {Tilting theory and cluster combinatorics},
	JOURNAL = {Adv. Math.},
	FJOURNAL = {Advances in Mathematics},
	VOLUME = {204},
	YEAR = {2006},
	NUMBER = {2},
	PAGES = {572--618},
}

@article {HI21,
	year = 2021,
	volume = {49},
	number = {10},
	pages = {4376--4415},
	author = {E. J. Hanson and K. Igusa},
	title = {$\tau$-cluster morphism categories and picture groups},
	journal = {Commun. Algebra}
}

@article {HI21p,
	year = 2021,
	publisher = {Elsevier {BV}},
	volume = {225},
	number = {6},
	pages = {Paper No. 106598},
	author = {E. J. Hanson and K. Igusa},
	title = {Pairwise compatibility for 2-simple minded collections},
	journal = {J. Pure Appl. Algebra}
}

@article{DIRRT17,
  author = {L. Demonet and O. Iyama and N. Reading and I. Reiten and H. Thomas},
  title = {Lattice theory of torsion classes: Beyond $\tau$-tilting theory},
    year = {2023},
    pages = {542-612},
    volume = {10},
    journal = {Trans. Am. Math. Soc., Ser. B},
  }

@incollection{Qui73,
	title={Higher algebraic {K}-theory: I},
	author={Quillen, Daniel},
	booktitle={Higher K-theories},
	pages={85--147},
	year={1973},
	publisher={Springer}
}

@phdthesis{Che23thesis,
	title={On exact dg categories},
	author={Chen, Xiaofa},
	school  = "Université Paris Cité",
	year={2023}
}

@misc{Che23,
  title={0-{A}uslander correspondence},
  author={Chen, Xiaofa},
  note={arXiv:2306.15958},
  year={2023}
}

@article{IY20,
  title={Quotients of triangulated categories and equivalences of {B}uchweitz, {O}rlov, and {A}miot--{G}uo--{K}eller},
  author={Iyama, Osamu and Yang, Dong},
  journal={Am. J. Math.},
  volume={142},
  number={5},
  pages={1641--1659},
  year={2020},
}

@misc{GNP23,
	title={Hereditary extriangulated categories: Silting objects, mutation, negative extensions},
	author={Gorsky, Mikhail and Nakaoka, Hiroyuki and Palu, Yann},
	eprint={2303.07134},
      archivePrefix={arXiv},
      primaryClass={math.RT},
      note={arXiv:2303.07134}, 
      year={2023},
}

@misc{Bor21,
	title={Two-term silting and $\tau$-cluster morphism categories},
	author={B{\o}rve, Erlend D.},
	note={arXiv:2110.03472},
	year={2021}
}

@article {STTW23,
	AUTHOR = {Schroll, Sibylle and Tattar, Aran and Treffinger, Hipolito and Williams, Nicholas J},
	TITLE = {A geometric perspective on the $\tau$-cluster morphism category},
	JOURNAL = {Math. Z.},
	FJOURNAL = {Math. Z.},
	VOLUME = {312},
	YEAR = {2026},
	PAGES = {77},
}

@article{FZ02i,
	title={Cluster algebras {I}: foundations},
	author={Fomin, Sergey and Zelevinsky, Andrei},
	journal={J. Am. Math. Soc.},
	volume={15},
	number={2},
	pages={497--529},
	year={2002}
}

@article{FZ02ii,
 author = {Fomin, Sergey and Zelevinsky, Andrei},
 title = {Cluster algebras. {II}: {Finite} type classification},
 fjournal = {Inventiones Mathematicae},
 journal = {Invent. Math.},
 volume = {154},
 number = {1},
 pages = {63--121},
 year = {2003},
}

@article{FZ07,
 Author = {Fomin, S. and Zelevinsky, A.},
 Title = {Cluster algebras. {IV}: {Coefficients}.},
 FJournal = {Compositio Mathematica},
 Journal = {Compos. Math.},
 Volume = {143},
 Number = {1},
 Pages = {112--164},
 Year = {2007},
}

@article{BFZ05,
	title={Cluster algebras {III}: {U}pper bounds and double {B}ruhat cells},
	author={Berenstein, Arkady and Fomin, Sergey and Zelevinsky, Andrei},
	journal={Duke Math. J.},
	volume={126},
	number={1},
	pages={1--52},
	year={2005},
}

@article{Gar23,
 author = {Garcia, Monica},
 title = {On thick subcategories of the category of projective presentations},
 fjournal = {Mathematische Zeitschrift},
 journal = {Math. Z.},
 volume = {312},
 number = {4},
 pages = {39},
 note = {Id/No 110},
 year = {2026},
}

@article{Tak13,
  title={Thick subcategories over {G}orenstein local rings that are locally hypersurfaces on the punctured spectra},
  author={Takahashi, Ryo},
  journal={J. Math. Soc. Japan},
  volume={65},
  number={2},
  pages={357--374},
  year={2013},
}

@book{CLS24,
  title={Toric varieties},
  author={Cox, David A and Little, John B and Schenck, Henry K},
  volume={124},
  year={2024},
  publisher={American Mathematical Society}
}

@incollection{BB80,
  author     = {Brenner, Sheila and Butler, M. C. R.},
  title      = {Generalizations of the {B}ernstein-{G}el'
                fand-{P}onomarev reflection functors},
  booktitle  = {Representation theory, {II} ({P}roc. {S}econd {I}nternat.
                {C}onf., {C}arleton {U}niv., {O}ttawa, {O}nt., 1979)},
  series     = {Lecture Notes in Math.},
  volume     = {832},
  pages      = {103--169},
  publisher  = {Springer, Berlin},
  year       = {1980},
  isbn       = {3-540-10264-7},
  mrclass    = {16A64 (16A46)},
  mrnumber   = {607151},
  mrreviewer = {Idun\ Reiten}
}

@article{HR82,
  author     = {Happel, Dieter and Ringel, Claus Michael},
  title      = {Tilted algebras},
  journal    = {Trans. Amer. Math. Soc.},
  fjournal   = {Transactions of the American Mathematical Society},
  volume     = {274},
  year       = {1982},
  number     = {2},
  pages      = {399--443},
}

@misc{BorveKaipeltautilt,
      title={Bricks and $\tau$-tilting theory under base field extensions}, 
      author={Erlend D. Børve and Eric J. Hanson and Maximilian Kaipel},
      year={2025},
      eprint={2508.01040},
      archivePrefix={arXiv},
      primaryClass={math.RT},
      note={arXiv:2508.01040}, 
}

@article {DerksenFei2015,
    AUTHOR = {Derksen, Harm and Fei, Jiarui},
     TITLE = {General presentations of algebras},
   JOURNAL = {Adv. Math.},
    VOLUME = {278},
      YEAR = {2015},
     PAGES = {210--237},
}

@article{Dickson66,
 author = {S. E. Dickson},
 journal = {Trans. Am. Math. Soc.},
 number = {1},
 pages = {223--235},
 title = {A Torsion Theory for Abelian Categories},
 volume = {121},
 year = {1966}
}

@Article{CFZ2002,
 Author = {Chapoton, Fr{\'e}d{\'e}ric and Fomin, Sergey and Zelevinsky, Andrei},
 Title = {Polytopal realizations of generalized associahedra},
 Journal = {Can. Math. Bull.},
 Volume = {45},
 Number = {4},
 Pages = {537--566},
 Year = {2002},
}

@article{FZ2003,
 author = {Fomin, Sergey and Zelevinsky, Andrei},
 title = {$Y$-systems and generalized associahedra},
 fjournal = {Annals of Mathematics. Second Series},
 journal = {Ann. Math. (2)},
 volume = {158},
 number = {3},
 pages = {977--1018},
 year = {2003},
}

@article{DWZ2010,
 author = {Derksen, Harm and Weyman, Jerzy and Zelevinsky, Andrei},
 title = {Quivers with potentials and their representations. {II}: {Applications} to cluster algebras.},
 fjournal = {Journal of the American Mathematical Society},
 journal = {J. Am. Math. Soc.},
 issn = {0894-0347},
 volume = {23},
 number = {3},
 pages = {749--790},
 year = {2010},
}

@article{HPS2018,
 author = {Hohlweg, Christophe and Pilaud, Vincent and Stella, Salvatore},
 title = {Polytopal realizations of finite type {{\(\mathbf{g}\)}}-vector fans},
 fjournal = {Advances in Mathematics},
 journal = {Adv. Math.},
 volume = {328},
 pages = {713--749},
 year = {2018},
}

@article{HLT2011,
 author = {Hohlweg, Christophe and Lange, Carsten E. M. C. and Thomas, Hugh},
 title = {Permutahedra and generalized associahedra.},
 fjournal = {Advances in Mathematics},
 journal = {Adv. Math.},
 volume = {226},
 number = {1},
 pages = {608--640},
 year = {2011},
}

@article{ReadingSpeyer2009,
 author = {Reading, Nathan and Speyer, David E.},
 title = {Cambrian fans.},
 fjournal = {Journal of the European Mathematical Society (JEMS)},
 journal = {J. Eur. Math. Soc. (JEMS)},
 volume = {11},
 number = {2},
 pages = {407--447},
 year = {2009},
}

@article{CK2006,
 author = {Caldero, Philippe and Keller, Bernhard},
 title = {From triangulated categories to cluster algebras. {II}.},
 fjournal = {Annales Scientifiques de l'{\'E}cole Normale Sup{\'e}rieure. Quatri{\`e}me S{\'e}rie},
 journal = {Ann. Sci. {\'E}c. Norm. Sup{\'e}r. (4)},
 volume = {39},
 number = {6},
 pages = {983--1009},
 year = {2006},
}

@article{BMR2008,
 author = {Buan, Aslak Bakke and Marsh, Bethany R. and Reiten, Idun},
 title = {Cluster mutation via quiver representations.},
 fjournal = {Commentarii Mathematici Helvetici},
 journal = {Comment. Math. Helv.},
 volume = {83},
 number = {1},
 pages = {143--177},
 year = {2008},
}

@misc{Nonis2026,
      title={Presilting sequences for 0-Auslander extriangulated categories}, 
      author={Iacopo Nonis},
      year={2026},
      eprint={2605.20957},
      archivePrefix={arXiv},
      primaryClass={math.RT},
      note={arXiv:2605.20957}, 
}

@misc{ZLZ2026,
      title={Homotopic morphisms and diagram theorems in extriangulated categories}, 
      author={Chencheng Zhang and Xue-Song Lu and Pu Zhang},
      year={2026},
      eprint={2604.22186},
      archivePrefix={arXiv},
      primaryClass={math.CT},
      note={arXiv:2604.22186}, 
}

@misc{Treffinger26,
      title={Scattering diagrams for Artin algebras}, 
      author={Hipolito Treffinger},
      year={2026},
      eprint={2608.04233},
      archivePrefix={arXiv},
      primaryClass={math.RT},
      note={arXiv:2608.04233}, 
}

\end{document}